\documentclass[a4paper,11pt,reqno]{amsart}
\usepackage{marginnote,amsmath,amsfonts,amscd,amssymb,graphicx,mathrsfs,eufrak,dsfont}

\usepackage[dvips,all,arc,curve,color,frame]{xy}
\usepackage[usenames]{color}

\usepackage[colorlinks]{hyperref}
\usepackage{tikz,mathrsfs}
\usetikzlibrary{arrows,decorations.pathmorphing,decorations.pathreplacing,positioning,shapes.geometric,shapes.misc,decorations.markings,decorations.fractals,calc,patterns}
\usepackage{tikz}
\usetikzlibrary{positioning}

\newcommand{\R}{\mathbb{R}}

\newcommand{\Z}{\mathbb{Z}}

\newtheorem{thm}{Theorem}[section]
\newtheorem{prop}[thm]{Proposition}
\newtheorem{lemma}[thm]{Lemma}

\newtheorem{cor}[thm]{Corollary}

\newtheorem{prob}[thm]{Problem}
\newtheorem{que}[thm]{Question}
\newtheorem{obs}[thm]{Observation}

\theoremstyle{definition}

\newtheorem{remark}[thm]{Remark}
\newtheorem{conj}[thm]{Conjecture}

\begin{document}
\title[\textsc{Solvability of Mazes by Blind Robots}]{\textsc{Solvability of Mazes by Blind Robots}}
\author{Stefan David and Marius Tiba}
\address{University of Cambridge}
\email{sd637@cam.ac.uk}
\email{mt576@cam.ac.uk}
\begin{abstract}

In this paper we introduce and investigate a new type of automata which turns out to be rich in deep and complex phenomena. For our model, a maze is a countable strongly connected digraph called the board together with a proper colouring of its edges (the edges leaving a vertex have distinct colours) and two special vertices: the origin and the destination. A pointer or robot starts at the origin of a maze and moves naturally between its vertices, according to a finite or infinite sequence of specific instructions from the set of all colours called an algorithm; if the robot is at a vertex for which there is no out-edge of the colour indicated by the instruction, it remains at that vertex and proceeds to execute the next instruction in the sequence. The central object of study is the existence of algorithms that simultaneously solve, that is guide the robot to visit the destination in, certain large sets of mazes.

One of the most natural and interesting sets of mazes arises from the square lattice $\Z^2$ viewed as a graph with arbitrarily many edges removed (each edge corresponds to a pair of opposite directed edges), together with the suggestive colouring that assigns to each directed edge the corresponding cardinal direction. In this set-up, a research question of Leader and Spink from 2011, which proved to be very profound, asks whether there exists an algorithm which solves this set of mazes.

In this paper we make progress towards this question. We consider the subset of all such mazes which have arbitrarily many horizontal edges removed but only finitely many vertical edges removed in consecutive columns, and construct an algorithm which solves this subset of mazes.  
\end{abstract}

\maketitle
\parindent 20pt
\parskip 0pt

\section{Introduction}\label{intro}
\indent
Though studied for decades, recent important breakthroughs in automata theory have turned it into an important field of study in discrete mathematics and theoretical computer science. For a comprehensive introduction in the theory and other related subjects, see the book of Hopcroft, Motwani and Ullman \cite{Hopcroft}.

One of the long standing famous conjectures in automata theory is the road colouring problem introduced in 1970 by Adler, Goodwyn and Weiss in \cite{Adler1}, \cite{Adler2}. The conjecture states that a strongly connected digraph in which all vertices have the same out-degree, which is aperiodic (i.e. the gcd of the lengths of all of its oriented cycles is one) has a synchronising colouring. A synchronising colouring of a strongly connected digraph $G$ of uniform out-degree $k$ is a labelling of the edges of $G$ with colours $1, \hdots, k$ such that all the vertices have out-edges of all colours and for every vertex $v$ of $G$ there exists a word $W_v$ in the alphabet of colours such that every path in $G$ corresponding to $W_v$ terminates at $v$. We note that the existence of a synchronising colouring makes it possible to reset the automaton back to its original state after the detection of an error. In fact, it is because of this important property that the road coloring problem has received so much attention over the past few decades. There have been many positive partial results published over the years, such as \cite{New1}, \cite{New2}, \cite{New3}. In 2009, Trahtman made one of the most notable advances in the field by proving this conjecture in \cite{Trahtman}. 

Another famous related problem in the field is {\v C}ern{\' y}'s conjecture which appeared in \cite{CCerny} in 1964 and states that the length of the shortest synchronising word for any $n$-state deterministic finite automaton is bounded above by $(n-1)^2$ (for more details see \cite{Pin15}, \cite{Trahtman16}). For some partial results concerning {\v C}ern{\' y}'s conjecture see \cite{Grech}, \cite{Cerny}.

In this paper we introduce and study a new model of automata which turns out to be abundant in profound and intricate phenomena. This model does not seem to have occurred in literature, and is motivated by the following coffee time problem of Leader, popularised by Balister.

\begin{prob}\label{pb1}
Consider the classical $8 \times 8$ chessboard as the board of a maze, where every small square is a room, such that between any two adjacent rooms there is either a wall that prevents the transit between them, or there is no wall and transit is possible. Additionally, the boundary of the board is formed only by walls.

Say that a robot starts in one of the $64$ squares and it receives a sequence of instructions from the set of cardinal directions: \emph{north}, \emph{south}, \emph{east}, \emph{west}. Each time the robot receives such an instruction, it executes it by moving to the corresponding adjacent room, provided there is no wall to prevent it from moving as instructed; if there is such a wall, the robot simply does not move and it continues with the following instruction. The robot does not give any feedback whether it moves or not when executing an instruction. 

Naturally, the board of the maze can be regarded as a subgraph of the square lattice $8 \times 8$ where there is an edge between two vertices if and only if there is no wall between the corresponding squares.  Without knowing the subgraph and the starting vertex of the robot, can one write a finite sequence of instructions such that at the end the robot is guaranteed to have visited all accessible vertices?
\end{prob}
To see the existence of such an algorithm, simply enumerate all the possible mazes and solve them one by one, keeping track of the updated position of the robot when passing to a new maze. A related problem which can be solved in the same way is the following:
\begin{prob}\label{pb2}
Consider a subgraph of some finite dimensional hypercube $Q_1, Q_2, \hdots$ as the board of a maze. Say that a robot starts in one of the vertices and it receives a sequence of instructions from the set of coordinate directions $\pm e_1, \pm e_2, \hdots$. Each time the robot receives such an instruction, it executes it by moving to the corresponding adjacent vertex, provided there is an edge between these two vertices; if there is no such edge, the robot simply does not move and it continues with the following instruction. Without knowing the subgraph and the starting vertex of the robot, can one write an infinite sequence of instructions such that at the end the robot is guaranteed to have visited all accessible vertices?
\end{prob}

Problem~\ref{pb1} lead Spink and Leader to ask the following research question, which was later passed to us by Balister.

\begin{que}\label{qloha}
What happens if in Problem~\ref{pb1} we replace the (finite) $8 \times 8$ square lattice with the infinite square lattice $\Z^2$?
\end{que}

To our knowledge, Question~\ref{qloha} turns out to be extremely difficult to answer. In this paper we make progress towards answering this question, by establishing the following main result. 

\begin{thm}\label{mthm}
There exists an infinite sequence of instructions for a robot to visit all accessible vertices in any maze for which the board is the graph $\mathbb{Z}^2$  with arbitrarily many horizontal edges removed but only finitely many vertical edges removed in consecutive columns.
\end{thm}

We note that Theorem~\ref{mthm} follows immediately from two separate beautiful results, Theorem~\ref{ch2th1} and Theorem~\ref{ch2th2} binded together by the more technical Proposition~\ref{ch2prop1}.

The structure of the paper is as follows. In Section~\S\ref{prelims} we start by developing a general set-up that encompasses a class of similar problems which we call ``solvability of mazes by blind robots''. We then return to the Leader-Spink problem and state all our main results in Section~\S\ref{results}. In Section~\S\ref{BC} we present a toy model that represents the foundation on which the general model is constructed. As part of the toy model, we prove Theorem~\ref{ch2th1}; this allows us to introduce and investigate some generic algorithms that are used as building blocks in the proof of Theorem~\ref{ch2th2}. In Section~\S\ref{fnv} we present a series of technical definitions that are used to construct a countable cover of the set of mazes in Theorem~\ref{ch2th2} with subsets of mazes that we can treat individually. In Section~\S\ref{final} we give the constructive proof of Theorem~\ref{ch2th2}. We continue with the proof of the technical result Proposition~\ref{ch2prop1} in Section~\S\ref{sectiunea1}. Finally, in Section~\S\ref{conclusions} we present several further directions of research and some of our conjectures.

\section{Preliminaries}\label{prelims}
\indent
In this section we introduce the general framework for our model giving formal definitions and some relevant examples. 

A \emph{maze} is a quadruple $(M, c, o, d)$, where $M$ is a countable strongly connected digraph called the board, and $c:E(M) \longrightarrow \mathbb{N}$ is a proper colouring of the edges of $M$, i.e. one in which the out-edges from any vertex have distinct colours. Further, $o$ and $d$ are two special vertices of $M$ called the \emph{origin} and the \emph{destination}, respectively. 

An \emph{instruction} $I \in \mathbb{N}$ is an element from the set of colours $\mathbb{N}$. An \emph{algorithm} 
\[
A = (I_i)_{i=1}^n \text{ or } A = (I_i)_{i=1}^\infty
\]
is a finite or infinite sequence of instructions. A \emph{subalgorithm} $A'$ of an infinite algorithm $A$ as above is any truncation of $A$ of the form
\[
A' = (I_i)_{i=k}^j \text{ or } A' = (I_i)_{i=k}^\infty,
\]
for some $k \leq j$. Similarly, a \emph{subalgorithm} $A'$ of a finite algorithm $A= (I_i)_{i=1}^n$ is any truncation of $A$ of the form $A' = (I_i)_{i=k}^j$ for some $k \leq j \leq n$. In order to describe dynamically our process of visiting the graph we look at the following model.

Given an algorithm $A = (I_i)_{i=1}^\infty$ and a maze $(M, c, o, d)$, a \emph{robot} - which is just a travelling tracking object or a pointer - starts at the origin $o$ and moves in between the vertices of $M$, as it follows the instructions $I_1, I_2, \hdots$ one by one in order: for $n \in \mathbb{N}$ the robot executes the $n$-th instruction $I_n \in \mathbb{N}$ by moving from its current vertex $v$ to the next vertex $w$ if and only if there exists an oriented edge $e$ of colour $I_n$ from $v$ to $w$; if there is no such oriented edge $e$, the robot remains at $v$. In short, we say that the robot \emph{follows the algorithm $A$ in the maze $(M, c, o, d)$}. We say that an algorithm $A$ \emph{solves} the maze $(M, c, o, d)$ if the robot visits the destination $d$ at some time by following $A$ in $(M, c, o, d)$. Similarly, we say that an algorithm $A$ \emph{solves} a set $\mathcal{M}$ of mazes if it solves every maze in $\mathcal{M}$.

We remark that each connected graph can be regarded as a strongly connected digraph by doubling edges. Throughout the paper all the boards of the mazes arise in this way and hence from now on we define the board of a maze to be a graph. We can omit the condition that the graph is connected if we require that the origin and destination are in the same connected component of the graph. 

In this set-up, the fundamental question that arises is the existence of algorithms that simultaneously solve certain natural sets of mazes. As we shall see from the arguments which appear in this paper, and also from our conclusions and open questions in Section~\S\ref{conclusions}, this set up is rich in very deep insights related to the phenomenon of state automata.

For example, we note that there is no algorithm that solves the set of all mazes. Indeed, let us assume for a contradiction that $A = (I_i)_{i=1}^\infty$ does the job. We construct $M$ to be the path with vertices $v_0, v_1, \hdots$ and its only edges $v_i \rightarrow v_{i-1}$ and $v_{i-1} \rightarrow v_i$ for all $i \in \mathbb{N}$. We set $o=v_1$, $d=v_0$ and colour the edge $v_i \rightarrow v_{i+1}$ with colour $I_i$ and the rest of the edges in any way that does not violate the proper colouring condition. A robot that starts in this maze and follows $A$ will visit in order $v_1, v_2, v_3, \hdots$ as it follows $I_1, I_2, \hdots$, never reaching $v_0=d$. As $M$ was constructed to be strongly connected, we have reached a contradiction.

As another example, we note that for any countable set of mazes, there exist algorithms that solve it. In particular, this solves Problem~\ref{pb1} and more importantly, it shows that there exist algorithms that solve the set of all finite mazes. Indeed, let $(M_1, c_1, o_1, d_1), (M_2, c_2, o_2, d_2) \hdots$ be an enumeration of a countable set of mazes $\mathcal{M}$. Considering the strongly connectedness property, given any maze $(M, c, o, d)$ one can write by inspection a finite algorithm that solves the maze. Then, let $A_1$ be any finite algorithm that solves $(M_1, c_1, o_1, d_1)$; let $o_2'$ be the position of the robot after it follows the algorithm $A_1$ in $(M_2, c_2, o_2, d_2)$; let $A_2$ be any finite algorithm that solves $(M_2, c_2, o_2', d_2)$ with origin $o_2'$; let $o_3'$ be the position of the robot after it follows the algorithm $A_1A_2$ in $(M_3, c_3, o_3, d_3)$, etc. Continue in this way to create algorithms $A_1, A_2, \hdots$. We claim that the algorithm $A=A_1 A_2 \hdots$ obtained by concatenating $A_1, A_2, \hdots$ solves the set of mazes $\mathcal{M}$. Indeed, consider the maze $(M_i, c_i, o_i, d_i) \in \mathcal{M}$ for some $i \geq 2$. After the robot follows the initial subalgorithm $A_1 A_2 \hdots A_{i-1}$ of $A$ it gets to the vertex $o_i'$ of $M_i$ and then after it follows $A_i$ it gets to the destination point $d_i$. Trivially, for the maze $(M_1, c_1, o_1, d_1)$, the robot gets to the destination point $d_1$ after it follows the initial subalgorithm $A_1$ of $A$. This shows that $A$ solves $\mathcal{M}$.

We can see from the two examples above that the most interesting cases of our model occur ''in between``, when we consider natural uncountable sets of mazes for which we seek to construct algorithms to solve them. We present below two uncountable sets of mazes, for which it is not hard to find such algorithms.

Firstly, let $Q=Q_1\cup Q_2 \cup \hdots$ be the nested union of all finite dimensional hypercubes i.e. the graph with vertices all possible infinite $\{ 0, 1 \}$ sequences with trailing zeros and edges between those pairs of vertices which differ in only one coordinate. Let $\mathcal{Q}$ be the set of mazes for which the board is a connected subgraph of $Q$ and the colouring assigns to each directed edge the corresponding coordinate direction $\pm e_1, \pm e_2, \hdots$.

Secondly, let $Z=\Z^1\cup \Z^2\cup \hdots$ be the nested union of all finite dimensional integer lattices i.e. the graph with vertices all possible infinite integer sequences with trailing zeros and edges between those pairs of vertices which differ in only one coordinate and the  difference is one. An increasing path inside $Z$, is an infinite path that passes through the origin and always goes in a positive coordinate direction $ +e_1, +e_2, \hdots$. Let $\mathcal{P}$ be the set of mazes for which the board is an increasing path and the colouring assigns to each directed edge the corresponding coordinate direction.

The main object of study in this paper, though much more challenging, resembles Problem~\ref{pb1}. One of the most fundamental and fascinating sets of mazes is the set $\mathcal{M}$ for which the board is the square lattice $\mathbb{Z}^2$ considered as a graph with arbitrarily many edges deleted, the colouring assigns to each directed edge the corresponding cardinal direction from the set $\{N, S, E, W\} = \{S^{-1}, N^{-1}, W^{-1}, E^{-1}\}$, and the origin and destination are in the same connected component. From now on we define a maze to be a triple $(M, \textbf{o}, \textbf{d}) \in \mathcal{M}$, where $M$ is the board, $\textbf{o} = (x_o, y_o)$ is the origin, and $\textbf{d} = (x_d, y_d)$ is the destination.

We call all the vertices in the connected component of the origin \emph{accessible points}. For a vertex $\textbf{x} = (x, y)$ we refer to its coordinates $x$ and $y$ as the \emph{longitude} and \emph{latitude}, respectively. We also label the columns and rows of $\mathbb{Z}^2$ by $c_i = \{(i, y) \mid y \in \mathbb{Z} \}$ and $r_i= \{(x, i) \mid x \in \mathbb{Z} \}$ for $i\in \mathbb{Z}$, respectively. We say that $c_i$ and $c_{i+1}$ are \emph{joined} at latitude $j$ if the vertices $(i, j)$ and $(i+1, j)$ are adjacent.

We often create new algorithms by concatenations of other algorithms, and it is very convenient to use multiplication to denote concatenation. For example 
\[
SNSSNS = SNS^2NS=(SNS)^2
\].

For a finite algorithm $A$, we write $|A|$ for the number of instructions in $A$; similarly we write $|A|_I$ for the number of instructions $I$ in $A$, for all $I \in \{N, S, E, W\}$.

\begin{figure}
\centering
\resizebox{0.5\textwidth}{!}{
\begin{tikzpicture}
    [
        dot/.style={circle,draw=black, fill,inner sep=2pt},
        cross/.style={cross out, draw=black, minimum size=2*(#1-\pgflinewidth), inner sep=4pt, outer sep=4pt},
    ]

\foreach \x in {-4,...,4}{
 \foreach \y in {-4,...,4}{
    \node[dot] at (\x,\y){ };
}}

\draw[line width=0.7mm, red] (-4, -4) -- (-4, 4);
\draw[line width=0.7mm, red] (-3, -4) -- (-3, 4);
\draw[line width=0.7mm, red] (2, -4) -- (2, 4);
\draw[line width=0.7mm, red] (3, -4) -- (3, 4);
\draw[line width=0.7mm, red] (4, -4) -- (4, 4);
\draw[line width=0.7mm, red] (-2, 4) -- (-2, 2);
\draw[line width=0.7mm, red] (-2, 0) -- (-2, -1);
\draw[line width=0.7mm, red] (-2, -2) -- (-2, -3);
\draw[line width=0.7mm, red] (-1, 3) -- (-1, 1);
\draw[line width=0.7mm, red] (-1, -1) -- (-1, -2);
\draw[line width=0.7mm, red] (0, -2) -- (0, -4);
\draw[line width=0.7mm, red] (0, 0) -- (0, 3);
\draw[line width=0.7mm, red] (1, 4) -- (1, 3);
\draw[line width=0.7mm, red] (1, 2) -- (1, -4);

\draw[line width=0.7mm, red] (3, 4) -- (4, 4);
\draw[line width=0.7mm, red] (2, 4) -- (-2, 4);
\draw[line width=0.7mm, red] (-2, 3) -- (0, 3);
\draw[line width=0.7mm, red] (1, 3) -- (3, 3);
\draw[line width=0.7mm, red] (-4, 3) -- (-3, 3);
\draw[line width=0.7mm, red] (-4, 2) -- (-3, 2);
\draw[line width=0.7mm, red] (-1, 2) -- (1, 2);
\draw[line width=0.7mm, red] (2, 2) -- (3, 2);
\draw[line width=0.7mm, red] (-1, 1) -- (2, 1);
\draw[line width=0.7mm, red] (-3, 1) -- (-2, 1);
\draw[line width=0.7mm, red] (3, 1) -- (4, 1);
\draw[line width=0.7mm, red] (-2, 0) -- (-1, 0);
\draw[line width=0.7mm, red] (1, 0) -- (3, 0);
\draw[line width=0.7mm, red] (-2, -1) -- (1, -1);
\draw[line width=0.7mm, red] (2, -1) -- (3, -1);
\draw[line width=0.7mm, red] (-4, -1) -- (-3, -1);
\draw[line width=0.7mm, red] (-4, -2) -- (-2, -2);
\draw[line width=0.7mm, red] (2, -2) -- (4, -2);
\draw[line width=0.7mm, red] (-1, -3) -- (2, -3);
\draw[line width=0.7mm, red] (3, -3) -- (4, -3);
\draw[line width=0.7mm, red] (-2, -4) -- (-1, -4);
\draw[line width=0.7mm, red] (0, -4) -- (2, -4);
\draw[line width=0.7mm, red] (3, -4) -- (4, -4);
\draw[line width=0.7mm, red] (-4, -4) -- (-3, -4);

\draw (3,-2) node[cross] {};

\foreach \x in {-4,...,4}
    \draw (\x,.1) -- node[below,yshift=-1mm] {\x} (\x,-.1);
\foreach \y in {-4,...,4}
    \draw (.1,\y) -- node[below,xshift=-3mm, yshift=3mm] {\y} (-.1,\y);
\draw[->,line width=0.15mm] (0,-4.5) -- (0,4.5);
\draw[->,line width=0.15mm] (-4.5,0) -- (4.5,0);
\end{tikzpicture}
}
\caption{A local representation of a general maze $M$, where edges are marked by red lines. We also mark the destination point $(3, -2)$ with a cross and note that in every maze there is a path from the origin to the destination point. When the robot follows the algorithm $SNWWN$ in $M$ it gets to the point $(-1, 2)$ it follows the path $(0, 0),$ $(0, 0),$ $(0, 1),$ $(-1, 1),$ $(-1, 1),$ $(-1, 2)$; the robot does not move when it executes the first and fourth instructions, as there is no edge between $(0, 0)$ and $(0, -1)$ and between $(-1, 1)$ and $(-2, 1)$.} \label{0}
\end{figure}
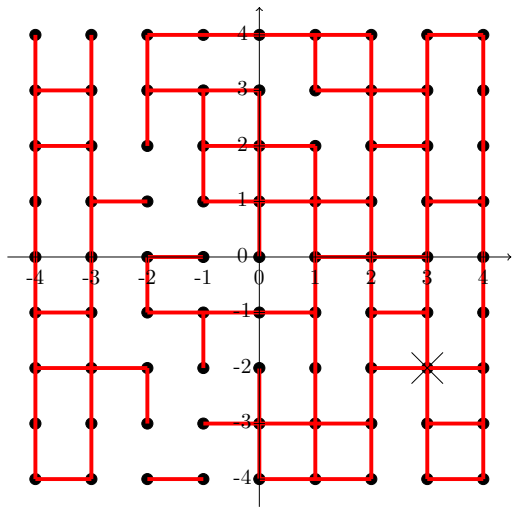

\section{Our Results}\label{results}
Our main result, Theorem~\ref{mthm} follows directly from Theorem~\ref{ch2th1}, Theorem~\ref{ch2th2} and Proposition~\ref{ch2prop1}, all of which are interesting results on their own.
\begin{thm}\label{ch2th1}
Let $\mathcal{C} \subseteq \mathcal{M}$ be the set of all mazes for which the board has arbitrarily many horizontal edges removed but no vertical edges removed. There exists an algorithm that solves $\mathcal{C}$.
\end{thm}
\begin{thm}\label{ch2th2}
Let $\mathcal{F} \subseteq \mathcal{M}$ be the set of all mazes for which the board has arbitrarily many horizontal edges removed and nonzero finitely many vertical edges removed in consecutive columns. There exists an algorithm that solves $\mathcal{F}$.
\end{thm}
We should note that, as one might expect, the proof of Theorem~\ref{ch2th2} turns out to be much more difficult than the proof of Theorem~\ref{ch2th1} and that both proofs are constructive. In Section~\S\ref{BC}, in which we give the proof of Theorem~\ref{ch2th1}, we also introduce some generic algorithms which constitute the main building blocks of the algorithm which solves $\mathcal{F}$. In Lemma~\ref{l2}, which is a technical key result of the paper, we present their properties that we use multiple times in the proof of Theorem~\ref{ch2th2}.

Finally, we use the following result as a binder between Theorem~\ref{ch2th1} and Theorem~\ref{ch2th2} in order to obtain Theorem~\ref{mthm}. This result ascertains the intuitive fact that under certain technical conditions, if two sets of mazes are solvable so is there union and moreover if a set of mazes is solvable, then it is also solvable infinitely often.

\begin{prop}\label{ch2prop1}
Let $e(\mathbb{Z}^2)$ be the set of edges of $\mathbb{Z}^2$. We can regard any board of a maze as an indicator function $f: e(\mathbb{Z}^2) \longrightarrow \{ 0,1 \}$. Hence, the set of boards of mazes equipped with the product topology is a compact metrizable space. Let $\mathcal{A}_1, \mathcal{A}_2 \subseteq \mathcal{M}$ be two sets of mazes with the following properties:
\begin{enumerate}
\item  for all $i\in \{1,2\}$, all origins $o\in \mathbb{Z}^2$, all destination $d\in \mathbb{Z}^2$ and all paths $P$ between $o$ and $d$, the sets of boards $B_i = \{ M \mid (M, o, d) \in \mathcal{A}_i, P\le M \}$ that contain the path $P$ are compact;
\item for all  $i\in \{1,2\}$ if $(M,o,d) \in \mathcal{A}_i$, then $(M,o',d') \in \mathcal{A}_i$ for all $o',d'$ in the same connected component as $o,d$; 
\item there exist algorithms $A_1$ and $A_2$ that solve the sets $\mathcal{A}_1$ and $\mathcal{A}_2$, respectively.
\end{enumerate}

Then there exists an algorithm $A$ that solves the set $\mathcal{A}=\mathcal{A}_1\cup \mathcal{A}_2$ and that furthermore guides the robot to visit the destination of any maze in the set infinitely often. Moreover, if we cut or add an initial segment to $A$, the algorithm obtained in this way has the same property.
\end{prop}

\section{The Toy Model}\label{BC}
The aim of this section is to prove Theorem~\ref{ch2th1} and to introduce the general strategy and some generic algorithms that are used in the proof of Theorem~\ref{ch2th2} as well.

For a subset of mazes $\mathcal{C} \subseteq \mathcal{M}$, in order to construct an algorithm $A$ that solves $\mathcal{C}$ we adopt the following natural strategy: we find a countable cover $\mathcal{C}=\cup_{i=1}^{\infty}\mathcal{C}_i$ such that for each $i \in \mathbb{N}$ and each finite algorithm $X$ we are able to find a finite algorithm $A_X^i$ such that the concatenated algorithm $XA_X^i$ solves $\mathcal{C}_i$. Then we are able to find an algorithm $A$ that solves $\mathcal{C}$. Indeed, we construct recursively the finite algorithms $(B_n)_{n\ge0}$ with $B_0= \emptyset$ and $B_{n}=A_{B_0B_1 \hdots B_{n-1}}^{n}$, then we take $A=B_1B_2 \hdots$.

In the toy model, let $\mathcal{C}$ be the set of mazes with no vertical edges removed. Without loss of generality, we assume that for any maze in $\mathcal{C}$ the origin is the point $(0, 0)$. The main property of this set of mazes is that at each step of the algorithm we know the robot's latitude.

\begin{proof}[Proof of Theorem~\ref{ch2th1}]
We begin the proof by defining two classes of algorithms. The aim of the first one is to move the robot eastwards in a certain organised pattern and we call it $move \_ east$; it is defined as follows for all $a, e \geq 1$: \\
$ME(a, e) := (((((E)^e NES)^e SEN)^e N^2ES^2)^e \hdots S^a E N^a)^e.$

We view $ME(a, e)$ as being composed from the multiple concatenation of $2a+1$ different building blocks which we call \emph{locomotory moves}: $E,$ $NES,$ $SEN,$ $N^2ES^2,$ $\hdots$ $N^aES^a$, $S^a E N^a.$ We constructed the class of algorithms $move \_ east$ in such a way so that the following holds:

Let $a, e$ be two natural numbers. Assume that the robot starts at the point $\textbf{x} = (x, y)$ in any maze $M \in \mathcal{C}$ with no vertical edges removed. Take the maximal $k \leq e$ such that in $M$ the columns $c_i$ and $c_{i+1}$ are joined at some latitude in $\{-a+y, \hdots, a+y \}$ for all $x \le i \le x+k-1$. Then, as the robot follows the algorithm $ME(a, e)$, it oscillates about the row $r_y$ at latitudes between $y-a$ and $y+a$. After the algorithm is followed, the robot gets to a point $\textbf{x'} = (x', y)$ with $x'\ge x+k$, in particular $x' = x+k$ if $k<e$. Moreover if we well order $\mathbb{Z}$ by $y<1+y<-1+y<2+y<-2+y< \hdots$, then for all $x \le i \le x+k-1$ the robot passes from the column $c_i$ to the column $c_{i+1}$ through the edge at the lowermost latitude with respect to this order.

This holds as a particular case of Lemma~\ref{l2}, which is a technical result used extensively, proved later in this section. One can also see how this statement follows from the construction of $ME(a, e)$, more specifically from the order in which the locomotory moves appear in the algorithm. The counterpart of $move \_ east$ is called $move \_ west$, and we have: \\
$MW(a, e) := (((((W)^e NWS)^e SWN)^e N^2WS^2)^e \hdots S^a W N^a)^e.$

The second class of algorithms that we define is called $oscillating \_ move \_ east$, which is a slight alteration of $move \_ east$ formed by inserting the \emph{oscillatory} algorithm $(N^bS^{2b}N^b)^e$ in between some locomotory moves; it is defined as follows for all $a, e \geq 1$ and $b \in \mathbb{Z}$: \\
$OME(a, e, b) := ((((((N^b S^{2b} N^b)^e E)^e NES)^e SEN)^e N^2ES^2)^e \hdots S^a E N^a)^e.$

We note that in every maze with no vertical edge removed, after the robot follows the oscillatory algorithm $(N^bS^{2b}N^b)^e$, it gets back to the starting point. Therefore, for any parameters $a, e, b$, as the robot follows $OME(a, e, b)$ in any maze $M \in \mathcal{C}$, it has the same dynamics as it follows $ME(a, e)$ in $M$ and in addition the robot visits some consecutive columns, beginning with the one which contains its starting point $\textbf{x} = (x, y)$, at all latitudes between $y-b$ and $y+b$. Finally, we use the oscillatory algorithm $(N^bS^{2b}N^b)^e$ instead of $N^bS^{2b}N^b$ which works just as well for this purpose, only because we want $OME(a, e, b)$ to be a particular case of a much more general algorithm, $SME(a, e, K)$ that is defined later in this section.

The counterpart of $oscillating \_ move \_ east$ is called $oscillating \_ move \_ west$, and we have:\\
$OMW(a, e, b) = ((((((N^b S^{2b} N^b)^e W)^e NWS)^e SWN)^e N^2WS^2)^e \hdots S^a W N^a)^e.$

We are now ready to prove the theorem using the general strategy described at the beginning of the section. In order to produce the desired countable cover, define $C_{n, \textbf{x}}$ to be the set of all mazes with no vertical edges removed, with the destination point $\textbf{x} = (x, y)$ and such that any two consecutive columns at longitude between $0$ and $x$ are joined at some latitude between $-n$ and $n$. Then $C=\cup_{n,\textbf{x}}C_{n,\textbf{x}}$ is a countable cover.

We let $X$ be any finite algorithm and we fix the values $n, \textbf{x}$. We now consider just the set of mazes $C_{n,\textbf{x}}$ and we aim to construct an algorithm $A$ such that $XA$ solves $C_{n,\textbf{x}}$, which by the discussion of our strategy at the beginning of the section is enough to conclude. 

Say that the robot starts in any maze $M \in C_{n,\textbf{x}}$ (as always, it starts in the origin) and it gets to the point $(a, 0)$ after it follows some finite algorithm $Y$. Define $a:=\max \{|Y|_N, |Y|_S\}$; $e:=|Y|_W$. The following observation is crucial: for each pair $\{i,i+1\}\subset \{0, \hdots, a\}$, the columns $c_i$ and $c_{i+1}$ are joined at some latitude in $\{-|Y|_S, \hdots, |Y|_N \} \subseteq \{ -|Y|, \hdots, |Y| \}$. Therefore, after the robot follows the algorithm $Y$ $ME(a,e)$ in $M$ it gets to some point $(a', 0)$ with $a' \ge 0$.

Now we build $A$ as a concatenation of three algorithms $A:=A_1A_2A_3$.

We construct $A_1:= S^{|X|_N-|X|_S}$; then after the robot follows the algorithm $XA_1$ in any maze $M \in C_{n,\textbf{x}}$ it gets to $r_0$.

Define $a:= \max \{|XA_1|_S, |XA_1|_N, n\}$; $e:=|XA_1|_W+|x|$. Define $A_2:= ME(a,e)$. Then after the robot follows the algorithm $X A_1 A_2$ in any maze $M \in C_{n,\textbf{x}}$ it gets to some point $(x^+,0)$ with $x^+\ge x$.

Define $a:= \max \{|XA_1A_2|_S, |XA_1A_2|_N, n\}$; $w:=|XA_1A_2|_E+ |x|+1$; $b:=|y|$. Define $A_3:= OMW(a,w,b)$. Then after the robot follows the algorithm $XA_1A_2A_3$ in any maze $M \in C_{n,\textbf{x}}$, it gets to some point $(x^-,0)$ with $x^-\le x$ and it visits every intermediate column $c_i$ with $x^-\le i\le x^+$ including $c_x$ at every latitude in $\{-b, \hdots, b \}$ including $y$. 

Therefore, after the robot follows $XA =XA_1A_2A_3$ in any maze $M \in C_{n,\textbf{x}}$, it visits the destination point $\textbf{x}$. Hence there exists an algorithm $A$ such that $XA$ that solves $C_{n,xy}$. This finishes the proof.
\end{proof}

We note that the missing vertical edges in the general model usually make the latitude of the robot unknown but it turns out that we can actually make use of the missing edges to regain the latitude of the robot. However, the unknown longitude and the missing edges require the robot to use a very subtle path to get to the destination point. As a result of these difficulties in the proof of Theorem~\ref{ch2th2} we need to make a much finer covering than in the proof of Theorem~\ref{ch2th1}.

In the remainder of this section we introduce an algorithm which is a generalisation of $ME(a, e)$ and $OME(a, e, b)$ called $special \_ move \_ east$ which is the main building block of the algorithms used in the general model. We then group all its properties in Lemma~\ref{l2}, which makes it one of the main results of the paper. For $a, e \geq 1$ and a finite algorithm $K$ we define:\\
$SME (a,e,K):=(((((K^eE)^eNES)^eSEN)^eN^2ES^2)^e...S^aEN^a)^e.$ We view $SME(a, e, K)$ as being composed from the multiple concatenation of $2a+2$ different building blocks: the $2a+1$ locomotory moves $E,$ $NES,$ $SEN,$ $\hdots$ $S^aEN^a$ and the special algorithm $K.$

Its counterpart, $special \_ move \_ west$ is defined as:\\
$SMW (a,e,K):=(((((K^eW)^eNWS)^eSWN)^eN^2WS^2)^e...S^aWN^a)^e.$

Recall that $\mathcal{C} \subset \mathcal{M}$ is the set of mazes with no vertical edges removed. The following result encompasses the main properties of $SME (a,e,K)$ that are used countless times in the proof of Theorem~\ref{ch2th2}.

\begin{lemma}\label{l2}
Let $a, e \geq 1$ and $K$ be a finite algorithm such that for any maze $M \in \mathcal{C}$, if the robot follows $K$ in $M$ starting from the origin, it returns on the $x$-axis and it has a non-negative longitude. Let $0\le l \le e-2$ have the following properties: 

(1) for any $0\le x \le l$ the columns $c_x$ and $c_{x+1}$ are joined at some latitude between $-a$ and $a$; 

(2) for any $\textbf{v}=(x_v,0)$ with $0 \le x_v \le l$, if the robot starts from $\textbf{v}$ and follows $K$ in any maze in $\mathcal{C}$ it gets at some point $\textbf{w} = (x_w,0)$ with $x_v \le x_w \le l$ without visiting any vertex on the column $c_{l+1}$, i.e. without visiting any point of longitude at least $l+1$.

Then, after the robot follows $SME(a,e,K)$ in any maze in $\mathcal{C}$, it gets to some point $\textbf{v}=(x_v,0)$ on the $x$-axis with $x_v \geq l+1$. Moreover, it does not pass from the column $c_l$ to the column $c_{l+1}$ for the first time while executing $K$; it passes from the column $c_l$ to the column $c_{l+1}$ for the first time while executing a locomotory move $N^mES^m$, where $m \in \Z$ is the lowermost latitude with respect to the standard well order on $\Z: 0<1<-1<2<-2< \hdots$ such that the columns $c_l$ and $c_{l+1}$ are joined at latitude $m$; finally, immediately after this locomotory move $N^mES^m$ is executed, the robot follows $K$.
\end{lemma}
\begin{proof}
Let $M \in \mathcal{C}$ be any maze with no vertical edge removed. We prove the result for $M$, so for brevity, we make the convention that every time we say that the robot follows an algorithm, it follows that algorithm in $M$.

By the hypothesis on $K$, if the robot is on the $x$-axis and follows $K$ (or $N^b E S^b$, $b \in \{ -a, \hdots a \}$), it returns to the $x$-axis and its longitude does not strictly decrease. We fix $x$ between $0$ and $l$, so that the columns $c_x$ and $c_{x+1}$ are joined at some latitude $b$ between $-a$ and $a$. Hence, if the robot starts from the point $(x, 0)$ and follows $N^bES^b$ it gets to the point $(x+1, 0)$. Therefore, if the robot is on the $x$-axis at some longitude between $0$ and $l$, then after each instance of the algorithm $(((((K^eE)^eNES)^eSEN)^eNNESS)^e...S^aEN^a)^1$, the longitude of the robot increases by at least one. This proves the first statement of the conclusion, that if the robot follows $SME(a,e,K)$, it gets to some point $\textbf{v}=(x_v,0)$ on the $x$-axis with $x_v \geq l+1$.

The second statement of the conclusion is that if the robot follows $SME(a,e,K)$, it does not pass from the column $c_l$ to the column $c_{l+1}$ for the first time while executing $K$. This follows directly from the hypothesis: indeed, for any $\textbf{v}=(x_v,0)$ with $0 \le x_v \le l$, if the robot starts from $\textbf{v}$ and follows $K$, it gets at some point $\textbf{w} = (x_w,0)$ with $x_v \le x_w \le l$ without visiting any vertex on the column $c_{l+1}$, i.e. without visiting any point of longitude at least $l+1$.

From the first two statements of the conclusion proved above it follows that the robot passes for the first time from the column $c_l$ to the column $c_{l+1}$ while executing some instance of the move of the form $N^bES^b$, $-a \leq b \leq a$. Assume for a contradiction that $b \neq 0$ is not the lowermost latitude with respect to the well order on $\Z$ at which the columns $c_l$ and $c_{l+1}$ are joined. Let $b' \in \Z$ be the predecessor of $b$ in the well order on $\Z$. Say $Y$ is the first segment of the algorithm $SME(a,e,K)$ strictly before this specific instance of this specific locomotory move, $N^bES^b$.  

We define $A=(((((K^eE)^eNES)^eSEN)^eNNESS)^e...N^{b'}ES^{b'})^1$ and note that $A' = A^e =  (((((K^eE)^eNES)^eSEN)^eNNESS)^e...N^{b'}ES^{b'})^e$ is a last segment of $Y$. Let $B$ be the first segment of $Y$ strictly before $A'$, i.e. $Y = BA'$. For some $0 \leq x \leq l$ we denote by $(x, 0)$ the vertex where the robot gets if it starts from the origin and follows $B$. If the robot starts from $(x, 0)$ and follows $A'$, it gets to the point $(l, 0)$. Also notice that if the robot starts from $(l, 0)$ and follows $A$, it gets to some point $(l', 0)$ with $l' \geq l+1$. Say the robot starts from the point $(x, 0)$ and follows the algorithm $A^{e+1}$. If the robot starts from the $x$- axis and follows $A$ it advances eastwards at least $0$ columns. When the robot starts from the $x$- axis and follows the $e+1$-th instance of $A$, it returns to the $x$-axis and advances eastwards at least one column. This means that if the robot starts from the $x$-axis and follows the $w$-th instance of $A$ it returns to the $x$-axis and advances eastwards at least one column for each $1 \leq w \leq e+1$. 

Therefore, if the robot starts from $(x, 0)$ and follows $A' = A^e$, it gets to the point $(l, 0)$ and advances eastwards at least $e$ columns. This is a contradiction as $l+1 \leq e$. This proves the third statement of the conclusion, that if the robot follows $SME(a,e,K)$, it passes from the column $c_l$ to the column $c_{l+1}$ for the first time while executing a locomotory move $N^mES^m$, where $m \in \Z$ is the lowermost latitude with respect to the standard well order on $\Z: 0<1<-1<2<-2< \hdots$ such that the columns $c_l$ and $c_{l+1}$ are joined at latitude $m$.

By the third statement of the conclusion, we know that the robot passes for the first time from the column $c_l$ to the column $c_{l+1}$ while executing the move $N^mES^m$. Assume $K$ does not follow immediately that after this move is executed. Say $Y$ is the first segment of the algorithm $SME(a,e,K)$ before and including this specific instance of this specific locomotory move, $N^mES^m$.

We define $A=(((((K^eE)^eNES)^eSEN)^eNNESS)^e...N^{m}ES^{m})^1$ and note that $A' = A^e =  (((((K^eE)^eNES)^eSEN)^eNNESS)^e...N^{m}ES^{m})^e$ is the last segment of $Y$.  Let $B$ be the first segment of $Y$ strictly before $A'$, i.e. $Y = BA'$. For some $0 \leq x \leq l$ we denote by $(x, 0)$ the vertex where the robot gets if it starts from the origin and follows $B$. If the robot starts from $(x, 0)$ and follows $A'$, it gets to the point $(l+1, 0)$. Say the robot starts from the point $(x, 0)$ and it follows the algorithm $A^e$. If the robot starts from the $x$- axis and follows $A$, it advances eastwards at least $0$ columns. When the robot starts from the $x$- axis and follows the $e$-th instance of $A$, it returns to the $x$-axis and advances eastwards at least one column. This means that if the robot starts from the $x$-axis and follows the $w$-th instance of $A$ it returns to the $x$-axis and advances eastwards at least one column for each $1 \leq w \leq e$. 

Therefore, if the robot starts from $(x, 0)$ and it follows $A' = A^e$, it gets to the point $(l, 0)$ and advances eastwards at least $e$ columns. This is a contradiction as $l+2 \leq e$, proving the last statement of the conclusion, that after the robot passes for the first time from $c_l$ to $c_{l+1}$ following the locomotory move $N^mES^m$, the robot follows $K$. This finishes the proof.
\end{proof}
We end this section with the following immediate corollary of Lemma~\ref{l2}.
\begin{cor}\label{c1}
Under the assumptions of Lemma~\ref{l2}, let us choose another order on $\Z$, say the \emph{$n$-special order on $\Z$}: $0<n<1<-1< \hdots$ instead of the well order on $\Z$ we considered in Lemma~\ref{l2}.  Then if we construct 
\[
SME^{(n)} (a, e, K):= (((((((K)^e N^{n} E S^{n})^e E)^e NES)^e SEN)^e NNESS)^e \hdots S^aEN^a)^e,
\]
the results in Lemma~\ref{l2} still hold, with the amendment that after the robot follows $SME^{(n)} (a, e, K)$ in any maze in $\mathcal{C}$, it passes for the first time from the column $c_l$ to the column $c_{l+1}$ while executing $N^mES^m$, where $m$ is the lowermost latitude with respect to the $n$-special order on $\Z$.
\end{cor}

\section{The Cover}\label{fnv}
In the general model, let $\mathcal{F} \subset \mathcal{M}$ be the set of mazes with nonzero finitely many vertical edges removed in consecutive columns. Without loss of generality, we assume that for any maze in $\mathcal{F}$ the origin is the point $(0, 0)$. In this section, we introduce a series of technical definitions that are used to classify the mazes in $\mathcal{F}$ in order to prove Theorem~\ref{ch2th2}.
\begin{figure}[h!]
\centering
\resizebox{0.5\textwidth}{!}{
\begin{tikzpicture}
    [
        dot/.style={circle,draw=black, fill,inner sep=2pt},
        cross/.style={cross out, draw=black, minimum size=2*(#1-\pgflinewidth), inner sep=4pt, outer sep=4pt},
    ]

\foreach \x in {-4,...,4}{
 \foreach \y in {-4,...,4}{
    \node[dot] at (\x,\y){ };
}}

\draw[line width=0.7mm, red] (-4, -4) -- (-4, 4);
\draw[line width=0.7mm, red] (-3, -4) -- (-3, 4);
\draw[line width=0.7mm, red] (2, -4) -- (2, 4);
\draw[line width=0.7mm, red] (3, -4) -- (3, 4);
\draw[line width=0.7mm, red] (4, -4) -- (4, 4);
\draw[line width=0.7mm, red] (-2, 4) -- (-2, 2);
\draw[line width=0.7mm, red] (-2, 0) -- (-2, -1);
\draw[line width=0.7mm, red] (-2, -2) -- (-2, -3);
\draw[line width=0.7mm, red] (-1, 3) -- (-1, 1);
\draw[line width=0.7mm, red] (-1, -1) -- (-1, -2);
\draw[line width=0.7mm, red] (0, -2) -- (0, -4);
\draw[line width=0.7mm, red] (0, 0) -- (0, 3);
\draw[line width=0.7mm, red] (1, 4) -- (1, 3);
\draw[line width=0.7mm, red] (1, 2) -- (1, -4);

\draw[line width=0.7mm, red] (3, 4) -- (4, 4);
\draw[line width=0.7mm, red] (2, 4) -- (-2, 4);
\draw[line width=0.7mm, red] (-2, 3) -- (0, 3);
\draw[line width=0.7mm, red] (1, 3) -- (3, 3);
\draw[line width=0.7mm, red] (-4, 3) -- (-3, 3);
\draw[line width=0.7mm, red] (-4, 2) -- (-3, 2);
\draw[line width=0.7mm, red] (-1, 2) -- (1, 2);
\draw[line width=0.7mm, red] (2, 2) -- (3, 2);
\draw[line width=0.7mm, red] (-1, 1) -- (2, 1);
\draw[line width=0.7mm, red] (-3, 1) -- (-2, 1);
\draw[line width=0.7mm, red] (3, 1) -- (4, 1);
\draw[line width=0.7mm, red] (-2, 0) -- (-1, 0);
\draw[line width=0.7mm, red] (1, 0) -- (3, 0);
\draw[line width=0.7mm, red] (-2, -1) -- (1, -1);
\draw[line width=0.7mm, red] (2, -1) -- (3, -1);
\draw[line width=0.7mm, red] (-4, -1) -- (-3, -1);
\draw[line width=0.7mm, red] (-4, -2) -- (-2, -2);
\draw[line width=0.7mm, red] (2, -2) -- (4, -2);
\draw[line width=0.7mm, red] (-1, -3) -- (2, -3);
\draw[line width=0.7mm, red] (3, -3) -- (4, -3);
\draw[line width=0.7mm, red] (-2, -4) -- (-1, -4);
\draw[line width=0.7mm, red] (0, -4) -- (2, -4);
\draw[line width=0.7mm, red] (-4, -4) -- (-3, -4);

\draw (3,-2) node[cross] {};

\foreach \x in {-4,...,4}
    \draw (\x,.1) -- node[below,yshift=-1mm] {\x} (\x,-.1);
\foreach \y in {-4,...,4}
    \draw (.1,\y) -- node[below,xshift=-3mm, yshift=3mm] {\y} (-.1,\y);
\draw[->,line width=0.15mm] (0,-4.5) -- (0,4.5);
\draw[->,line width=0.15mm] (-4.5,0) -- (4.5,0);
\end{tikzpicture}
}
\caption{A local representation of a general maze $M \in \mathcal{F}$ that we use in order to illustrate our definitions. The destination point $(3, -2)$ is marked with an 'X'. We assume that there are no vertical edges removed from $M$ other than the ones shown in the figure. For simplicity, we further assume that $M$ is connected, though this may not be true for all mazes.} \label{1}
\end{figure}
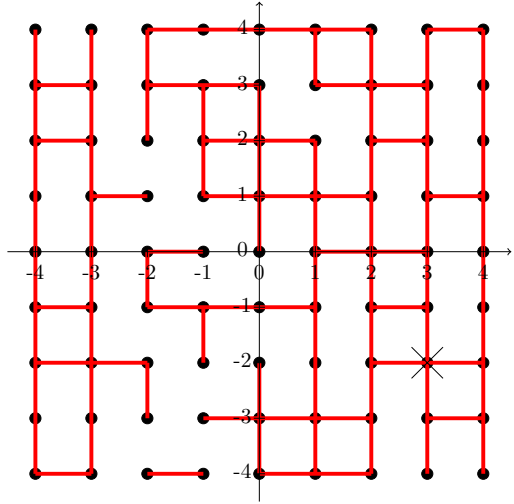

We recall that in order to construct an algorithm $A$ that solves the set of mazes $\mathcal{F} \subset \mathcal{M}$ we adopt the following strategy: we find a countable cover $\mathcal{F}=\cup_{i=1}^{\infty}F_i$ with $(F_i)_{i\ge1}\subseteq \mathcal{F}$ such that for each $i \in \mathbb{N}$ and each finite algorithm $X$ we are able to find a finite algorithm $A_X^i$ such that the concatenated algorithm $XA_X^i$ solves $F_i$.

%make a picture with the bar process
The aim of this section is to introduce the definitions that we need to use in order to construct the cover $(F_i)_{i \geq 1}$.

For any maze $M \in \mathcal{F}$ we denote by \emph{HE}, \emph{HNE}, \emph{VE}, \emph{VNE} a horizontal edge, horizontal non edge, vertical edge and vertical non edge, respectively. For $M$ as in Figure~\ref{1}, between $(2, 2)$ to $(3, 2)$ there is a HE, between $(-1, -2)$ and $(0, -2)$ there is a HNE, between $(0, 0)$ and $(0, 1)$ there is a VE and between $(1, 2)$ and $(1, 3)$ there is a VNE.

From any maze $M \in \mathcal{F}$ we construct the maze $\overline{M} \in \mathcal{F}$ by adding all the possible VEs such that the connected component of the origin is unchanged in the process. The new maze $\overline{M}$ has the nice property that the robot can get from the origin to one vertex of every VNE in $\overline{M}$. We note that an algorithm solves $M$ if and only if it solves $\overline{M}$. Therefore, in order to prove Theorem~\ref{ch2th2} it is enough to construct an algorithm $A$ which solves $\overline{\mathcal{F}} = \{ \overline {M} \mid M \in \mathcal{F} \} \subseteq \mathcal{F}$.

The rest of the section will only address mazes in $\overline{\mathcal{F}}$, so for any maze $\overline{M} \in \overline{\mathcal{F}}$ we introduce the following definitions. 

Let $S$ be the smallest vertical strip that contains all the VNEs, the origin and the destination point together with all the HEs incident to it on its left and right sides. As there is only a finite number of VNEs, $S$ contains a finite number of (consecutive) columns. For $M$ as in Figure~\ref{1}, $S$ is the subgraph formed from the columns $c_{-2}, \hdots, c_3$ together with all the HEs between $c_{-3}$ and $c_{-2}$ and all the HEs between $c_3$ and $c_4$; in particular the vertex $(-3, -2)$ and the edge between $(3, 1)$ and $(4, 1)$ are in $S$, but the vertex $(-3, 2)$ is not.

Considering the maze with all its HEs deleted, we can label the connected components obtained in this way by \emph{upper infinite columns}, \emph{lower infinite columns}, \emph{infinite columns}, and \emph{finite columns} accordingly. For $M$ as in Figure~\ref{1}, there are $4$ upper infinite columns, e.g. the infinite path $(-2,2), (-2, 3), \hdots$; there are also $4$ lower infinite columns, e.g. the infinite path $(-2, -4), (-2, -5), \hdots$; the infinite columns are $c_{-3}, c_{-4}, \hdots$ and $c_2, c_3, \hdots$; examples of finite columns are $(-2, 1)$, the path $(-2, -1), (-2, 0)$ or the path $(0, 0), (0, 1), (0, 2), (0, 3)$.

Considering only the HEs in $S$, we call a \emph{pass} any of the following edges:
\begin{enumerate}
\item the HE of smallest latitude between two upper infinite columns, or between an upper infinite column and an infinite column, e.g. the edge between $(-2, 4)$ and $(-1, 4)$ or the edge between $(1, 3)$ and $(2, 3)$, respectively in Figure~\ref{1};
\item the HE of largest latitude between two lower infinite columns, or between an lower infinite column and an infinite column, e.g. the edge between $(0, -3)$ and $(1, -3)$ or the edge between $(1, 1)$ and $(2, 1)$, respectively in Figure~\ref{1};
\item the HE of smallest latitude between two infinite columns with respect to the well order on $\Z: 0<1<-1<2<-2< \hdots $, e.g. the edge between $(2, 0)$ and $(3, 0)$ in Figure~\ref{1}.
\end{enumerate}

Every maze has a finite number of VNEs, so every maze has a finite number of passes. We further note that between two consecutive columns in $S$ there might not be a pass, if there is no HE between them. Finally, as a few more revealing examples, we note that in Figure~\ref{1} the edge between $(-3, 1)$ and $(-2, 1)$ is not a pass, and neither is the one between $(-4, -1)$ and $(-3, -1)$ which is not in $S$; however, the edge between $(3, 1)$ and $(4, 1)$ is in $S$ and it is also a pass.

Furthermore, we define the following regions: the \emph{obstacle strip} is the smallest vertical strip that contains all VNEs, together with all the HEs incident to it on its left and right sides. For example, in Figure~\ref{1} the obstacle strip is formed from the columns $c_{-2}, \hdots c_1$ together with all the HEs incident with any vertex on $c_{-2}$ or $c_1$. The \emph{west strip} and \emph{east strip} are the vertical strips situated at the left and right of the obstacle strip, respectively. For example, in Figure~\ref{1} is formed from the columns $c_{-3}, c_{-4}, \hdots$ and the east strip is formed from the columns $c_2, c_3, \hdots$. We note that the obstacle strip and the east or west strip may intersect only in a certain set of vertices, i.e. the eastern or western endvertices of the edges that emerge on the right or left side of the obstacle strip, respectively; they have no edges in common. 

We define the \emph{primary rectangle} to be the subgraph contained in the smallest rectangle that contains the origin, the destination point, all the passes and all the VNEs. The primary rectangle is well defined, as there is a finite number of passes and VNEs. Let $p$ be the smallest positive integer such that the primary rectangle is strictly contained in the interior of the square centred at the origin with the set of vertices $\{ ( \pm p, \pm p )\}$ (see Figure~\ref{2}). We call $p$ the \emph{parameter of the primary rectangle}.
\begin{figure}[h!]
\centering
\resizebox{0.5\textwidth}{!}{
\begin{tikzpicture}
    [
        dot/.style={circle,draw=black, fill,inner sep=2pt},
        cross/.style={cross out, draw=black, minimum size=2*(#1-\pgflinewidth), inner sep=5pt, outer sep=5pt},
    ]

\foreach \x in {-4,...,4}{
 \foreach \y in {-4,...,4}{
    \node[dot] at (\x,\y){ };
}}

\foreach \x in {-2,...,4}{
\draw node[fill,circle,inner sep=5.5pt,minimum size=2pt] at (\x, 4) {};
}
\foreach \x in {-2,...,3}{
\draw node[fill,circle,inner sep=5.5pt,minimum size=2pt] at (\x, 3) {};
}
\foreach \x in {-3,...,3}{
\draw node[fill,circle,inner sep=5.5pt,minimum size=2pt] at (\x, 2) {};
}
\foreach \x in {-3,...,4}{
\draw node[fill,circle,inner sep=5.5pt,minimum size=2pt] at (\x, 1) {};
}
\foreach \x in {-2,...,3}{
\draw node[fill,circle,inner sep=5.5pt,minimum size=2pt] at (\x, 0) {};
}
\foreach \x in {-2,...,3}{
\draw node[fill,circle,inner sep=5.5pt,minimum size=2pt] at (\x, -1) {};
}
\foreach \x in {-3,...,4}{
\draw node[fill,circle,inner sep=5.5pt,minimum size=2pt] at (\x, -2) {};
}
\foreach \x in {-2,...,3}{
\draw node[fill,circle,inner sep=5.5pt,minimum size=2pt] at (\x, -3) {};
}

\draw[line width=0.7mm, red] (-4, -4) -- (-4, 4);
\draw[line width=0.7mm, red] (-3, -4) -- (-3, 4);
\draw[line width=0.7mm, red] (2, -4) -- (2, 4);
\draw[line width=0.7mm, red] (3, -4) -- (3, 4);
\draw[line width=0.7mm, red] (4, -4) -- (4, 4);
\draw[line width=0.7mm, red] (-2, 4) -- (-2, 2);
\draw[line width=0.7mm, red] (-2, 0) -- (-2, -1);
\draw[line width=0.7mm, red] (-2, -2) -- (-2, -4);
\draw[line width=0.7mm, red] (-1, 3) -- (-1, 1);
\draw[line width=0.7mm, red] (-1, -1) -- (-1, -2);
\draw[line width=0.7mm, red] (-1, -3) -- (-1, -4);
\draw[line width=0.7mm, red] (0, -2) -- (0, -4);
\draw[line width=0.7mm, red] (0, 0) -- (0, 3);
\draw[line width=0.7mm, red] (1, 4) -- (1, 3);
\draw[line width=0.7mm, red] (1, 2) -- (1, -4);

\draw[line width=0.7mm, red] (3, 4) -- (4, 4);
\draw[line width=0.7mm, green] (-2, 4) -- (1, 4);
\draw[line width=0.7mm, red] (1, 4) -- (2, 4);
\draw[line width=0.7mm, red] (-2, 3) -- (0, 3);
\draw[line width=0.7mm, green] (1, 3) -- (2, 3);
\draw[line width=0.7mm, red] (2, 3) -- (3, 3);
\draw[line width=0.7mm, red] (-4, 3) -- (-3, 3);
\draw[line width=0.7mm, red] (-4, 2) -- (-3, 2);
\draw[line width=0.7mm, red] (-1, 2) -- (1, 2);
\draw[line width=0.7mm, red] (2, 2) -- (3, 2);
\draw[line width=0.7mm, green] (-3, 2) -- (-2, 2);
\draw[line width=0.7mm, red] (-1, 1) -- (1, 1);
\draw[line width=0.7mm, green] (1, 1) -- (2, 1);
\draw[line width=0.7mm, red] (-3, 1) -- (-2, 1);
\draw[line width=0.7mm, green] (3, 1) -- (4, 1);
\draw[line width=0.7mm, red] (-2, 0) -- (-1, 0);
\draw[line width=0.7mm, red] (1, 0) -- (2, 0);
\draw[line width=0.7mm, green] (2, 0) -- (3, 0);
\draw[line width=0.7mm, red] (-2, -1) -- (1, -1);
\draw[line width=0.7mm, red] (2, -1) -- (3, -1);
\draw[line width=0.7mm, red] (-4, -1) -- (-3, -1);
\draw[line width=0.7mm, red] (-4, -2) -- (-3, -2);
\draw[line width=0.7mm, green] (-3, -2) -- (-2, -2);
\draw[line width=0.7mm, red] (2, -2) -- (4, -2);
\draw[line width=0.7mm, green] (-1, -3) -- (1, -3);
\draw[line width=0.7mm, red] (1, -3) -- (2, -3);
\draw[line width=0.7mm, red] (3, -3) -- (4, -3);
\draw[line width=0.7mm, green] (-2, -4) -- (-1, -4);
\draw[line width=0.7mm, red] (0, -4) -- (2, -4);
\draw[line width=0.7mm, red] (-4, -4) -- (-2, -4);

\draw (3,-2) node[cross] {};

\foreach \x in {-4,...,4}
    \draw (\x,.1) -- node[below,yshift=-2mm] {\x} (\x,-.1);
\foreach \y in {-4,...,4}
    \draw (.1,\y) -- node[below,xshift=-4mm, yshift=3mm] {\y} (-.1,\y);
\draw[->,line width=0.15mm] (0,-4.5) -- (0,4.5);
\draw[->,line width=0.15mm] (-4.5,0) -- (4.5,0);
\end{tikzpicture}
}
\caption{We assume that there are no vertical edges removed other than the ones shown in the figure. The destination point is $(3, -2)$. All the passes are marked with green edges. The primary rectangle has vertices $(-3, 4), (4, 4), (-4, 4), (-3, -4)$ and $p=5$. The special vertices are drawn larger.} \label{2}
\end{figure}
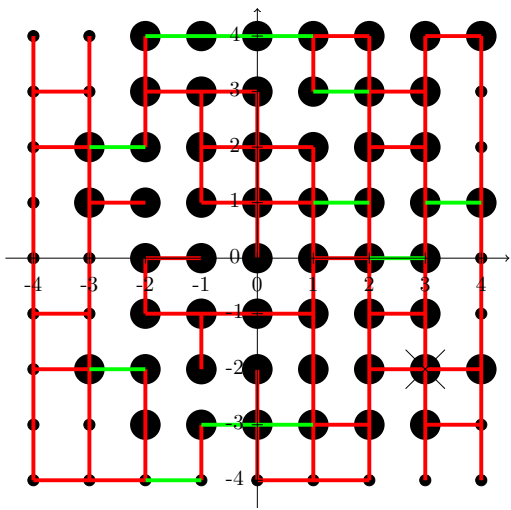

We define the \emph{special vertices} to be all the vertices in $S$ that are connected to the destination point and have the same latitude as an endpoint of a VNE (see Figure~\ref{2}). Notice that there is a finite number of special vertices and label them $1,2, \hdots, s$. We note that there must exist a path contained in the primary rectangle between each special vertex and the destination point. Indeed, the fact that all all VNEs are contained in the primary rectangle and the way we define passes allows us to find paths contained in the primary rectangle between the accessible infinite/upper and lower infinite and finite columns of the primary rectangle; this further allows us to find paths contained in the primary rectangle from the special points to the destination point.

Let $l_i$ be the length of a shortest path contained in the primary rectangle from $i \in \{1, 2, \hdots, s\}$ to the destination point and set the following constant which depends only on the local configuration of the maze inside the primary rectangle:
$$l'=1+ ((((l_1)2+l_2)2+l_3)2+ \hdots + l_{s-1})2+l_s.$$

The \emph{secondary rectangle} is obtained from the primary rectangle by augmenting it $l'$ units in each of the four directions. Note that given the local configuration of the maze inside the secondary rectangle, we can construct a finite algorithm $L'$ such that if the robot follows $L'$ starting from any special point, it visits the destination point without leaving the secondary rectangle. Indeed, assume the robot starts at the special point labeled $1$. We construct firstly a finite algorithm $L_1$ that takes the robot to the destination point with $|L_1| = l_1$. Then assume that the robot starts at the special point labeled $2$ and that it first follows the algorithm $L_1$. We write the algorithm $L_2$ as a concatenation of two sub-algorithms. The first cancels the action of $L_1$ and brings the robot back to the special point $2$ and the second sub-algorithm takes the robot further to the destination point. This can be done with at most $l_1 + l_2$ instructions, so without loss of generality $|L_2| \le l_1 + l_2$. Moreover, if the robot starts at any of the special points labeled $1$ or $2$ and follows $L_1L_2$ it gets to the destination point. We continue in this way: given $L_1, L_2, \hdots, L_{i-1}$ and assuming that the robot starts at the special point $i$, we construct $L_i$ as a concatenation of two sub-algorithms. The first brings the robot back to the special point $i$ and the second takes the robot further to the destination point. This can be done with $|L_i| \leq (|L_1| + \hdots + |L_{i-1}|)+l_i$. Finally, take $L'=L_1L_2 \hdots L_s$ with $|L'| \leq l'$ which has the property that if the robot follows $L'$ starting from any special point, it visits the destination point. The role of adding $1$ to the sum is that to ensure that the secondary rectangle augments non-trivially the primary rectangle.

In the rest of the section we define a series of very technical configurations. We group the mazes according to these configurations and obtain the desired countable cover at the end of the section. The importance of these configurations only becomes clear in Section~\ref{final} and we will recall them where appropriate.

For simplicity we use cardinal directions in our definitions. We say that \emph{the row $r_i$ is to the north of the row $r_j$} or \emph{above row $r_j$}, provided $i>j$. By an \emph{easternmost H(N)E} $e$ with a certain property $\mathcal{P}$ we mean that $e$ has $\mathcal{P}$ and no other H(N)E with $\mathcal{P}$ has longitude greater than $e$. These definitions easily extend to the other directions: \emph{westernmost}, \emph{uppermost}, \emph{lowermost}. In pairings (e.g.``the lowermost easternmost HNE with $\mathcal{P}$'') we always give priority to the first direction and then to the second one. For example, in order to find the uppermost easternmost HNE below all VNEs in the west strip, we first look for the row of highest latitude below all VNEs on which there is a HNE in the west strip and then on this row we pick the one HNE in the west strip with the largest longitude.

Define a \emph{west bump} to be any of the easternmost HNE in the west strip or at the border between the west strip and the obstacle strip (i.e. with at least one vertex in the west strip) on a row that intersects some finite column. For example in Figure~\ref{1}, the HNE between $(-4, 1)$ and $(-3, 1)$ and the HNE between $(-3, 2)$ and $(-2, 2)$ are both west bumps with the rows $r_1$ and $r_2$ intersecting the finite column $(-1, 1), (-1, 2), (-1, 3)$. Using symmetry, define similarly an \emph{east bump}. We note that there are a finite number of west and east bumps.

If there exists a row which is a path when restricted to the west strip, but contains a HNE, then call the smallest such row with respect to the standard well order on $\Z$ a \emph{magical west row}; define its \emph{west cutoff} to be its westernmost HNE. Define similarly a \emph{magical east row} and its \emph{east cutoff}.

We define a \emph{west pipe} to be any of the easternmost configurations in the west strip of three vertices $(x, y), (x+1, y), (x+2, y)$ where between $(x, y)$ and $(x+1, y)$ there is a HE and between $(x+1, y)$ and $(x+2, y)$ there is a HNE, which can be at the border between the west strip and the obstacle strip. For example in Figure~\ref{1}, $(-4, 2), (-3, 2), (-2, 2)$ is a west pipe. Note that a maze may have infinitely many west pipes. We define similarly an \emph{east pipe} to be any of the westernmost configurations in the east strip of three vertices $(x, y), (x+1, y), (x+2, y)$ where between $(x+1, y)$ and $(x+2, y)$ there is a HE and between $(x, y)$ and $(x+1, y)$ there is a HNE, which can be at the border between the east strip and the obstacle strip. For example in Figure~\ref{1}, $(2, 1), (3, 1), (4, 1)$ is an east pipe.

Furthermore, we define the \emph{special west pipe} to be the west pipe on the smallest row that has a west pipe, with respect to the standard well order on $\Z$, if such a row exists. Note that in Figure~\ref{1} the special west pipe may not be $(-4, -1), (-3, -1), (-2, -1)$ as we do not know from the picture whether there are west pipes on $r_0$ or $r_1$, but we do know that it is the west pipe on $r_{-1}$. We define similarly the \emph{special east pipe}. Note that if a maze does not have any special west pipe, then in the west strip any row is either a path or it is the complement of an infinite path followed by a finite path. 

We define an \emph{almost empty west row} to be a row that in the west strip is the complement of an infinite path followed by a non-empty finite path. Thus, in Figure~\ref{1}, both $r_0$ and $r_1$ cannot be almost empty west rows as the non-empty finite path in the west strip is missing for both of these columns; the edge between $(-3, 1)$ and $(-2, 1)$ does not belong to the west strip. We define similarly an \emph{almost empty east row}. We define the \emph{special almost empty west row} to be the smallest almost empty west row with respect to the standard well order on $\Z$, if such a row exists. We define the \emph{west cutoff} of a special almost empty west row to be its easternmost HNE in the west strip. We define similarly the \emph{special almost empty east row} and its \emph{east cutoff}. For example, if in Figure~\ref{1} $r_2$ was the special almost empty east row, its east cutoff would be the edge between $(3, 2)$ and $(4, 2)$. Finally, we define an \emph{empty west row} to be a row that in the west strip is empty; for the `special' label in this context, we need in addition that the latitude of the row is large in absolute value. So we define the \emph{special empty west row} to be the empty west row of smallest latitude, greater than $-3p$ (where the parameter of the primary rectangle, $p$, is defined above) with respect to the standard well order on $\Z$, if such a row exists. We define the \emph{natural special empty west row} to be the empty west row of smallest latitude, without the additional constraint. We define similarly the \emph{special empty east row} and the \emph{natural special empty west row}.

We define the \emph{upper west pass} to be the lowermost HE between the easternmost infinite column of the west strip and the westernmost upper infinite column with the property that its latitude $k$ is greater than that of any pass in the obstacle strip, if such a HE exists. We define similarly the \emph{upper east pass}, \emph{lower west pass} and \emph{lower east pass}. For example, in Figure~\ref{2} the edge between $(-3, -4)$ and $(-2, -4)$ is the lower west pass. Also, in Figure~\ref{3} the upper/lower west/east passes are the green edges.

Let us call the pair of columns at the border between the west strip and the obstacle strip $(c_a, c_{a+1})$, so $c_a$ is in the west strip and $c_{a+1}$ is in the obstacle strip. Let us call the pair of columns at the border between the obstacle strip and the east strip $(c_b, c_{b+1})$, so $c_b$ is in the obstacle strip and $c_{b+1}$ is in the east strip. We define the \emph{west ascending chain} (if such a structure exists) to be the finite sequence of HEs: $HE_a, HE_{a+1}, \hdots , HE_b$ such that $HE_a$ is the upper west pass and $HE_{m}$ is the lowermost HE between the pair of columns $(c_m, c_{m+1})$ at latitude at least that of $HE_{m-1}$ for $m=a+1, \hdots, b$ (see Figure~\ref{3}). Similarly, we define the \emph{east ascending chain}, \emph{west descending chain} and \emph{east descending chain}. If a west ascending chain $HE_a, \hdots , HE_b$ exists with $HE_b$ on some row $r_t$, we define the \emph{upper west constant} $c_{uw}:=t+p$, where $p$ is the parameter of the primary rectangle. We define similarly the constants \emph{lower west constant}, \emph{upper east constant} and \emph{lower east constant}.
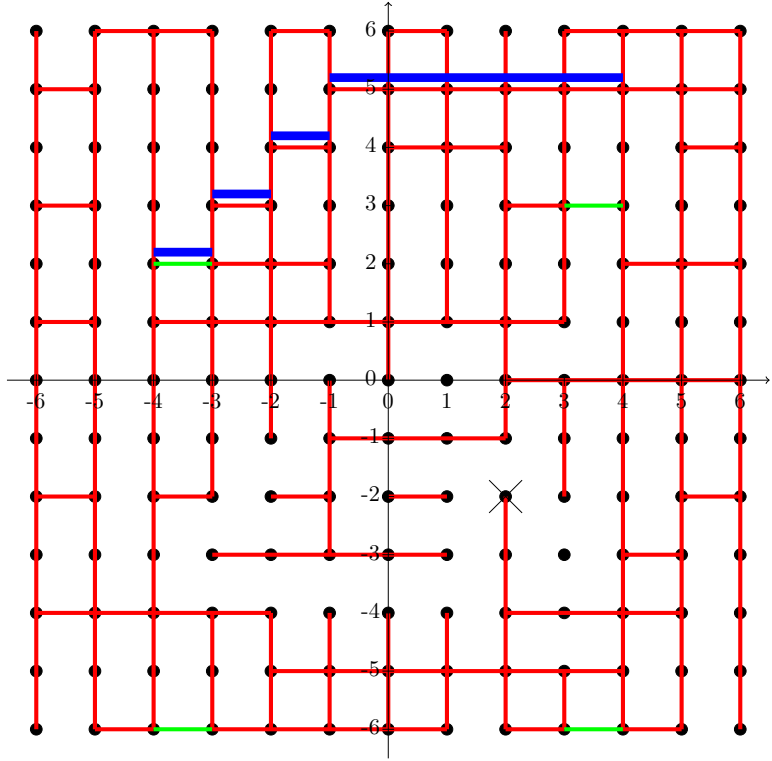
\begin{figure}[h!]
\centering
\resizebox{0.75\textwidth}{!}{
\begin{tikzpicture}
    [
        dot/.style={circle,draw=black, fill,inner sep=2pt},
        cross/.style={cross out, draw=black, minimum size=2*(#1-\pgflinewidth), inner sep=4pt, outer sep=4pt},
    ]

\foreach \x in {-6,...,6}{
 \foreach \y in {-6,...,6}{
    \node[dot] at (\x,\y){ };
}}

\draw[line width=0.7mm, red] (-6, -6) -- (-6, 6);
\draw[line width=0.7mm, red] (-5, -6) -- (-5, 6);
\draw[line width=0.7mm, red] (-4, -6) -- (-4, 6);
\draw[line width=0.7mm, red] (-3, -6) -- (-3, -4);
\draw[line width=0.7mm, red] (-3, -2) -- (-3, 6);
\draw[line width=0.7mm, red] (-2, -6) -- (-2, -4);
\draw[line width=0.7mm, red] (-2, -1) -- (-2, 6);
\draw[line width=0.7mm, red] (-1, -6) -- (-1, -4);
\draw[line width=0.7mm, red] (-1, -3) -- (-1, 0);
\draw[line width=0.7mm, red] (-1, 1) -- (-1, 6);
\draw[line width=0.7mm, red] (0, -6) -- (0, -4);
\draw[line width=0.7mm, red] (0, 0) -- (0, 6);
\draw[line width=0.7mm, red] (1, -6) -- (1, -4);
\draw[line width=0.7mm, red] (1, 1) -- (1, 6);
\draw[line width=0.7mm, red] (2, -6) -- (2, -2);
\draw[line width=0.7mm, red] (2, -1) -- (2, 6);
\draw[line width=0.7mm, red] (3, -6) -- (3, -5);
\draw[line width=0.7mm, red] (3, -2) -- (3, 0);
\draw[line width=0.7mm, red] (3, 1) -- (3, 6);
\draw[line width=0.7mm, red] (4, -6) -- (4, 6);
\draw[line width=0.7mm, red] (5, -6) -- (5, 6);
\draw[line width=0.7mm, red] (6, -6) -- (6, 6);

\draw[line width=0.7mm, red] (-5, 6) -- (-3, 6);
\draw[line width=0.7mm, red] (-2, 6) -- (-1, 6);
\draw[line width=0.7mm, red] (0, 6) -- (1, 6);
\draw[line width=0.7mm, red] (3, 6) -- (6, 6);
\draw[line width=0.7mm, red] (-6, 5) -- (-5, 5);
\draw[line width=1.5mm, blue] (-1, 5.2) -- (4, 5.2);
\draw[line width=0.7mm, red] (-1, 5) -- (6, 5);
\draw[line width=1.5mm, blue] (-2, 4.2) -- (-1, 4.2);
\draw[line width=0.7mm, red] (-2, 4) -- (-1, 4);
\draw[line width=0.7mm, red] (0, 4) -- (2, 4);
\draw[line width=0.7mm, red] (5, 4) -- (6, 4);
\draw[line width=0.7mm, red] (-6, 3) -- (-5, 3);
\draw[line width=1.5mm, blue] (-3, 3.2) -- (-2, 3.2);
\draw[line width=0.7mm, red] (-3, 3) -- (-2, 3);
\draw[line width=0.7mm, green] (3, 3) -- (4, 3);
\draw[line width=0.7mm, red] (2, 3) -- (3, 3);
\draw[line width=0.7mm, green] (-4, 2) -- (-3, 2);
\draw[line width=1.5mm, blue] (-4, 2.2) -- (-3, 2.2);
\draw[line width=0.7mm, red] (-3, 2) -- (-1, 2);
\draw[line width=0.7mm, red] (4, 2) -- (6, 2);
\draw[line width=0.7mm, red] (-4, 1) -- (2, 1);
\draw[line width=0.7mm, red] (-6, 1) -- (-5, 1);
\draw[line width=0.7mm, red] (2, 1) -- (3, 1);
\draw[line width=0.7mm, red] (2, 0) -- (6, 0);
\draw[line width=0.7mm, red] (-1, -1) -- (2, -1);
\draw[line width=0.7mm, red] (-6, -2) -- (-5, -2);
\draw[line width=0.7mm, red] (-4, -2) -- (-3, -2);
\draw[line width=0.7mm, red] (-2, -2) -- (-1, -2);
\draw[line width=0.7mm, red] (0, -2) -- (1, -2);
\draw[line width=0.7mm, red] (5, -2) -- (6, -2);
\draw[line width=0.7mm, red] (-3, -3) -- (1, -3);
\draw[line width=0.7mm, red] (4, -3) -- (5, -3);
\draw[line width=0.7mm, red] (-6, -4) -- (-2, -4);
\draw[line width=0.7mm, red] (2, -4) -- (5, -4);
\draw[line width=0.7mm, red] (-5, -6) -- (-4, -6);
\draw[line width=0.7mm, green] (-4, -6) -- (-3, -6);
\draw[line width=0.7mm, red] (-3, -6) -- (1, -6);
\draw[line width=0.7mm, red] (2, -6) -- (3, -6);
\draw[line width=0.7mm, green] (3, -6) -- (4, -6);
\draw[line width=0.7mm, red] (4, -6) -- (5, -6);
\draw[line width=0.7mm, red] (-2, -5) -- (4, -5);

\draw (2,-2) node[cross] {};

\foreach \x in {-6,...,6}
    \draw (\x,.1) -- node[below,yshift=-1mm] {\x} (\x,-.1);
\foreach \y in {-6,...,6}
    \draw (.1,\y) -- node[below,xshift=-3mm, yshift=3mm] {\y} (-.1,\y);
\draw[->,line width=0.15mm] (0,-6.5) -- (0,6.5);
\draw[->,line width=0.15mm] (-6.5,0) -- (6.5,0);
\end{tikzpicture}
}
\caption{We assume that there are no vertical edges removed other than the ones shown in the figure. The green edges are the upper/lower west/east passes. Neither the HE $(-4, -2), (-3, -2)$ nor $(-4, 1), (-3, 1)$ is the upper west pass, as they are not above all the passes in the obstacle strip. The blue coloured edges in order from left to right form the west ascending chain.} \label{3}
\end{figure}

Assume that the upper west pass is on some row $r_k$. We define an \emph{upper west paired HNEs} to be any pair of HNEs with the same longitude in the west strip such that the upper HNE is at latitude $k$ and the lower HNE is at latitude at most $k-c_{uw}$, where $c_{uw}$ defined above is the upper-west constant. We define similarly the \emph{upper east paired HNEs}, \emph{lower west paired HNEs} and \emph{lower east paired HNEs} with respect to the corresponding constants $c_{ue}$, $c_{lw}$, and $c_{le}$, respectively. We define the \emph{special upper west paired HNEs} (if such a structure exists) to be the upper west paired HNEs with the uppermost and easternmost lower HNE. We recall that in all such instances we give priority to the first condition and then the second one. We define similarly the \emph{special upper east paired HNEs}, \emph{special lower west paired HNEs} and \emph{special lower east paired HNEs}.

With $k$ being as always the latitude of the upper west pass, we define the \emph{upper west pipe} to be the west pipe on the row $r_k$, if one exists. We define similarly the \emph{lower west pipe}, the \emph{upper east pipe} and the \emph{lower east pipe}. We define the \emph{upper west cutoff} to be the easternmost HNE on the row $r_k$ in the west strip, if one exists. We define similarly the \emph{lower west cutoff}, the \emph{upper east cutoff} and the \emph{lower east cutoff}.

We define the \emph{upper west HNE} (if such a structure exists) to be the lowermost westernmost HNE at the north-east of the uppermost westernmost VNE. We define similarly the \emph{upper east HNE}, \emph{lower west HNE} and the \emph{lower east HNE}.

Being consistent with the constants $a$ and $b$ introduced in the definition of the west ascending chain, we define the parameters $h_{(m, m+1)}$ to be the latitude of the uppermost HE between two consecutive upper infinite columns or between an infinite column and an upper infinite column on $c_m$ and $c_{m+1}$ if such a HE exists and $\infty$ otherwise for $m=a, \hdots, b$. 

We define similarly the parameters $l_{(m, m+1)}$ to be the latitude of the lowermost HE between two consecutive lower infinite columns or between an infinite column and a lower infinite column on $c_m$ and $c_{m+1}$ if such a HE exists and infinity otherwise for $m=a, \hdots, b$.

We define the parameters $w_1, e_1, w_2, e_2, w_3, e_3, w_4, e_4$ to be the latitude of the magical west/east row, the special almost empty west/east row, the special empty west/east row, and the natural special empty west/east row if such a configuration exists and infinity otherwise, respectively. 

We finally define the \emph{tertiary rectangle} to be the subgraph contained in the smallest rectangle that contains the secondary rectangle and all the west/east bumps, upper/lower west/east cutoffs, upper/lower west/east pipes, special west/east pipes, upper/lower west/east passes, west/east ascending/descending chains, upper/lower west/east paired HNEs and the upper/lower west/east HNEs.

As in the case of the primary rectangle, let $q$ be the smallest positive integer such that the tertiary rectangle is strictly contained 
in the interior of the square centred at the origin with vertices $\{ \pm \frac{q}{3},  \pm \frac{q}{3} \}$. We call $q$, together with the upper/lower west/east constants, $h_{(m, m+1)}, l_{(m, m+1)}$ for $m=a, \hdots, b$, $w_1, e_1, w_2, e_2, w_3, e_3, w_4, e_4$ the \emph{parameters of the tertiary rectangle}. Therefore, when we construct an algorithm by inspecting the tertiary rectangle, we have access to the subgraph contained in the tertiary rectangle and the values of all its parameters.

We group the mazes in $\overline{\mathcal{F}}$ according to agreeing on the destination point, the subgraph contained in the square $\{ \pm q, \pm q \}$ and the set of parameters of the tertiary rectangle. We thus obtain a countable cover $\overline{\mathcal{F}} = \cup_{i=1}^{ \infty} F_i$. It is obvious directly from the definitions that we set above that such a construction is achievable.

All of these definitions are used in the following section to prove Theorem~\ref{ch2th2} and the relevant ones will be recalled where appropriate.

\section{The General Model}\label{final}
In this section we prove Theorem~\ref{ch2th2}. 

Following our strategy, we assume that we are given $F_i$ and a finite algorithm $X$ and we aim to construct a finite algorithm $A$ such that $XA$ visits the destination point of $F_i$. We construct the algorithm $A$ from several sub-algorithms treated in separate subsections, each with a specific task: in \textbf{Part I} we position the robot in the east strip, in \textbf{Part II} we position the robot in the west strip at latitude 0 and in \textbf{Part III} we guide the robot through the destination point. In each part we consider a finite number of cases for the subsets $F_i$ so that although a sub-algorithm depends quantitatively on $F_i$ and $X$, it does depend qualitatively only on the case. We treat each case separately. According to their degree of generality, we label the broader cases as ``Propositions'', and the more specific cases as ``Claims''.

\begin{proof}[Proof of Theorem~\ref{ch2th2}]
Let $F_i$ be any of the classes of mazes defined above and assume we are given a finite algorithm $X$. Let $\lambda := |X|$.
\subsection*{Part 0.} 
We recall the finite algorithm $L'$ defined in Section~\S\ref{fnv} for a particular maze $M$, which had the property that if the robot starts at any special point of $M$ and follows $L'$, it visits the destination point. Take $M \in F_i$ and construct its $L'$ as described in Section~\S\ref{fnv}. We claim that for any $M' \in F_i$, the algorithm $L'$ has the same property in $M'$, i.e. if the robot starts at any special point of $M'$ and follows $L'$, it visits the destination point of $M'$. This follows from the fact that all the mazes in $F_i$ share the destination point, the secondary rectangle and in particular the set of special points. Therefore, we pick this $L'$ as a representative for the set of mazes $F_i$.

We now define the algorithm $$L=L_E=L' \text{ } N^{\varepsilon} \text{ } ME(|L' N^{\varepsilon}|, |L' N^{\varepsilon}|),$$ where the correcting constant ${\varepsilon} \in \Z$ is picked such that $|L' N^{\varepsilon}|_N=|L' N^{\varepsilon}|_S$ and therefore $|L|_N=|L|_S$; let $l:=|L|$. The counterpart of $L=L_E$ is $$L_W = L' \text{ } N^{\varepsilon} \text{ } MW(|L' N^{\varepsilon}|, |L' N^{\varepsilon}|),$$ and as before we have $|L_W|_N=|L_W|_S$ and also $|L_W|=l$.

The algorithm $L$ is a generic algorithm used in several other algorithms below. We remark that if the robot starts from a special point and it follows $L$ in any maze in $F_i$, it gets to the destination point; this property is inherited from $L'$. We further note that if the robot is at the origin on a maze with no VNEs and it follows $L$, then it returns to the $x$-axis and its longitude does not decrease. These properties are crucial in order to apply Lemma~\ref{l2}.

We finally note that all mazes in $F_i$ also share the same parameter of the primary rectangle $p$ and parameter of the tertiary rectangle $q$ and we keep this notation consistent for the rest of the proof.
\subsection*{Part I}
The algorithm $rough \_ positioning \_ east$ defined in this part aims to either position the robot in the east strip or to make the robot visit the destination point. We define
\[
RPE := ME (\lambda+p, \lambda+p) \text{ } N^{l+\lambda+4p} \text{ } L \text{ } S^{2(l+\lambda+4p)} \text{ } L.
\]
\begin{prop}\label{pp1}
For any maze in $F_i$, after the robot follows the algorithm $X \text{ } RPE$, it is either in the east strip or it has visited the destination point.
\end{prop}
\begin{proof}
Pick any maze in $F_i$. We claim that by our choice of parameters of $ME$, after the robot follows $X \text{ } ME (\lambda+p, \lambda+p)$, it is either in the east strip or in the obstacle strip, but not in the west strip. Indeed, assume for a contradiction that after the robot follows $X \text{ } ME (\lambda+p, \lambda+p)$, it is in the west strip. Denote by $\textbf{x} = (x, y)$ the position of the robot after it follows $X$ starting from the origin. By assumptionm $\textbf{x}$ must be in the west strip as the algorithm $ME$ has no instruction $W$. Therefore, as the robot follows $ME (\lambda+p, \lambda+p)$, it does not visit any endvertex of a VNE. We recall that all mazes in $F_i \subset \overline{\mathcal{F}}$ have the property that for every VNE at least one of its vertices is accessible, hence the westernmost column of the obstacle strip $c_{a+1}$ is accessible from $\textbf{x}$. The robot starts in the origin which is at most $p$ units in longitude away from the obstacle strip as the primary rectangle contains the origin and all the VNEs. Hence the column$c_{y+\lambda+p}$ is not in the west strip. Moreover, every pair of consecutive columns at longitude between $y$ and $a+1$ are connected by a HE at some latitude between $x+\lambda+p$ and $x-\lambda-p$, as the primary rectangle contains all the passes and the VNEs. Therefore, if the robot starts from $\textbf{x}$ and follows $ME (\lambda+p, \lambda+p)$, it gets to a longitude at least $y+\lambda+p$, which is not in the west strip. This contradiction proves the claim.

Hence, after the robot follows $X \text{ } ME (\lambda+p, \lambda+p)$, its longitude is at least $a+1$ and so it is either in the east strip or in the obstacle strip. In the former case, after the robot follows $X \text{ } RPE$, it remains in the east strip. Indeed, while the robot follows $N^{l+\lambda+4p} \text{ } L$ starting in the east strip, its latitude is too high to meet any VNE and on a maze with no VNE if the robot follows $L$ its longitude does not decrease so the robot remains in the east strip. Therefore, after the robot follows also $S^{2(l+\lambda+4p)} \text{ } L$ its latitude is too low to meet any VNE, and it remains in the east strip by a similar argument. To conclude, if the robot gets to the east strip after it follows the initial segment $X \text{ } ME (\lambda+p, \lambda+p)$, then it remains in the east strip after it follows $X \text{ } RPE$.

In the latter case, after the robot follows $X \text{ } ME (\lambda+p, \lambda+p)$, it is in the obstacle strip either in (1) a lower infinite column or a finite column or (2) an upper infinite column. In case (1), after the robot follows $X \text{ } ME (\lambda+p, \lambda+p) \text{ } N^{l+\lambda+4p}$ it gets to a special point. Therefore after the robot follows $X \text{ } ME (\lambda+p, \lambda+p) \text{ } N^{l+\lambda+4p} \text{ } L$, it gets to the destination point. In case (2), while the robot follows $N^{l+\lambda+4p} \text{ } L$ it does not meet any VNE and its longitude does not decrease, so after it follows the initial segment $X \text{ } ME (\lambda+p, \lambda+p) \text{ } N^{l+\lambda+4p} \text{ } L$, it is either in the east strip or in the obstacle strip in an upper infinite column. In both cases, it is clear that after the robot follows $X \text{ } RPE$ it is either in the east strip or it has visited the destination point.
\end{proof}
\begin{remark}
In the first part of the proof of Proposition~\ref{pp1} we argue that the parameters $(\lambda+p, \lambda+p)$ of $ME$ are large enough for the robot to have longitude at least $a+1$. The key of this argument is two-fold: firstly, all the passes are in the primary rectangle which has parameter $p$; secondly, if the robot starts from the origin and follows the algorithm $X$ with $|X|=\lambda$, it can not advance more than $\lambda$ columns east or west and any two consecutive columns between its initial and final position are connected at latitude no more than $\lambda$ in absolute value. We do not expand this argument every time we use it, but instead we use the phrase ``by our choice of parameters'' to mark that the same reasoning is used in similar instances to prove that the robot advances westwards/eastwards to the desired longitude.
\end{remark}

At the end of \textbf{Part I}, we note that although we used in the proof of Proposition~\ref{pp1} the fact there are no infinite columns in the obstacle strip, a variation of $RPE$ can be used to position the robot in the east strip, even if we drop this assumption. This note is important, because it shows that \textbf{Part I} can be generalised to improve Theorem~\ref{ch2th2} by dropping the consecutive column condition for the finite number of VNEs. To present this variation, assume that infinite columns are allowed in the obstacle strip, i.e. the (finitely many) VNEs need not be in consecutive columns. 

We define now the algorithm $RPE'$ that generalises $RPE$ as described above. It is formed by $\lambda+p$ subalgorithms $S_1, \hdots S_{\lambda+p}$ concatenated in order. We define
\[
S_i = N^{\lambda_i+p+2l} \text{ } L \text{ } S^{\mu_i+p+2l} \text{ } ME (\gamma_i, 1),
\]
for $i=1, \hdots, \lambda+p$. The parameters $\lambda_i, \mu_i, \gamma_i \in \mathbb{N}$ are chosen to be at least the number of instructions written in the whole algorithm until they occur, for example we can take $\lambda_1 = |X| = \lambda$, $\mu_1 = |X \text{ } N^{\lambda_1+p+2l} \text{ } L|$, $\gamma_1 = |X \text{ } N^{\lambda_1+p+2l} \text{ } L \text{ } S^{\mu_1+p+2l}|$, etc. Finally, let
\[
RPE' = S_1 \text{ } S_2 \text{ } \hdots \text{ } S_{\lambda+p}.
\]

Note that for the mazes we consider, we first replace all the VNEs with VEs which do not change the connected component of the origin, so every pair of consecutive columns in the obstacle strip must be connected by an accessible HE. The reason why $RPE'$ indeed generalises $RPE$ is similar to the argument in the proof of Proposition~\ref{pp1}: here, after the robot follows every $S_i$ it either moves at least one column to the east or it has visited the destination point.

Moving on from this digression, by Proposition~\ref{pp1} we may assume that we are given $F_i$ and a finite algorithm $X$ with $\lambda = |X|$ such that after the robot follows $X$ in any maze in $F_i$, it is either in the east strip or it has visited the destination point. Without loss of generality, we assume that the robot is in the east strip and our aim is to build a finite algorithm $A$ such that $X A$ solves $F_i$.
\subsection*{Part II} The algorithm $reset\_latitude\_west$ defined in this part aims to either position the robot in the west strip on the $x$-axis (i.e. at latitude $0$) or to make the robot visit the destination point.
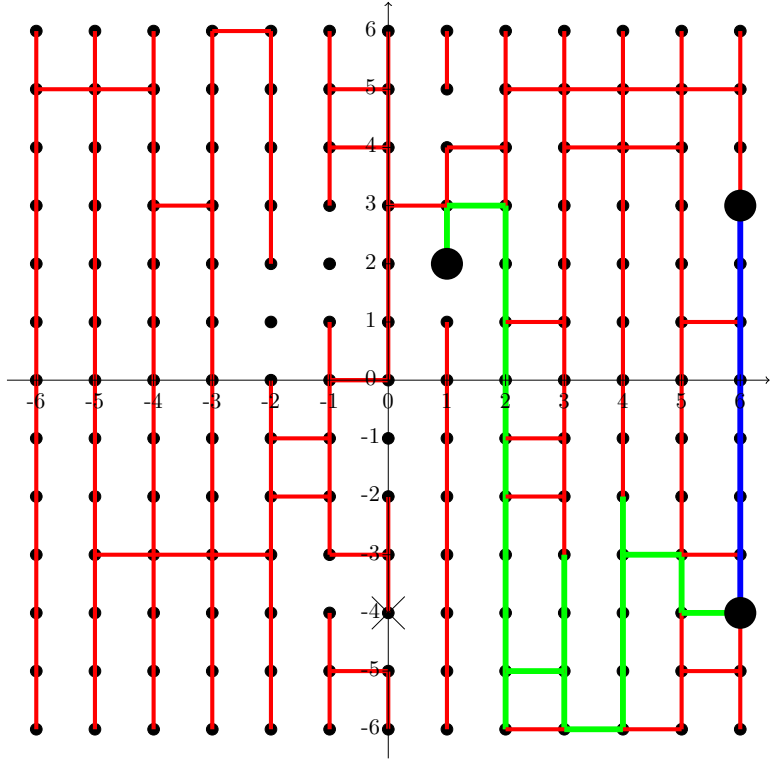
\begin{figure}[h!]
\centering
\resizebox{0.75\textwidth}{!}{
\begin{tikzpicture}
    [
        dot/.style={circle,draw=black, fill,inner sep=2pt},
        cross/.style={cross out, draw=black, minimum size=2*(#1-\pgflinewidth), inner sep=4pt, outer sep=4pt},
    ]

\foreach \x in {-6,...,6}{
 \foreach \y in {-6,...,6}{
    \node[dot] at (\x,\y){ };
}}

\draw[line width=0.7mm, red] (-6, -6) -- (-6, 6);
\draw[line width=0.7mm, red] (-5, -6) -- (-5, 6);
\draw[line width=0.7mm, red] (-4, -6) -- (-4, 6);
\draw[line width=0.7mm, red] (-3, -6) -- (-3, 6);
\draw[line width=0.7mm, red] (-2, 6) -- (-2, 2);
\draw[line width=0.7mm, red] (-2, 0) -- (-2, -6);
\draw[line width=0.7mm, red] (-1, 6) -- (-1, 3);
\draw[line width=0.7mm, red] (-1, 1) -- (-1, -3);
\draw[line width=0.7mm, red] (-1, -4) -- (-1, -6);
\draw[line width=0.7mm, red] (0, 6) -- (0, 0);
\draw[line width=0.7mm, red] (0, -2) -- (0, -4);
\draw[line width=0.7mm, red] (0, -5) -- (0, -6);
\draw[line width=0.7mm, red] (1, 6) -- (1, 5);
\draw[line width=0.7mm, red] (1, 4) -- (1, 2);
\draw[line width=0.7mm, red] (1, 1) -- (1, -6);
\draw[line width=0.7mm, red] (2, 6) -- (2, -6);
\draw[line width=0.7mm, red] (3, 6) -- (3, -6);
\draw[line width=0.7mm, red] (4, 6) -- (4, -6);
\draw[line width=0.7mm, red] (5, 6) -- (5, -6);
\draw[line width=0.7mm, red] (6, 6) -- (6, -6);

\draw[line width=1mm, green] (5, -4) -- (5, -3);
\draw[line width=1mm, green] (4, -2) -- (4, -6);
\draw[line width=1mm, green] (3, -3) -- (3, -6);
\draw[line width=1mm, green] (2, 3) -- (2, -6);
\draw[line width=1mm, green] (1, 2) -- (1, 3);

\draw[line width=0.7mm, red] (2, -6) -- (5, -6);
\draw[line width=0.7mm, red] (-1, -5) -- (0, -5);
\draw[line width=0.7mm, red] (2, -5) -- (3, -5);
\draw[line width=0.7mm, red] (5, -5) -- (6, -5);
\draw[line width=0.7mm, red] (5, -4) -- (6, -4);
\draw[line width=0.7mm, red] (-5, -3) -- (-2, -3);
\draw[line width=0.7mm, red] (-1, -3) -- (0, -3);
\draw[line width=0.7mm, red] (4, -3) -- (6, -3);
\draw[line width=0.7mm, red] (-2, -2) -- (-1, -2);
\draw[line width=0.7mm, red] (2, -2) -- (3, -2);
\draw[line width=0.7mm, red] (-2, -1) -- (-1, -1);
\draw[line width=0.7mm, red] (2, -1) -- (3, -1);
\draw[line width=0.7mm, red] (-1, 0) -- (0, 0);
\draw[line width=0.7mm, red] (-1, 0) -- (0, 0);
\draw[line width=0.7mm, red] (2, 1) -- (3, 1);
\draw[line width=0.7mm, red] (5, 1) -- (6, 1);
\draw[line width=0.7mm, red] (-4, 3) -- (-3, 3);
\draw[line width=0.7mm, red] (0, 3) -- (2, 3);
\draw[line width=0.7mm, red] (-1, 4) -- (0, 4);
\draw[line width=0.7mm, red] (1, 4) -- (2, 4);
\draw[line width=0.7mm, red] (3, 4) -- (5, 4);
\draw[line width=0.7mm, red] (-6, 5) -- (-4, 5);
\draw[line width=0.7mm, red] (-1, 5) -- (0, 5);
\draw[line width=0.7mm, red] (2, 5) -- (6, 5);
\draw[line width=0.7mm, red] (-3, 6) -- (-2, 6);

\draw[line width=1mm, green] (5, -4) -- (6, -4);
\draw[line width=1mm, green] (5, -3) -- (4, -3);
\draw[line width=1mm, green] (4, -6) -- (3, -6);
\draw[line width=1mm, green] (3, -5) -- (2, -5);
\draw[line width=1mm, green] (1, 3) -- (2, 3);

\draw[line width=1mm, blue] (6, 3) -- (6, -4);

\draw node[fill,circle,inner sep=5.5pt,minimum size=2pt] at (6, -4) {};
\draw node[fill,circle,inner sep=5.5pt,minimum size=2pt] at (1, 2) {};
\draw node[fill,circle,inner sep=5.5pt,minimum size=2pt] at (6, 3) {};

\draw (0,-4) node[cross] {};

\foreach \x in {-6,...,6}
    \draw (\x,.1) -- node[below,yshift=-1mm] {\x} (\x,-.1);
\foreach \y in {-6,...,6}
    \draw (.1,\y) -- node[below,xshift=-3mm, yshift=3mm] {\y} (-.1,\y);
\draw[->,line width=0.15mm] (0,-6.5) -- (0,6.5);
\draw[->,line width=0.15mm] (-6.5,0) -- (6.5,0);
\end{tikzpicture}
}
\caption{\textbf{Part II, Case (1).} There is no pass between the obstacle strip and the east strip. We assume that there are no VEs removed other than the ones shown in the figure and that the position of the robot after following $X$ is $(6, 3)$. The first segment $S^{\lambda+p+l}$ of $RLW$ takes the robot to some very small latitude $j$ such that if it follows $L$ starting from any point of $r_j$ in the east strip, it will always remain in the east strip. This can be done, as there is no pass between the obstacle strip and the east strip. For our example, we may assume that after the robot follows $X S^{\lambda+p+l}$ it gets to the point $(6, -4)$, though this latitude should be much smaller. The green route to the special point $(1, 2)$ is the route of the robot if it would follow the algorithm $MW$. The algorithm $SMW$ used in $RLW$ generalises $MW$ by inserting the algorithm $L$ between locomotory moves. However, $L$ is constructed in such a way that if the robot follows it while it is in the east strip, its longitude does not increase. Moreover, the latitude of the robot is so small that it will never pass from the column $c_{b+1}$ to $c_b$ while following $L$. By Lemma~\ref{l2} and the choice of parameters of $SMW$, the robot reaches the special point $(1, 2)$ while executing a locomotory move. Immediately afterwards, it executes $L$ and it gets to the destination point. Finally, we remark that when the robot reaches the obstacle strip from the east strip for the first time, it does not enter the finite column $(1, 2), (1, 3), (1, 4)$ via the HE $(1, 4), (2, 4)$ or indeed it does not enter any other finite column which is above $R$. Indeed, this follows from the order of locomotory moves in $SMW$ which prioritises smaller latitudes.} \label{4}
\end{figure}

\noindent
\textbf{Case (1). } We assume that the mazes in $F_i$ do not contain a pass between the obstacle strip and the east strip. Then from the assumptions on the mazes in $F_i$, the east strip is connected to a finite column in the easternmost column of the obstacle strip $c_b$. This follows from the fact that for every VNE of every maze in $F_i$, at least one of its vertices is accessible from the origin. Let $R$ be the lowermost finite column in $c_b$ such that there exists a HE between $R$ and the east strip. Let $\textbf{v} = (b,i)$ be the lowermost vertex of the finite column $R$. In this case, we define the algorithm
\[
RLW:= S^{\lambda+p+l}  \text{ } SMW(2\lambda+2p+l,\lambda+2p,L).
\]
\textbf{Claim. } For any maze in $F_i$, after the robot follows the algorithm $X \text{ }RLW$, it visits the destination point.
\begin{proof}
After the robot follows $X \text{ } S^{\lambda+p+l}$, it is in the east strip at a certain point $\textbf{x} = (x, j)$, with $j \le i-l$. By the choice of parameters and by Lemma~\ref{l2}, while the robot follows $SMW(2\lambda+2p+l,\lambda+2p,L)$ it advances westwards in the east strip oscillating about the row $r_j$. It passes for the first time from the column $c_{b+1}$ to the column $c_b$ not while executing $L$, but while executing a locomotory move. Moreover, if we well order $\Z$ by $j<1+j<-1+j<2+j<-2+j< \hdots$, then the robot passes for the first time from the column $c_{b+1}$ to the column $c_b$ through the smallest HE with respect to this order and so it gets to the point $\textbf{v}$, which is a special point. Immediately afterwards, it follows $L$ and it reaches the destination point (see Figure~\ref{4}). The conclusion follows.
\end{proof}
\noindent
\begin{figure}[h!]
\centering
\resizebox{0.75\textwidth}{!}{
\begin{tikzpicture}
    [
        dot/.style={circle,draw=black, fill,inner sep=2pt},
        cross/.style={cross out, draw=black, minimum size=2*(#1-\pgflinewidth), inner sep=4pt, outer sep=4pt},
    ]

\foreach \x in {-6,...,6}{
 \foreach \y in {-6,...,6}{
    \node[dot] at (\x,\y){ };
}}

\draw[line width=0.7mm, red] (-6, -6) -- (-6, 6);
\draw[line width=0.7mm, red] (-5, -6) -- (-5, 6);
\draw[line width=0.7mm, red] (-4, -6) -- (-4, 6);
\draw[line width=0.7mm, red] (-3, -6) -- (-3, 6);
\draw[line width=0.7mm, red] (-2, 6) -- (-2, 2);
\draw[line width=0.7mm, red] (-2, 0) -- (-2, -6);
\draw[line width=0.7mm, red] (-1, 6) -- (-1, 3);
\draw[line width=0.7mm, red] (-1, 1) -- (-1, -3);
\draw[line width=0.7mm, red] (-1, -4) -- (-1, -6);
\draw[line width=0.7mm, red] (0, 6) -- (0, 0);
\draw[line width=0.7mm, red] (0, -2) -- (0, -4);
\draw[line width=0.7mm, red] (0, -5) -- (0, -6);
\draw[line width=0.7mm, red] (1, 6) -- (1, 5);
\draw[line width=0.7mm, red] (1, 4) -- (1, 2);
\draw[line width=0.7mm, red] (1, 1) -- (1, -6);
\draw[line width=0.7mm, red] (2, 6) -- (2, -6);
\draw[line width=0.7mm, red] (3, 6) -- (3, -6);
\draw[line width=0.7mm, red] (4, 6) -- (4, -6);
\draw[line width=0.7mm, red] (5, 6) -- (5, -6);
\draw[line width=0.7mm, red] (6, 6) -- (6, -6);

\draw[line width=1mm, green] (5, -4) -- (5, -3);
\draw[line width=1mm, green] (4, -4) -- (4, -3);
\draw[line width=1mm, green] (4, -2) -- (4, -6);
\draw[line width=1mm, green] (3, -3) -- (3, -6);
\draw[line width=1mm, green] (2, -6) -- (2, -2);
\draw[line width=1mm, green] (1, -2) -- (1, -4);

\draw[line width=0.7mm, red] (2, -6) -- (5, -6);
\draw[line width=0.7mm, red] (-1, -5) -- (0, -5);
\draw[line width=0.7mm, red] (2, -5) -- (3, -5);
\draw[line width=0.7mm, red] (5, -5) -- (6, -5);
\draw[line width=0.7mm, red] (5, -4) -- (6, -4);
\draw[line width=0.7mm, red] (-5, -3) -- (-2, -3);
\draw[line width=0.7mm, red] (-1, -3) -- (0, -3);
\draw[line width=0.7mm, red] (4, -3) -- (6, -3);
\draw[line width=0.7mm, red] (-2, -2) -- (-1, -2);
\draw[line width=0.7mm, red] (1, -2) -- (2, -2);
\draw[line width=0.7mm, red] (2, -2) -- (3, -2);
\draw[line width=0.7mm, red] (-2, -1) -- (-1, -1);
\draw[line width=0.7mm, red] (2, -1) -- (3, -1);
\draw[line width=0.7mm, red] (-1, 0) -- (0, 0);
\draw[line width=0.7mm, red] (-1, 0) -- (0, 0);
\draw[line width=0.7mm, red] (1, 0) -- (2, 0);
\draw[line width=0.7mm, red] (2, 1) -- (3, 1);
\draw[line width=0.7mm, red] (5, 1) -- (6, 1);
\draw[line width=0.7mm, red] (-4, 3) -- (-3, 3);
\draw[line width=0.7mm, red] (0, 3) -- (2, 3);
\draw[line width=0.7mm, red] (-1, 4) -- (0, 4);
\draw[line width=0.7mm, red] (1, 4) -- (2, 4);
\draw[line width=0.7mm, red] (3, 4) -- (5, 4);
\draw[line width=0.7mm, red] (-6, 5) -- (-4, 5);
\draw[line width=0.7mm, red] (-1, 5) -- (0, 5);
\draw[line width=0.7mm, red] (2, 5) -- (6, 5);
\draw[line width=0.7mm, red] (-3, 6) -- (-2, 6);

\draw[line width=1mm, green] (6, -4) -- (5, -4);
\draw[line width=1mm, green] (4, -3) -- (5, -3);
\draw[line width=1mm, green] (4, -6) -- (3, -6);
\draw[line width=1mm, green] (3, -5) -- (2, -5);
\draw[line width=1mm, green] (2, -2) -- (1, -2);

\draw[line width=1mm, blue] (6, 3) -- (6, -4);

\draw node[fill,circle,inner sep=5.5pt,minimum size=2pt] at (6, 3) {};
\draw node[fill,circle,inner sep=5.5pt,minimum size=2pt] at (6, -4) {};
\draw node[fill,circle,inner sep=5.5pt,minimum size=2pt] at (1, -4) {};

\draw (0,-4) node[cross] {};

\foreach \x in {-6,...,6}
    \draw (\x,.1) -- node[below,yshift=-1mm] {\x} (\x,-.1);
\foreach \y in {-6,...,6}
    \draw (.1,\y) -- node[below,xshift=-3mm, yshift=3mm] {\y} (-.1,\y);
\draw[->,line width=0.15mm] (0,-6.5) -- (0,6.5);
\draw[->,line width=0.15mm] (-6.5,0) -- (6.5,0);
\end{tikzpicture}
}
\caption{\textbf{Part II, Case (2).} There is a pass $\pi$, which for this example is $(1, 0), (2, 0)$ between a lower infinite column and the east strip. We assume that there are no VEs removed other than the ones shown in the figure and that the position of the robot after following $X$ is $(6, 3)$. The first segment $S^{\lambda+p}$ of $RLW$ takes the robot at a latitude lower than that of the pass $\pi$. For our example, we may assume that after the robot follows $X S^{\lambda+p}$ it gets to the point $(6, -4)$, though this latitude should be much smaller. While the robot is in the east strip, after it executes $K = N^{2\lambda+4p} S^{2\lambda+4p}$, it returns to the starting point. By the choice of parameters, the robot enters the easternmost lower infinite column at longitude $b$ for the first time via a locomotory move (in our case, $b=1$). Ignoring, as we may, the action of $K$ in the east strip, the path of the robot to the column $c_b$ is coloured in green. Immediately after the robot enters the column $c_b$, it executes $K$ which sets it latitude so small that the parameters of $SMW$ are not large enough to make the robot visit any other configurations in the obstacle strip other than the lower infinite columns.} \label{5}
\end{figure}
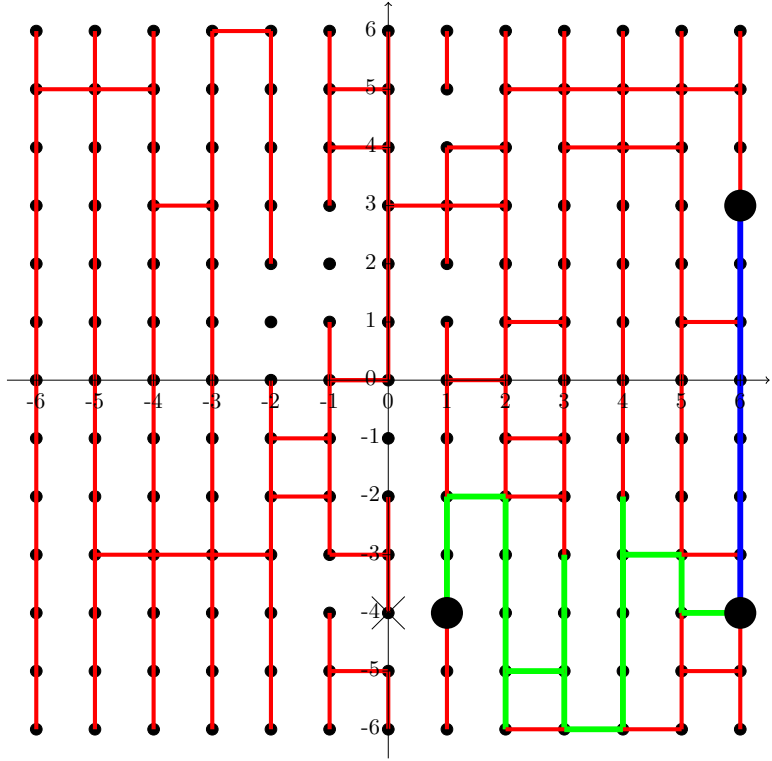
\textbf{Case (2).} We assume without loss of generality that the mazes in $F_i$ contain a pass $\pi$ between the easternmost lower infinite column and the east strip. In this case, we define the algorithm
\[
RLW:= S^{\lambda+p} \text{ } SMW(2\lambda+2p,\lambda+2p,K) \text{ } N^{2 \lambda+6p+l} \text{ } L_W \text{ } S^{2p+l-k},
\]
where $K = N^{2\lambda+4p} S^{2\lambda+4p}$ and $k$ is the latitude of the lowermost special vertex.
\begin{prop}\label{pp192}
For any maze in $F_i$, after the robot follows the algorithm $X \text{ }RLW$, it is either in the west strip on the $x$-axis or it has visited the destination point.
\end{prop}
\begin{proof}
After the robot follows $X \text{ } S^{\lambda+p}$ it is in the east strip at a certain latitude say $j$, smaller than the latitude of the pass $\pi$. By the choice of parameters and by Lemma~\ref{l2}, while the robot follows $SMW(2\lambda+2p,\lambda+2p,K)$ it advances westwards in the east strip oscillating about the row $r_j$. It passes for the first time from the east strip to the obstacle strip not while executing $K$. Moreover, if we well order $\Z$ by $j<1+j<-1+j<2+j<\hdots,$ then the robot passes from the east strip to the easternmost lower infinite column in the obstacle strip through the smallest HE with respect to this order. Immediately afterwards, it follows $K$ and gets at latitude $2\lambda+4p$ below the easternmost lowermost special vertex. By the choice of parameters the robot advances westwards only through lower infinite columns while in the obstacle strip. Therefore, after the robot follows $X \text{ } S^{\lambda+p} \text{ } SMW(2\lambda+2p,\lambda+2p,K)$, it is either (1) in the west strip at latitude $2\lambda+4p$ below the lowermost special vertex, i.e. at latitude $k-2\lambda-4p$ or (2) in the obstacle strip in a lower infinite column $c_m$ at latitude $2\lambda+4p$ below some special vertex (see Figure~\ref{5}).

In case (1), while the robot follows $N^{2 \lambda+6p+l} \text{ } L_W$ its latitude is too large for it to hit any VNE and after it follows $N^{2 \lambda+6p+l} \text{ } L_W$ its longitude does not increase, so it remains in the west strip. Hence, after it follows $X \text{ } RLW$, the robot is in the west strip on the $x$-axis.

In case (2), after the robot follows $N^{2 \lambda+6p+l}$ it gets to a special point, more specifically to the uppermost vertex of the lower infinite column $c_m$. Immediately afterwards, it follows $L_W$ and it reaches the destination point. The conclusion follows.
\end{proof}
In \textbf{Part II} we see an example on how we divide all the sets of mazes $F_i$ in two classes in such a way that our algorithm $RLW$ depends qualitatively only on the class. This is why we treat each class in a separate case. In \textbf{Part III} the principle is the same, but we need to consider many more cases and write a different algorithm for each one of them.

At the end of \textbf{Part II}, we note that although we used in this part the fact there are no infinite columns in the obstacle strip, a variation of $RLW$ can be used to position the robot in the west strip on the $x$-axis, even if we drop this assumption. This note is important, because it shows that \textbf{Part II} can also be generalised to improve Theorem~\ref{ch2th2} by dropping the consecutive column condition for the finite number of VNEs. To present this variation, assume that infinite columns are allowed in the obstacle strip, i.e. the (finitely many) VNEs need not be in consecutive columns.

We begin with the remark that \textbf{Case (1)} considered above can be treated in the exact same way with or without infinite columns in the obstacle strip, so we may assume without loss of generality that \textbf{Case (2)} holds, i.e. that every maze in the class of mazes we consider contain a pass $\pi$ between a lower infinite column and the east strip. We recall that by the $\overline{()}$ transformation we apply on mazes, there are always passes between any two consecutive infinite columns in the obstacle strip. We now need to consider $2$ cases: (i) there exist passes between all consecutive lower infinite columns and between consecutive lower infinite columns and infinite columns; this case can be treated similarly with \textbf{Case (2)} above; (ii) there exist two entities, one of which is a lower infinite column and the other is either a lower infinite column or an infinite column with no pass between them. In this case we define the algorithm $RLW'$ which generalises $RLW$ as described above,
\[
RLW' := S^{\lambda+p} \text{ } OMW(2\lambda+2p, \lambda+p, 2\lambda+4p) \text{ } N^{2 \lambda+4p} \text{ } L.
\]

The reason why $RLW'$ indeed generalises $RLW$ in this case is that after the robot follows $X \text{ } S^{\lambda+p} \text{ } OMW(2\lambda+2p, \lambda+p, 2\lambda+4p)$, it remains trapped in the lower infinite column or infinite column in the obstacle strip with largest longitude $m$ which is not connected with the lower infinite column or infinite column at longitude $m-1$. The robot's latitude is $2\lambda+4p$ below the lowermost VNE at longitude between $m$ and $b$. Hence from this starting position, after the robot follows $N^{2 \lambda+4p}$ it gets to a special vertex (by definition) and therefore, $X \text{ } RLW'$ takes the robot to the destination point. 

Moving on from this digression, by \textbf{Case (1)} and \textbf{Case (2)}, we may assume that we are given $F_i$ and a finite algorithm $X$ with $\lambda=|X|$ such that after the robot follows $X$ it is either in the west strip on the $x$-axis or it has visited the destination point. Without loss of generality, we assume that the robot is in the west strip on the $x$-axis and our aim is to build a finite algorithm $F$ such that $X F$ solves $F_i$.

\subsection*{Part III} The algorithm $finish$ defined in this part aims to make the robot visit the destination point. \\
\textbf{Case (1). } We assume that the destination point is in an infinite column in the west strip. We define the algorithm:
 \[
 F = MW(p, 2p) \text{ } OME (\lambda+\mu, \lambda+\mu, p),
 \]
where $\mu = |MW(p, 2p)|$.

\textbf{Claim. } For any such maze in $F_i$, after the robot follows $X \text{ } F$, it visits the destination point. 
\begin{proof}
After the robot follows $X \text{ } MW(p, 2p)$ it is in the west strip, to the west of the origin or it has already visited the destination point. By the choice of parameters and by the consequence of Lemma~\ref{l2} when applied to the particular algorithm $OME$, after it follows $X \text{ } F$, the robot visits the destination point. 
\end{proof}

At the end of \textbf{Case (1)} we note that we note that although we used the fact there are no infinite columns in the obstacle strip, a variation of $F$ can be used in order to make the robot visit the destination point, even if we drop this assumption. To present this variation, we assume that infinite columns are allowed in the obstacle strip, i.e. the (finitely many) VNEs need not be in consecutive columns. Let us assume for now that the destination point is in an infinite column in the east strip or obstacle strip.

In this case, given any finite algorithm $A$ we will construct a finite algorithm $U(A)$ with the following $2$ properties: if the robot starts in the origin of any maze in $F_i$, it follows $A$ and it gets to the west of the destination point then (1) after the robot follows $U(A)$, either its latitude strictly increases or the robot remains stuck in a finite column, or upper/lower infinite column at some longitude $i$ with no HE connecting that column to points at longitude $i+1$; (2) as the robot follows $U(A)$, if the robot visits the infinite column which contains the destination point, then the robot visits the destination point. We will construct our algorithm $U$ from bricks of the form
\[
B(k, A) = N^{|A|+2p} S^{2|A|+4p} N^{|A|+2p} N^k E S^k
\]
where $k$ is an integer and $A$ is a finite algorithm. Every time we insert a brick $B(k, A)$ as a subalgorithm of $X \text{ } F$, we take $A$ to be the entire algorithm written until that instance of $B(k, A)$. Hence, every brick depends on the length of the algorithm written up to it in $X \text{ } F$. With this convention, from now on we shall drop the second argument from the definition of a brick and let $B(k) = B(k, A)$. We note in advance that the aim of the first segment $N^{|A|+2p} S^{2|A|+4p} N^{|A|+2p}$ of a brick $B(k, A)$ is the following: for any maze in $F_i$, if the robot is in the same column as the destination point after it follows a finite algorithm $A$, if the robot then follows $N^{|A|+2p} S^{2|A|+4p} N^{|A|+2p}$, it visits the destination point. Hence we regard the first segment $N^{|A|+2p} S^{2|A|+4p} N^{|A|+2p}$ of a brick just as an oscillation large enough to make the robot visit the destination point after it reaches the right longitude.

We note that for any $k$, if the robot follows $B(k)$ starting in any point of any maze, its longitude does not decrease. We define four types of steps by concatenating bricks, so each of the steps also have this property.

The first step $U_1$ is designed to have the following property: if the robot starts in the origin of any maze in $F_i$, it follows a finite algorithm $A$ and it gets to an infinite column strictly at the west of the destination point, if the robot then follows $U_1$, its longitude strictly increases. For instance we can take
\[
U_1 (A) := B(-|A|-p) \text{ } B(-|A|-p+1) \hdots B(|A|+p),
\]
where $A$ is always taken to be the entire algorithm written before the occurrence of this step. With this convention, we drop the argument $A$ and $U_1$ has the desired property (cf. steps $3$ and $4$ below). We also note that formally, the first brick in $U_1(A)$ is $B(-|A|-p, A)$, the second brick is $B(-|A|-p+1, A \text{ } B(-|A|-p, A))$, etc. 

The second step $U_2$ is designed to have the following property: if the robot starts in the origin of any maze in $F_i$, it follows a finite algorithm $A$ and it gets to a finite column at longitude $i$ which is connected to any point at longitude $i+1$ by a HE (i.e. no HEs emerging in the east part of the finite column), if the robot then follows $U_2$, its longitude strictly increases. For instance we can take
\[
U_2 (A) := B(0) \text{ } B(1) \hdots B(2p),
\]
where the definition of $U_2(A)$ does not depend on $A$, as every finite column has at most $2p$ vertices. Therefore, let $U_2 = U_2 (A)$ have the desired property (cf. steps $3$ and $4$ below).

The third and forth step $U_3$ and $U_4$ are designed to have the following property: if the robot starts in the origin of any maze in $F_i$, it follows a finite algorithm $A$ and it gets to an upper infinite or lower infinite column at longitude $i$ which is connected to any point at longitude $i+1$ by a HE, if the robot then follows $U_3$ or $U_4$, respectively, its longitude strictly increases. For instance we can take both $U_3$ and $U_4$ to be a concatenation of $2p$ bricks in the following way
\[
U_3 (A) = B(-|A|-2p)\text{ } B(-|A|-|B(-|A|-2p)|-2p+1) \text{ } 
\]
\[
B(-|A|-|B(-|A|-2p)|-|B(-|A|-|B(-|A|-2p)|-2p+1)|-2p+2) \hdots ,
\]
\[
U_4 (A) = B(|A|+2p)\text{ } B(|A|+|B(|A|+2p)|+2p-1) \text{ } 
\]
\[
B(|A|+|B(|A|+2p)|+|B(|A|+|B(|A|+2p)|+2p-1)|+2p-2) \hdots ,
\]
where $A$ is always taken to be the entire algorithm written before the occurrence of this step. With this convention, we drop the argument $A$ and $U_3$, $U_4$ have the desired property. Indeed, let's assume that the robot is in an upper infinite column $c = (i, y), (i, y+1), \hdots$ at longitude $i$ which is connected to any point at longitude $i+1$ by a HE, and it follows $U_3$. Let $j$ be the smallest non-negative integer such that the vertices $(i, y+j)$ and $(i+1, y+j)$ are connected by a HE. From the definition of passes and the primary rectangle, we first note that $-p \leq y+j \leq p$ and $j \leq 2p$. As the robot follows the first brick $B(-|A|-2p) = N^{|A|+2p} S^{2|A|+4p} N^{|A|+2p} S^{|A|+2p} E N^{|A|+2p}$ in $U_3$, it oscillates in $c$, executing an $E$ instruction at the vertex $(i, y)$ in $c$. If $j=0$ we are done; otherwise, after the robot follows $B(-|A|-2p)$, it gets at the vertex $(i, y+|A|+2p)$. Therefore, we can track the position of the robot as it follows the second brick $B(-|A|-|B(-|A|-2p)|-2p+1)$ in $U_3$, and we observe that it oscillates in $c$, executing an $E$ instruction at the vertex $(i, y+1)$ in $c$. We continue in the same way; as $-p \leq y+j \leq p$, $j \leq 2p$, we are done.

Let us make one more remark regarding these steps. If the robot follows a concatenation of bricks and it reaches a finite column, an upper infinite column or a lower infinite column at longitude $i$ with no HE connecting it to points at longitude $i+1$, the robot remains stuck in that structure while it follows the rest of the algorithm. Let us define
\[
U (A) = U_1 U_2 U_3 U_4,
\]
or formally $U(A) = U_1(A) \text{ } U_2(A \text{ } U_1 (A)) \text{ } U_3 (A \text{ } U_1(A) \text{ } U_2(A \text{ } U_1 (A))) \text{ } U_4 (\hdots)$. As usual, every time we use the algorithm $U(A)$ as a subalgorithm, we take $A$ to be the entire algorithm written before the occurrence of $U(A)$, so with this convention we drop the argument of $U$. Therefore, it is clear that the algorithm $U$ has the two promised properties at the beginning of the case: if the robot starts in the origin of any maze in $F_i$, it follows a finite algorithm $A$ and it gets to the west of the destination point then (1) after the robot follows $U$, either its latitude strictly increases or the robot remains stuck in a finite column, or upper/lower infinite column at some longitude $i$ with no HE connecting that column to points at longitude $i+1$; (2) as the robot follows $U$, if the robot visits the infinite column which contains the destination point, then the robot visits the destination point. Furthermore, let $V(A) = \underbrace{UU \hdots U}_{\lambda+p},$ or formally
\[
V(A) = \underbrace{U(A) \text{ } U(A \text{ } U(A)) \text{ } U(A \text{ } U(A) \text{ } U(A \text{ } U(A))) \hdots}_{\lambda+p \text{ terms}}.
\]
We finally define the algorithm
\[
F = V(X) \text{ } N^{|V(X)|+l+p} \text{ } L \text{ } S^{2|V(X)|+l+2p} L.
\]
Let us see that indeed, after the robot follows $X \text{ } F$, it visits the destination point. We may assume without loss of generality that after the robot follows $X$ it is in the west strip on the $X$ axis, clearly to the west of the destination point (which is assumed to be in the east strip or in the obstacle strip). From property (1) of $U$, after the robot follows $X \text{ } V(X)$, it either visits the destination point or it remains stuck in a finite, lower or upper infinite column at longitude $i$, at the west of the destination point, with no HE connecting it to points at longitude $i+1$. In the first two cases, it is clear that after the robot follows $X \text V(X) \text{ } N^{|V(X)|+l+p} \text{ } L$ it visits the destination point. In the third case, the robot is stuck in an upper infinite column $c$ after it follows $X \text V(X) \text{ } N^{|V(X)|+l+p}$. We claim that the robot returns to the same vertex in $c$ after it follows $X \text V(X) \text{ } N^{|V(X)|+l+p} \text{ } L$. Indeed, as the robot follows $L$ its latitude is too large to meet any VNE, so by the construction of $L$ the latitude of the robot does not change after it follows $L$. Moreover, the fact that $c$ has no HE connecting it to points at longitude $i+1$ makes the robot return at longitude $i$ after it follows $L$. After that, the robot follows $S^{2|V(X)|+l+2p}$ and it gets to a special point, and then it follows $L$ which further takes it to the destination point.

Moving on from this digression, we have solved \textbf{Case (1)} in which the destination point is in an infinite column in the west strip. By the symmetry of this case and \textbf{Part II}, we similarly solve the case when the destination point is in an infinite column in the east strip. Likewise, the generalisation of \textbf{Case (1)} proved at the end of the section generalises to the case when the destination point is in an infinite column in the west strip. In fact, in the generalisation of \textbf{Case (1)} we could have only considered the case when the destination point is in an infinite column in the obstacle strip, as \textbf{Case (1)} itself works just as well even in the generalised set up, when the destination point is either in the west strip or east strip. That would not have simplified the argument, though.
\noindent

\textbf{Case (2). } We assume that the destination point is in the obstacle strip in a finite column, upper infinite column or lower infinite column and it is connected to the west strip via a path through a (finite) sequence of finite columns. Let $R_1, R_2, \hdots , R_k$ be a sequence of finite columns and $R$ be a finite column, upper infinite column or lower infinite column such that $R$ contains the destination point and there exists a HE between the west strip and $R_1$, between $R_m$ and $R_{m+1}$ for $1 \leq m \leq k$, where by convention $R_{k+1}=R$. Let $\textbf{w} = (a+1, u)$ be the uppermost point of the finite column $R_1$.

We consider the following sub-cases: \\
\begin{figure}[h!]
\centering
\resizebox{0.75\textwidth}{!}{
\begin{tikzpicture}
    [
        dot/.style={circle,draw=black, fill,inner sep=2pt},
        cross/.style={cross out, draw=black, minimum size=2*(#1-\pgflinewidth), inner sep=4pt, outer sep=4pt},
    ]

\foreach \x in {-6,...,6}{
 \foreach \y in {-6,...,6}{
    \node[dot] at (\x,\y){ };
}}

\draw[line width=0.7mm, red] (-6, -6) -- (-6, 6);
\draw[line width=0.7mm, red] (-5, -6) -- (-5, 6);
\draw[line width=0.7mm, red] (-4, -6) -- (-4, 6);
\draw[line width=0.7mm, red] (-3, -6) -- (-3, 6);
\draw[line width=0.7mm, red] (-2, -6) -- (-2, 6);
\draw[line width=0.7mm, red] (-1, -6) -- (-1, 6);
\draw[line width=0.7mm, red] (0, -6) -- (0, 6);

\draw[line width=0.7mm, red] (1, 6) -- (1, 4);
\draw[line width=0.7mm, red] (1, 3) -- (1, 1);
\draw[line width=0.7mm, red] (1, 0) -- (1, -2);
\draw[line width=0.7mm, red] (1, -3) -- (1, -6);

\draw[line width=0.7mm, red] (2, 6) -- (2, 3);
\draw[line width=0.7mm, red] (2, 2) -- (2, 0);
\draw[line width=0.7mm, red] (2, -2) -- (2, -6);

\draw[line width=0.7mm, red] (3, 6) -- (3, 3);
\draw[line width=0.7mm, red] (3, 2) -- (3, 1);
\draw[line width=0.7mm, red] (3, -1) -- (3, -3);
\draw[line width=0.7mm, red] (3, -4) -- (3, -6);

\draw[line width=0.7mm, red] (4, 6) -- (4, 4);
\draw[line width=0.7mm, red] (4, 3) -- (4, 2);
\draw[line width=0.7mm, red] (4, 1) -- (4, -3);
\draw[line width=0.7mm, red] (4, -4) -- (4, -6);

\draw[line width=0.7mm, red] (5, -6) -- (5, 6);
\draw[line width=0.7mm, red] (6, -6) -- (6, 6);

\draw[line width=0.7mm, red] (-5, 5) -- (-2, 5);
\draw[line width=0.7mm, red] (1, 5) -- (3, 5);
\draw[line width=0.7mm, red] (5, 5) -- (6, 5);

\draw[line width=0.7mm, red] (-6, 4) -- (-4, 4);
\draw[line width=0.7mm, red] (-2, 4) -- (-1, 4);
\draw[line width=0.7mm, red] (3, 4) -- (5, 4);

\draw[line width=0.7mm, red] (-3, 3) -- (-2, 3);
\draw[line width=0.7mm, red] (-1, 3) -- (1, 3);

\draw[line width=0.7mm, red] (-2, 2) -- (2, 2);
\draw[line width=0.7mm, red] (-5, 2) -- (-3, 2);

\draw[line width=0.7mm, red] (2, 1) -- (4, 1);
\draw[line width=0.7mm, red] (-4, 1) -- (-2, 1);

\draw[line width=0.7mm, red] (-5, -1) -- (-4, -1);
\draw[line width=0.7mm, red] (-2, -1) -- (2, -1);

\draw[line width=0.7mm, red] (-6, -2) -- (-3, -2);
\draw[line width=0.7mm, red] (-2, -2) -- (-1, -2);
\draw[line width=0.7mm, red] (3, -2) -- (4, -2);
\draw[line width=0.7mm, red] (5, -2) -- (6, -2);

\draw[line width=0.7mm, red] (-1, -3) -- (0, -3);

\draw[line width=0.7mm, red] (-4, -4) -- (-2, -4);

\draw[line width=0.7mm, red] (-6, -5) -- (-5, -5);
\draw[line width=0.7mm, red] (-2, -5) -- (-1, -5);
\draw[line width=0.7mm, red] (0, -5) -- (1, -5);
\draw[line width=0.7mm, red] (4, -5) -- (6, -5);

\draw[line width=0.7mm, red] (-6, -6) -- (-5, -6);
\draw[line width=0.7mm, red] (1, -6) -- (4, -6);

\draw[line width=1mm, green] (-2, 2) -- (2, 2);
\draw[line width=1mm, green] (2, 2) -- (2, 1);
\draw[line width=1mm, green] (2, 1) -- (4, 1);
\draw[line width=1mm, green] (4, 1) -- (4, -2);
\draw[line width=1mm, green] (4, -2) -- (3, -2);
\draw[line width=1mm, green] (3, -1) -- (3, -3);

\draw node[fill,circle,inner sep=5.5pt,minimum size=2pt] at (-2, 2) {};

\draw (-1.7, 2.3) -- node {\textbf{v}} (-1.7, 2.3);

\draw (3,-3) node[cross] {};

\foreach \x in {-6,...,6}
    \draw (\x,.1) -- node[below,yshift=-1mm] {\x} (\x,-.1);
\foreach \y in {-6,...,6}
    \draw (.1,\y) -- node[below,xshift=-3mm, yshift=3mm] {\y} (-.1,\y);
\draw[->,line width=0.15mm] (0,-6.5) -- (0,6.5);
\draw[->,line width=0.15mm] (-6.5,0) -- (6.5,0);
\end{tikzpicture}
}
\caption{\textbf{Part III, Case (2)(i).} We assume that there exists a row $r_i$ that intersects the finite column $R_1$ and that $r_i$ has a west bump. We assume that there are no VEs removed other than the ones shown in the figure. In this example we further take $R_1 = (1, 1), (1, 2), (1, 3)$, $R_2 = (2, 0), (2, 1), (2, 2)$, $R_3 = (3, 1), (3, 2)$, $R_4 = (4, -3), ... (4, 1)$, $R_5=R=(3, -1), (3, -2), (3, -3)$ and $r_i = r_2$ is the row that intersects $R_1$ with its west bump $(-3, 2), (-2, 2)$ and $\textbf{v}=(-2, 2)$. In general, we do not require $R$ to be a finite column. Let $H' = EEEENN^{-1}EEN^3N^{-3}E^{-1}N^2N^{-2}E$ and note that if the robot starts at $\textbf{v}$ and follows $H'$ it follows the green path and gets to the destination point. However, if the robot starts at the west of $\textbf{v}$ on $r_i$ and it follows $H'$, its longitude is always strictly smaller than that of $\textbf{v}$. Therefore, it does not hit any VNE and as $|H'|_N=|H'|_S$ its latitude does not change. In our case, $H = H' E^{20}$ which has the extra property that after the robot follows $H$ in a maze with no VNEs, its longitude does not decrease - in fact, it can be proven that there always exits a certain $H'$ that has this property itself, but we do not wish to complicate the argument.} \label{6}
\end{figure}
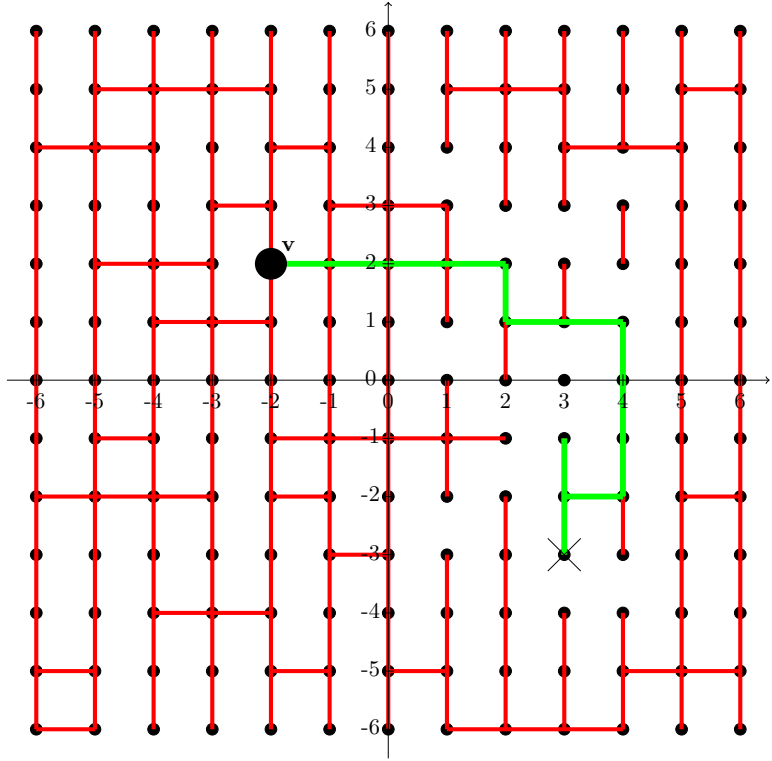
\textbf{2(i)} We assume that there exists a row $r_i$ that intersects the finite column $R_1$ and that it has a west bump. We recall that the west bumps are the easternmost HNEs with at least one vertex in the west strip on a row that intersects some finite column. Assume first that the eastern vertex $\textbf{v}$ of that west bump is in the west strip. By inspecting the longitude of the west bump and the primary rectangle we can construct an algorithm of the form $H' := \prod_{m=1}^h N^{k_m} N^{-k_m} E^{\varepsilon_m}$, where $\varepsilon_m \in \{ -1,1 \}$ and $k_m$ is an integer for all $1 \leq m \leq h$, such that if the robot starts at $v$ and follows $H'$, it visits the destination point. Indeed if the robot is at some specified latitude in the finite column $R_m$ and follows $N^{k_m} N^{-k_m}E$ for suitable $k_m$, it gets to some specified latitude in the finite column $R_{m+1}$. Let $H = H' E^{|H'|}$. We define the algorithm
\[
F = N^i \text{ } W^q \text{ } SME (\lambda+q, \lambda+q, H).
\]
\begin{prop}\label{plm77}
For any maze in $F_i$, after the robot follows $X \text{ } F$, it visits the destination point.
\end{prop}
\begin{proof}
We may assume without loss of generality that after the robot follows $X \text{ } N^i$, it is on the row $r_i$. Hence after the robot follows $X \text{ } N^i \text{ } W^q$ it is on the row $r_i$ at a longitude at most that of $\textbf{v}$. By the choice of parameters and by Lemma~\ref{l2}, while the robot follows $SME (\lambda+q, \lambda+q, H)$ it advances eastwards in the west strip oscillating about row $r_i$ and passing through the smallest HE with respect to the well order on $\Z$: $i<1+i<-1+i<2+i<\hdots$. Considering that $|H|_N=|H|_S$, while the robot is in the west strip, after it follows $H$ its latitude does not change and its longitude does not decrease. It eventually arrives at the point $\textbf{v}$ on $r_i$ not while executing $H$ (from the form of $H$ and the shape of the maze which has a HNE with its eastern vertex at $v$). Immediately after the robot reaches $\textbf{v}$, it follows $H$ and it gets to the destination point (see Figure~\ref{6}).
\end{proof}

In the case that there exists a west bump positioned at the border between the obstacle strip and the west strip on a row $r_i$ that intersects $R_1$, we consider $r_j$ to be a row on which there exists a HE between the west strip and $R_1$. As before, let $\textbf{v}$ be the eastern vertex of the west bump, $\textbf{v} \in \R_1$. We recall the algorithm
\[
SME^{(j-i)} (a, e, H):= (((((((H)^e N^{j-i} E S^{j-i})^e E)^e NES)^e SEN)^e NNESS)^e \hdots S^aEN^a)^e
\]
introduced in Corollary~\ref{c1}. We further define the algorithm
\[
F = N^i \text{ } W^q \text{ } SME^{(j-i)} (\lambda+q, \lambda+q, H).
\]
\textbf{Claim. } For any maze in $F_i$ after the robot follows $X \text{ } F$, it visits the destination point.
\begin{proof}
The conclusion follows by the same reasoning as in the proof of Proposition~\ref{plm77} and by Corollary~\ref{c1}.
\end{proof}
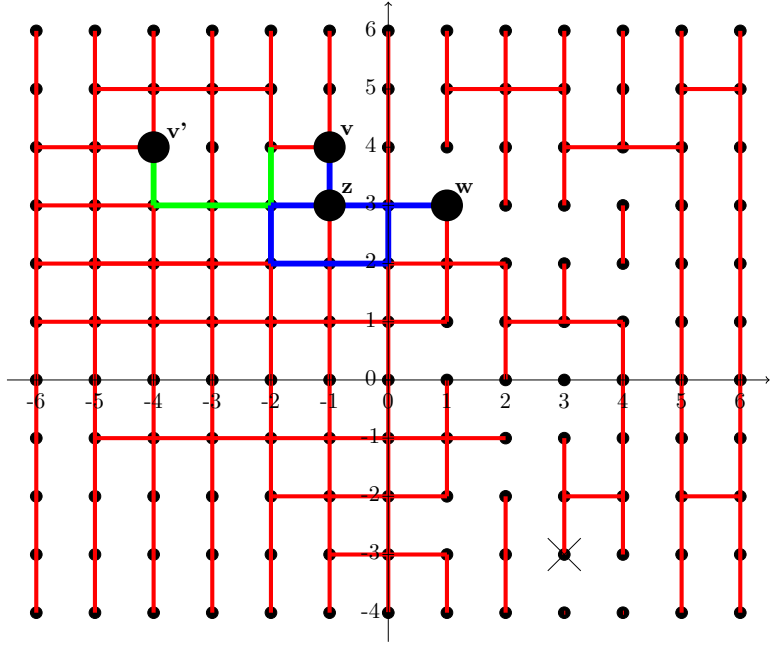
\begin{figure}[h!]
\centering
\resizebox{0.75\textwidth}{!}{
\begin{tikzpicture}
    [
        dot/.style={circle,draw=black, fill,inner sep=2pt},
        cross/.style={cross out, draw=black, minimum size=2*(#1-\pgflinewidth), inner sep=4pt, outer sep=4pt},
    ]

\foreach \x in {-6,...,6}{
 \foreach \y in {-4,...,6}{
    \node[dot] at (\x,\y){ };
}}

\draw[line width=0.7mm, red] (-6, -4) -- (-6, 6);
\draw[line width=0.7mm, red] (-5, -4) -- (-5, 6);
\draw[line width=0.7mm, red] (-4, -4) -- (-4, 6);
\draw[line width=0.7mm, red] (-3, -4) -- (-3, 6);
\draw[line width=0.7mm, red] (-2, -4) -- (-2, 6);
\draw[line width=0.7mm, red] (-1, -4) -- (-1, 6);
\draw[line width=0.7mm, red] (0, -4) -- (0, 6);

\draw[line width=0.7mm, red] (1, 6) -- (1, 4);
\draw[line width=0.7mm, red] (1, 3) -- (1, 1);
\draw[line width=0.7mm, red] (1, 0) -- (1, -2);
\draw[line width=0.7mm, red] (1, -3) -- (1, -4);

\draw[line width=0.7mm, red] (2, 6) -- (2, 3);
\draw[line width=0.7mm, red] (2, 2) -- (2, 0);
\draw[line width=0.7mm, red] (2, -2) -- (2, -4);

\draw[line width=0.7mm, red] (3, 6) -- (3, 3);
\draw[line width=0.7mm, red] (3, 2) -- (3, 1);
\draw[line width=0.7mm, red] (3, -1) -- (3, -3);
\draw[line width=0.7mm, red] (3, -4) -- (3, -4);

\draw[line width=0.7mm, red] (4, 6) -- (4, 4);
\draw[line width=0.7mm, red] (4, 3) -- (4, 2);
\draw[line width=0.7mm, red] (4, 1) -- (4, -3);
\draw[line width=0.7mm, red] (4, -4) -- (4, -4);

\draw[line width=0.7mm, red] (5, -4) -- (5, 6);
\draw[line width=0.7mm, red] (6, -4) -- (6, 6);

\draw[line width=0.7mm, red] (-5, 5) -- (-2, 5);
\draw[line width=0.7mm, red] (1, 5) -- (3, 5);
\draw[line width=0.7mm, red] (5, 5) -- (6, 5);

\draw[line width=0.7mm, red] (-6, 4) -- (-4, 4);
\draw[line width=0.7mm, red] (-2, 4) -- (-1, 4);
\draw[line width=0.7mm, red] (3, 4) -- (5, 4);

\draw[line width=0.7mm, red] (-6, 3) -- (1, 3);

\draw[line width=0.7mm, red] (-6, 2) -- (2, 2);
\draw[line width=0.7mm, red] (-5, 2) -- (-3, 2);

\draw[line width=0.7mm, red] (2, 1) -- (4, 1);
\draw[line width=0.7mm, red] (-6, 1) -- (1, 1);

\draw[line width=0.7mm, red] (-5, -1) -- (2, -1);

\draw[line width=0.7mm, red] (-2, -2) -- (1, -2);
\draw[line width=0.7mm, red] (3, -2) -- (4, -2);
\draw[line width=0.7mm, red] (5, -2) -- (6, -2);

\draw[line width=0.7mm, red] (-1, -3) -- (1, -3);

\draw[line width=1mm, green] (-4, 4) -- (-4, 3);
\draw[line width=1mm, green] (-4, 3) -- (-2, 3);
\draw[line width=1mm, green] (-2, 3) -- (-2, 4);

\draw[line width=1mm, blue] (-1, 4) -- (-1, 3);
\draw[line width=1mm, blue] (-1, 3) -- (1, 3);
\draw[line width=1mm, blue] (0, 3) -- (0, 2);
\draw[line width=1mm, blue] (0, 2) -- (-2, 2);
\draw[line width=1mm, blue] (-2, 2) -- (-2, 3);
\draw[line width=1mm, blue] (-2, 3) -- (-1, 3);

\draw node[fill,circle,inner sep=5.5pt,minimum size=2pt] at (-1, 4) {};
\draw node[fill,circle,inner sep=5.5pt,minimum size=2pt] at (-4, 4) {};
\draw node[fill,circle,inner sep=5.5pt,minimum size=2pt] at (1, 3) {};
\draw node[fill,circle,inner sep=5.5pt,minimum size=2pt] at (-1, 3) {};

\draw (-0.7, 4.3) -- node {\textbf{v}} (-0.7, 4.3);
\draw (-3.6, 4.3) -- node {\textbf{v'}} (-3.6, 4.3);
\draw (1.3, 3.3) -- node {\textbf{w}} (1.3, 3.3);
\draw (-0.7, 3.3) -- node {\textbf{z}} (-0.7, 3.3);

\draw (3,-3) node[cross] {};

\foreach \x in {-6,...,6}
    \draw (\x,.1) -- node[below,yshift=-1mm] {\x} (\x,-.1);
\foreach \y in {-4,...,6}
    \draw (.1,\y) -- node[below,xshift=-3mm, yshift=3mm] {\y} (-.1,\y);
\draw[->,line width=0.15mm] (0,-4.5) -- (0,6.5);
\draw[->,line width=0.15mm] (-6.5,0) -- (6.5,0);
\end{tikzpicture}
}
\caption{\textbf{Part III, Case (2)(ii).} Every row that intersects $R_1$ does not have a west bump and there exists a special west pipe on some row $r_j$. We assume that there are no VEs removed other than the ones shown in the figure. We have $r_j = r_4$ and so the special west pipe is $(-2, 4), (-1, 4), (0, 4)$ with $\textbf{v} = (-1, 4)$ and $\textbf{w} = (1, 3)$. Let us observe that after the robot follows $WPF(1, 1000)$ starting at $(-6, 4)$ it gets to $\textbf{v'} = (-4, 4)$ which is the middle vertex of the ``fake'' west pipe $(-5, 4), (-4, 4), (-3, 4)$. Further note that after the robot follows $WPF(1, 1000)$ starting at $(-3, 4)$ it gets to $\textbf{v}$. For this example we have $K=SE^2NWSW^2NE$ and after the robot follows $K$ starting from $\textbf{v}$ it gets to $\textbf{z} = (-1, 3)$, which is indeed on $r_u=r_3$ (see the blue walk). In addition, note that if the robot starts from $\textbf{v'}$ and follows $K$ it gets back to $\textbf{v'}$ (see the green circuit). We can take $H' = EENSENSEEN^3S^3E^{-1}N^2S^2E$ which has the required form and the property that after the robot starts from $\textbf{z} = (-1, 3)$ and follows $H'$ it visits the destination point. The reader may assume that the robot starts at $(-4, 0)$ and it follows $F = N^3 W^2 N (WPF(1, 1000) \text{ } K \text{ } H \text{ } SEN)^{10}$ to see how the algorithm $F$ solves the maze: after it follows $N^3 W^2 N$, the robot gets to $(-6, 4)$; further, after the first iteration of $WPF(1, 1000) \text{ } K \text{ } H \text{ } SEN$, it gets to $(-3, 4)$ as $K$ and $H$ do not change the position of the robot while it is strictly at the west of $\textbf{v}$; after the second iteration of $WPF(1, 1000) \text{ } K \text{ } H \text{ } SEN$, the robot visits the destination point.} \label{7}
\end{figure}
\noindent
\textbf{2(ii)} We assume that the previous case does not hold, so every row that intersects the column $R_1$ does not have a west bump, i.e. each such row is a path in the west strip. In addition, we assume there exists a special west pipe on some row $r_j$. We recall that the west pipes are easternmost configurations in the west strip formed by a HE followed by a HNE. Denote by $\textbf{v}$ the easternmost vertex of the HE of the special west pipe. Assume without loss of generality that $j>u$, where $\textbf{w} = (a+1, u)$ is the uppermost point of the finite column $R_1$.

We start by defining a new algorithm called $west \_ pipe \_ finder$:
\[
WPF(a,e):=(E^eWS^aEN^a)^e,
\]
with its counterpart $east \_ pipe \_ finder$. This is used directly in the final algorithm $F$ and it will be analysed later (see Figure~\ref{7}).

We then define the algorithm $$K=S^{j-u}E^dN^{j-u}WS^{j-u}W^dN^{j-u}E,$$ where $d$ is the difference in longitude between $c_{a+1}$ and $\textbf{v}$. 

\textbf{Claim. } For any maze in $F_i$, if the robot starts at $\textbf{v}$ and follows $K$ it gets at a certain known point $\textbf{z}$ (given the tertiary rectangle) on the row $r_u$.
\begin{proof}
Starting at $\textbf{v}$, after the robot follows $S^{j-u}$ it gets on the row $r_u$; after it follows $S^{j-u}E^d$ it gets to $\textbf{w}$; after it follows $S^{j-u}E^dN^{j-u}$ it remains fixed at $\textbf{w}$; after it follows $S^{j-u}E^dN^{j-u}W$ it gets to $(a, u)$, on the row $r_u$ to the west of $\textbf{w}$; finally, while it executes $S^{j-u}W^dN^{j-u}E$ starting at $(a, u)$ it does not leave the square $\{ (\pm q, \: \pm q) \}$; while it executes both the subalgorithms $S^{j-u}$ and $N^{j-u}$ of $S^{j-u}W^dN^{j-u}E$ it does not hit any VNE (see Figure~\ref{7}). The conclusion follows.
\end{proof}
\noindent
\textbf{Remark. } It is easy to check that if the robot starts from the easternmost vertex $\textbf{v'}$ of a HE followed by a HNE on $r_j$ with $\textbf{v'}$ strictly at the west of $\textbf{v}$ and it follows $K$, then the robot remains in the west strip while following $K$ and after it follows $K$, it returns back to the starting point $\textbf{v'}$ (see Figure~\ref{7}). The algorithm $K$ was constructed specifically to have this property, together with the one proved in the Claim above.

By inspecting the tertiary rectangle, we construct an algorithm $H'$ of the form $H'=\prod_{i=1}^hN^{k_i}N^{-k_i}E^{\epsilon_i}$, where $\epsilon_i \in \{-1,1\}$ and $k_i$ is an integer for all $1 \le i \le h$, such that if the robot starts at $\textbf{z}$ and it follows $H'$ it visits the destination point. Let $H=H'E^{|H'|}$. We observe that if the robot starts from the easternmost vertex $\textbf{v'}$ of a HE followed by a HNE on the row $r_j$ in the west strip and follows $H$, it remains in the west strip and it returns to the same point $\textbf{v'}$. We finally define the algorithm:
\[
F = N^u \text{ } W^q \text{ } N^{j-u} \text{ } (WPF(j-u,\lambda+q) \text{ } K \text{ } H \text{ } S^{j-u} \text{ } E \text{ } N^{j-u})^{\lambda+q}.
\]
\begin{prop}
For any maze in $F_i$, after the robot follows $X \text{ } F$, it visits the destination point.
\end{prop}
\begin{proof}
We may assume without loss of generality that after the robot follows $X \text{ } N^u \text{ } W^q \text{ } N^{j-u}$ it gets on the row $r_j$, to the west of the point $\textbf{v}$. While the robot is at the west of the point $\textbf{v}$ on the row $r_j$, after each instance of $WPF (j-u,\lambda+q)$ it advances eastwards to the easternmost vertex $\textbf{v'}$ of a HE followed by a HNE on the row $r_j$. While $\textbf{v'}$ is strictly at the west of $\textbf{v}$, the robot follows the algorithm $K \text{ } H$ and returns back to $\textbf{v'}$; while the robot follows the algorithm $S^{j-u}EN^{j-u}$ it advances one unit to the east of $\textbf{v'}$ on the row $r_j$. By the choice of parameters, the robot eventually arrives at $\textbf{v'}=\textbf{v}$. Immediately afterwards, it follows $K \text{ } H$ and it visits the destination point (see Figure~\ref{7}).
\end{proof}
\noindent
\textbf{2(iii)} We assume there exists a magical west row $r_j$. We recall that a magical west row is a row which is a path when restricted to the west strip, and it contains a HNE; its west cutoff is its westernmost HNE. Denote by $\textbf{v}$ the westernmost vertex of the west cutoff of $r_j$. Then, by inspecting the tertiary rectangle, we can construct an algorithm $K$ such that if the robot starts from $\textbf{v}$ and follows $K$ it gets to the destination point. We define the algorithm
\[
F = N^jE^{\lambda+q}K.
\]
\textbf{Claim. } For any maze in $F_i$, after the robot follows $X \text{ } F$, it visits the destination point.
\begin{proof}
We may assume without loss of generality that after the robot follows $X \text{ } N^j$, it gets on the row $r_j$. Therefore, after it follows $X \text{ } N^j \text{ } E^{\lambda+q}$ the robot gets to the point $\textbf{v}$. Hence, after the robot follows $X \textbf{ } F$ it gets to the destination point.
\end{proof}
\begin{figure}[h!]
\centering
\resizebox{0.75\textwidth}{!}{
\begin{tikzpicture}
    [
        dot/.style={circle,draw=black, fill,inner sep=2pt},
        cross/.style={cross out, draw=black, minimum size=2*(#1-\pgflinewidth), inner sep=4pt, outer sep=4pt},
    ]

\foreach \x in {-6,...,6}{
 \foreach \y in {-4,...,6}{
    \node[dot] at (\x,\y){ };
}}

\draw[line width=0.7mm, red] (-6, -4) -- (-6, 6);
\draw[line width=0.7mm, red] (-5, -4) -- (-5, 6);
\draw[line width=0.7mm, red] (-4, -4) -- (-4, 6);
\draw[line width=0.7mm, red] (-3, -4) -- (-3, 6);
\draw[line width=0.7mm, red] (-2, -4) -- (-2, 6);
\draw[line width=0.7mm, red] (-1, -4) -- (-1, 6);
\draw[line width=0.7mm, red] (0, -4) -- (0, 6);

\draw[line width=0.7mm, red] (1, 6) -- (1, 4);
\draw[line width=0.7mm, red] (1, 3) -- (1, 1);
\draw[line width=0.7mm, red] (1, 0) -- (1, -2);
\draw[line width=0.7mm, red] (1, -3) -- (1, -4);

\draw[line width=0.7mm, red] (2, 6) -- (2, 3);
\draw[line width=0.7mm, red] (2, 2) -- (2, 0);
\draw[line width=0.7mm, red] (2, -2) -- (2, -4);

\draw[line width=0.7mm, red] (3, 6) -- (3, 3);
\draw[line width=0.7mm, red] (3, 2) -- (3, 1);
\draw[line width=0.7mm, red] (3, -1) -- (3, -3);
\draw[line width=0.7mm, red] (3, -4) -- (3, -4);

\draw[line width=0.7mm, red] (4, 6) -- (4, 4);
\draw[line width=0.7mm, red] (4, 3) -- (4, 2);
\draw[line width=0.7mm, red] (4, 1) -- (4, -3);
\draw[line width=0.7mm, red] (4, -4) -- (4, -4);

\draw[line width=0.7mm, red] (5, -4) -- (5, 6);
\draw[line width=0.7mm, red] (6, -4) -- (6, 6);

\draw[line width=0.7mm, red] (-5, 5) -- (-2, 5);
\draw[line width=0.7mm, red] (1, 5) -- (3, 5);
\draw[line width=0.7mm, red] (5, 5) -- (6, 5);

\draw[line width=0.7mm, red] (-6, 4) -- (-4, 4);
\draw[line width=0.7mm, red] (-2, 4) -- (-1, 4);
\draw[line width=0.7mm, red] (3, 4) -- (5, 4);

\draw[line width=0.7mm, red] (-6, 3) -- (1, 3);

\draw[line width=0.7mm, red] (-6, 2) -- (2, 2);
\draw[line width=0.7mm, red] (-5, 2) -- (-3, 2);

\draw[line width=0.7mm, red] (2, 1) -- (4, 1);
\draw[line width=0.7mm, red] (-6, 1) -- (1, 1);

\draw[line width=0.7mm, red] (-3, -1) -- (2, -1);

\draw[line width=0.7mm, red] (-2, -2) -- (1, -2);
\draw[line width=0.7mm, red] (3, -2) -- (4, -2);
\draw[line width=0.7mm, red] (5, -2) -- (6, -2);

\draw[line width=0.7mm, red] (-1, -3) -- (1, -3);

\draw[line width=1mm, red] (-4, 4) -- (-4, 3);
\draw[line width=1mm, red] (-4, 3) -- (-2, 3);
\draw[line width=1mm, red] (-2, 3) -- (-2, 4);

\draw[line width=1mm, red] (-1, 4) -- (-1, 3);
\draw[line width=1mm, red] (-1, 3) -- (1, 3);
\draw[line width=1mm, red] (0, 3) -- (0, 2);
\draw[line width=1mm, red] (0, 2) -- (-2, 2);
\draw[line width=1mm, red] (-2, 2) -- (-2, 3);
\draw[line width=1mm, red] (-2, 3) -- (-1, 3);

\draw[line width=1mm, blue] (-5, -1) -- (-5, 3);
\draw[line width=1mm, blue] (-5, 3) -- (-4, 3);
\draw[line width=1mm, blue] (-4, -1) -- (-4, 3);

\draw[line width=1mm, green] (-3, -1) -- (-3, 3);
\draw[line width=1mm, green] (-3, 3) -- (-2, 3);
\draw[line width=1mm, green] (-3, -1) -- (-2, -1);
\draw[line width=1mm, green] (-2, -1) -- (-2, 3);

\draw node[fill,circle,inner sep=5.5pt,minimum size=2pt] at (-3, -1) {};

\draw (-2.7, -0.7) -- node {\textbf{v}} (-2.7, -0.7);

\draw (3,-3) node[cross] {};

\foreach \x in {-6,...,6}
    \draw (\x,.1) -- node[below,yshift=-1mm] {\x} (\x,-.1);
\foreach \y in {-4,...,6}
    \draw (.1,\y) -- node[below,xshift=-3mm, yshift=3mm] {\y} (-.1,\y);
\draw[->,line width=0.15mm] (0,-4.5) -- (0,6.5);
\draw[->,line width=0.15mm] (-6.5,0) -- (6.5,0);
\end{tikzpicture}
}
\caption{\textbf{Part III, Case (2)(iv).} Every row that intersects $R_1$ does not have a west bump, i.e. all such rows are paths in the west strip and there exists a special almost empty west row $r_j$. We assume that there are no VEs removed other than the ones shown in the figure. We have $R_1 = (1, 1), (1, 2), (1, 3)$, so $r_1$, $r_2$, $r_3$ are paths in the west strip, moreover $j = -1$, so $r_{-1}$ is the special almost empty west row. Its west cutoff is the HNE $\{ (-4, -1), \textbf{v} = (-3, -1) \}$. We construct an algorithm $K$ by inspecting the tertiary rectangle such that if the robot starts from $\textbf{v}$ and follows $K$, it gets to the destination point. For example we may take $K = N^3 E^5 S E^2 S^3 W S$. We may assume that the robot starts at $(-5, 0)$ and it follows $F = N^{-1}$ $W^{100}$ $(S^{-4} E N^{-4} W)^{100}$ $K$. After the robot follows $N^{-1}$ $W^{100}$, it gets to $(-5, -1)$ on the row $r_j=r_{-1}$ at a longitude not greater than that of $\textbf{v}$. Let us see what is the position of the robot after it follows one instance of $(S^{-4} E N^{-4} W)$, starting from $r_{-1}$: while it starts strictly at the west of $\textbf{v}$, its longitude increases by $1$ (see the blue path); if it starts at $\textbf{v}$, it comes back to $\textbf{v}$ (see the green path). The exponent of $(S^{-4} E N^{-4} W)$ is large enough for the robot to reach $\textbf{v}$ after it follows $(S^{-4} E N^{-4} W)^{100}$. After that, the robot follows $K$ and it visits the destination point.} \label{8}
\end{figure}
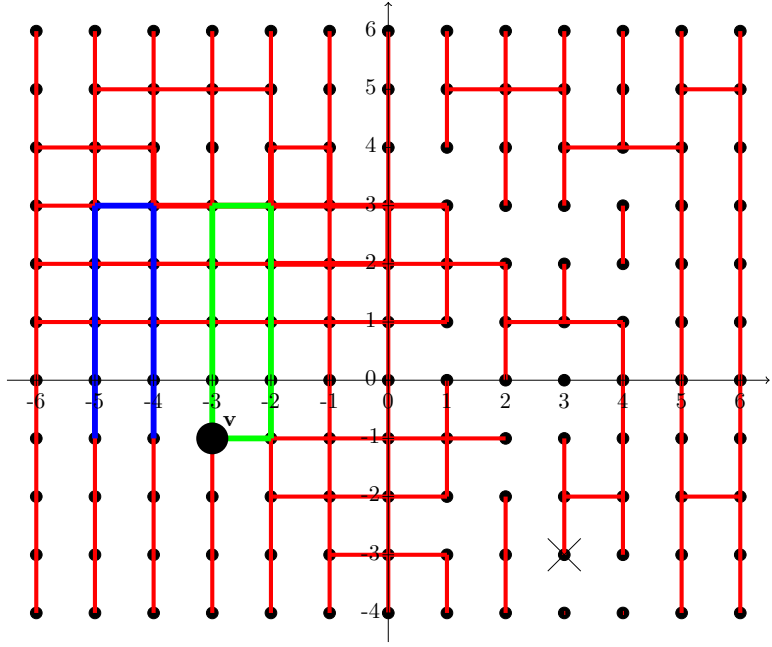
\noindent
\textbf{2(iv) } We assume that every row that intersects the finite column $R_1$ does not have a west bump and there exists a special almost empty west row $r_j$. We recall that a special almost empty west row is a row that in the west strip is the complement of an infinite path followed by a non-empty finite path; its west cutoff is its easternmost HNE in the west strip. We recall that $\textbf{w} = (a+1, u)$ is the uppermost point of $R_1$ and let $\textbf{v}$ be the easternmost vertex of the west cutoff of $r_j$. Then, by inspecting the tertiary rectangle, we can construct an algorithm $K$ such that if the robot starts from $\textbf{v}$ and follows $K$ it gets to the destination point. We define the algorithm
\[
F= N^j \text{ } W^{\lambda+q} \text{ } (S^{j-u}EN^{j-u}W)^{\lambda+q} \text{ } K.
\]
\textbf{Claim. } For any maze in $F_i$, after the robot follows $X \text{ } F$, it visits the destination point.
\begin{proof}
We may assume without loss of generality that after the robot follows $X \text{ } N^j \text{ } W^{\lambda+q}$ it gets on the row $r_j$ to the west of the point $\textbf{v}$. While the robot follows one instance of $S^{j-u}EN^{j-u}W$ it returns on the row $r_j$ and advances one unit eastwards if it is at the westernmost vertex of a HNE; it returns to the same point if it is at the westernmost vertex of a HE. By the choice of exponent, after the robot follows $N^j \text{ } W^{\lambda+q} \text{ } (S^{j-u}EN^{j-u}W)^{\lambda+q}$ it remains stuck at the point $\textbf{v}$. Immediately afterwards, it follows $K$ and it gets to the destination point (see Figure~\ref{8}).
\end{proof}
\noindent
\textbf{2(v) } We assume that every row that intersects the column $R_1$ does not have a west bump. In addition we assume that there exists a special empty west row $r_{w_3}$. We recall that an empty west row is a row that in the west strip is empty and the special empty west row is the empty west row of smallest latitude greater than $-3p$ with respect to the standard well order on $\Z$. We recall that $\textbf{w} = (a+1, u)$ is the uppermost point of the finite column $R_1$ and let $\textbf{v}$ be the easternmost vertex in the west strip on the row $r_{w_3}$. We may assume without loss of generality that $w_3>u$.

By inspecting the primary rectangle, we construct an algorithm $H'$ of the form $H' = \prod_{m=1}^hN^{k_m}N^{-k_m}E^{\epsilon_m}$, where ${\epsilon_m} \in \{-1, 1\}$ and $k_m$ is an integer with $|k_m| \le 2p$ for all $1 \le m \le h$, such that if the robot starts at $\textbf{w}$ and it follows $H'$, it visits the destination point (see $H'$ in Figure~\ref{7}). Let $H = H'W^{|H'|}$. We note that if the robot is in the origin in a maze with no VNEs and it follows $H$ it returns to the $x$-axis and its latitude does not increase. We further note that if the robot starts from $\textbf{v}$ and it follows $H$, it oscillates about latitude $w_3$ without hitting any VNE and at the end it returns back to the starting point $\textbf{v}$.

We define the algorithm
\[
F = N^{w_3} (S^{w_3-u} E N^{w_3-u} H)^{\lambda+q}.
\]
\textbf{Claim. } For any maze in $F_i$, after the robot follows $X \text{ } F$, it visits the destination point.
\begin{proof}
We may assume without loss of generality that after the robot follows $X \text{ } N^{w_3}$ it gets on the row $r_{w_3}$ at the west of the point $\textbf{v}$. While the robot follows each instance of $S^{w_3-u} E N^{w_3-u} H$ in the west strip, it advances eastwards one unit making an oscillation about the row $r_{w_3}$. By the choice of exponent, after a certain instance of $S^{w_3-u} E N^{w_3-u} H$, the robot eventually gets to the point $\textbf{v}$. Immediately afterwards, it follows another instance of $S^{w_3-u} E N^{w_3-u} H$ and it gets to the destination point. Indeed, if the robot starts at the point $\textbf{v}$ and it follows $S^{w_3-u} E N^{w_3-u}$, it gets to the point $\textbf{w}$. If the robot starts at $\textbf{w}$ and it follows $H$, it gets to the destination point. The conclusion follows.
\end{proof}
\noindent
\textbf{2(vi) } This is the final case, where we may assume all of the following: every row that intersects the column $R_1$ does not have a west bump; there does not exist a west pipe; there does not exist a magical west row; there does not exist a special almost empty west row; there does not exist a special empty west row. Then every row at latitude greater than $-3p$ with respect to the well order on $\Z$ is a path in the west strip and indeed a path in the maze; every row that intersects the finite column $R_1$ is a path in the west strip and indeed a path in the maze; each row at latitude at most $3p$ with respect to the standard well order on $\Z$ is known to be either a path or the complement of a path in the west strip. We recall that $\textbf{w} = (a+1, u)$ is the uppermost point of the finite column $R_1$.

By inspecting the primary rectangle we can construct an algorithm $H'$ of the form $H' = \prod_{m=1}^hN^{k_m}N^{-k_m}E^{\epsilon_m}$, where ${\epsilon_m} \in \{-1, 1\}$ and $k_m$ is an integer with $|k_m| \le 2p$ for all $1 \le m \le h$, such that if the robot starts at $\textbf{w}$ and follows $H'$ it visits the destination point (see $H'$ in Figure~\ref{7}). Let $H = H' E^r$, where $r$ is an integer such that if the robot follows $H$ on a maze without meeting any VNE and HNE then it returns back to its starting point. We construct the algorithm
\[
F = N^u (E N^{6p} H S^{6p})^{\lambda+p}.
\]
\textbf{Claim. } For any maze in $F_i$, after the robot follows $X \text{ } F$, it visits the destination point.
\begin{proof}
We may assume without loss of generality that after the robot follows $X \text{ } N^u$ it gets in the west strip on the row $r_u$. While the robot executes one instance of $E N^{6p} H S^{6p}$ it advances one unit eastwards in the west strip on the row $r_u$ without meeting any VNE or HNE. Indeed, every row at latitude greater than $3p$ is a path in the maze. The robot eventually eventually gets at $\textbf{w}$. Immediately afterwards, it follows $N^{6p}$, remaining at $\textbf{w}$ and then $H$, hence it gets to the destination point.
\end{proof}
This finally solves \textbf{Case (2)} in which the destination point was connected with the west strip by a finite number of finite columns. It is immediate to see that the presence of infinite columns in the obstacle strip does not affect any of the arguments made in this case. \\
\textbf{Case (3). } We assume that the destination point is in the obstacle strip and there exists some parameter $h_{(i, i+1)} < \infty$. We recall that this is equivalent to the existence of a pair of consecutive upper infinite columns (or a consecutive upper infinite column and an infinite column at the border of the obstacle strip and either the east or west strip) which are not connected by HEs at arbitrarily high latitudes. By symmetry, treating this case also solves the homologous case in which there exists some parameter $l_{(i, i+1)} < \infty$. \\
\textbf{3(i) } We assume $h_{(a, a+1)}<\infty$. We recall that the pair of columns $(c_a, c_{a+1})$ is at the border between the west strip and the obstacle strip and we also recall that the pair of columns $(c_b, c_{b+1})$ is at the border between the obstacle strip and the east strip. We assume without loss of generality that there exists a HE between the west strip and a finite column or a lower infinite column (otherwise we are done by \textbf{Part I}). Let $R$ be a finite column or a lower infinite column on the column $c_{a+1}$ such that there exists a HE between the west strip and $R$ on some row $r_c$. Let $\textbf{w}$ be the uppermost vertex of $R$. Let $j=h_{(a,a+1)}+l$ and $\textbf{v} = (a,j)$ be the easternmost point on the row $r_j$ in the west strip. We recall the generic algorithm
\[
SME^{(j-c)}(a,e,L)=((((((L)^e S^{j-c} E N^{j-c})^e E)^e N E S)^e S E N)^e N N E S S)^e...S^a E N^a)^e.
\]
We define the algorithm
\[
F = N^j SME^{(j-c)}(\lambda+j+q,\lambda+j+q,L).
\]
\textbf{Claim. } For any maze in $F_i$, after the robot follows $X \text{ } F$, it visits the destination point.
\begin{proof}
We may assume without loss of generality that after the robot follows $X \text{ } N^j$ it is in the west strip on the row $r_j$. By the choice of parameters and by Corollary~\ref{c1}, while the robot follows $SME^{(j-c)}(\lambda+j+q,\lambda+j+q,L)$ it advances eastwards in the west strip oscillating about row $r_j$. After the robot starts from some point on the row $r_j$ in the west strip and follows $L$ its longitude does not decrease and it remains in the west strip. It eventually gets to the point $\textbf{v}$. After the robot starts from $\textbf{v}$ and follows $S^{j-c}EN^{j-c}$ it gets to the point $\textbf{w}$. Immediately afterwards, it follows $L$ and gets to the destination point.
\end{proof}
\noindent
\textbf{3(ii) } Consider the pair of consecutive columns $(c_i, c_{i+1})$ which is not at the border between the west strip and the obstacle strip. Assume there are not arbitrarily high HEs between the columns $c_i$ and $c_{i+1}$, i.e. $h_{(i,i+1)}<\infty$. Assume further that there exists a pass on some row $r_c$ between the west strip and an upper infinite column $R$ (see the case \textbf{3(i)}). We define $K = S^{\lambda+2q+|h_{(i,i+1)}|+1} N^{\lambda+2q+|h_{(i,i+1)}|+1}$. We define the algorithm
\[
F = N^c \text{ } SME(\lambda+q,\lambda+q,K) \text{ } S^{\lambda+2q+|h_{(i,i+1)}|+1} \text{ } L.
\]
\textbf{Claim. } For any maze in $F_i$, after the robot follows $X \text{ } F$, it visits the destination point.
\begin{proof}
We may assume without loss of generality that after the robot follows $X \text{ } N^c$ it gets in the west strip on the row $r_c$. While the robot follows $SME(\lambda+q,\lambda+q,K)$ it advances eastwards in the west strip oscillating about the row $r_c$. It eventually enters the upper infinite columns $R$. Immediately afterwards it executes $K$ and gets to latitude at least $\lambda+q+|h_{(i,i+1)}|+1$ in $R$. While the robot is in the obstacle strip and follows $SME$ it advances eastwards through upper infinite columns at latitudes greater than $h_{(i,i+1)}+1$. Hence the robot remains stuck in some column $c_j$ with $j \le i$ at latitude $\lambda+2q+|h_{(i,i+1)}|+1$ above the highest VNE in the columns $c_r$ with $a \le r \le j$. After that the robot follows $S^{\lambda+2q+|h_{(i,i+1)}|+1}$ and gets to a special point. Therefore, after the robot follows $X \text{ } F$ it gets to the destination point.
\end{proof}
This finally solves \textbf{Case (3)} in which the destination point is in the obstacle strip in a finite or infinite column and there exists some parameter $h_{(i, i+1)} < \infty$. Moreover, the case in which there exists some parameter $l_{(i, i+1)} < \infty$ is tackled similarly by symmetry. Finally, it is immediate to see that the presence of infinite columns in the obstacle strip does not affect any of the arguments made in this case.\\
\textbf{Case (4). } This is the final case, in which we may assume that \textbf{Case (3)} does not hold and the destination point is in the obstacle strip in a finite column, an upper infinite column or lower infinite column and it is connected to the west strip by a (finite, possibly empty) sequence of finite columns followed by a (finite, non-empty) sequence of upper infinite columns, in this order starting from the destination point and advancing towards the west strip. Indeed, we may assume that the west strip is accessible by \textbf{Part II}. The case in which is the destination point is not in the obstacle strip is tackled in \textbf{Case (1)}. Furthermore, if we assume that the destination point is in the obstacle strip, it may either be reachable from the west strip through a finite sequence of finite columns tackled in \textbf{Case (2)} or otherwise it must be reachable from the west strip through a finite sequence of upper/lower infinite and finite columns which contains at least one upper or lower infinite column. Choose any such finite sequence of columns which leads to the destination point starting from the west strip and call the last upper or lower infinite column in the sequence $c$; this may either be the last element of the sequence or it might be followed by a finite sequence of finite columns. By \textbf{Case (3)} we may assume that there are horizontal edges between consecutive upper infinite columns and between consecutive lower infinite columns at latitudes arbitrarily high and low, respectively. Hence, assuming without loss of generality as we may that $c$ is an upper infinite column, $c$ can be reached from the west strip through a finite sequence of upper infinite columns. Therefore, the last case that we tackle is the one in which we assume that the destination point is connected to the west strip by a (finite, possibly empty) sequence of finite columns followed by a (finite, non-empty) sequence of upper infinite columns, in this order starting from the destination point and advancing towards the west strip.

The condition that \textbf{Case (3)} does not hold means that in this case we assume that all the parameters $h_{(i, i+1)}$ and $l_{(i, i+1)}$ are all infinity for $a \leq i \leq b$; in particular, this implies that there exists a west ascending chain. We recall that the pair of columns $(c_a, c_{a+1})$ are at the border between the west strip and the obstacle strip; we also recall that the pair of columns $(c_b, c_{b+1})$ are at the border between the obstacle strip and the east strip. We further recall that a west ascending chain is a finite sequence of HEs: $HE_a, HE_{a+1}, \hdots, HE_b$ such that $HE_a$ is the upper west pass (i.e. the lowermost HE between the west strip and the upper infinite column on $c_{a+1}$ above all passes in the obstacle strip) and $HE_m$ is the lowermost HE between the pair of columns $(c_m, c_{m+1})$ at latitude at least that of $HE_{m-1}$ for $m=a+1, \hdots, b$. In this case, we take $R_{a+1}, R_{a+2}, \hdots , R_n$ to be a finite non-empty sequence of upper infinite columns and $R_{n+1}, \hdots , R_k$ to be a finite possibly empty sequence of finite columns and finally we take $R$ to be a finite, upper infinite or lower infinite column such that $R$ contains the destination point and there exists a HE between the west strip and $R_{a+1}$, between $R_m$ and $R_{m+1}$ for $a+1 \leq m \leq k -1$ and between $R_k$ and $R$. By the discussion at the beginning of the case, we may assume that such a sequence has the extra property that $R_m$ is on the column $c_m$ for $a+1 \leq m \leq n$. Moreover, if $R_{n+1}$ exists we may assume that $R_{n+1} \in c_{n+1}$; indeed, $R_{n+1} \in c_{n+1}$ or $R_{n+1} \in c_{n-1}$ and if $R_{n+1} \in c_{n-1}$ then we can use the symmetry of the argument in \textbf{Part II} to assume that the robot is in the east strip on the $x$-axis. From that perspective, we can use the arguments from the case that we are treating with $R_{n+1} \in c_{n+1}$. Obviously, if $R_{n+1}$ does not exist, by the same argument we may assume that $R$ is in $c_{n+1}$. Finally, say that the row $r_i$ contains the upper west pass and note that the upper west pass is above all passes in the obstacle strip and therefore, as \textbf{Case (3)} does not hold, it is above all special vertices. \\
\textbf{4(i) }We assume there exists a magical west row $r_j$. We recall that a magical west row is a row which is a path when restricted to the west strip, and it contains a HNE; its west cutoff is its westernmost HNE. We see in the end that our argument also solves the case when there exists a magical east row. Denote by $\textbf{v}$ the westernmost vertex of the west cutoff of $r_j$. Then, by inspecting the tertiary rectangle, we can construct an algorithm $K$ such that if the robot starts from $\textbf{v}$ and follows $K$ it gets to the destination point. We construct the algorithm
\[
F = N^jE^{\lambda+q}K.
\]
\textbf{Claim. } For any maze in $F_i$, after the robot follows $X \text{ } F$, it visits the destination point.
\begin{proof}
We may assume without loss of generality that after the robot follows $X \text{ } N^j$, it gets on the row $r_j$. Therefore, after it follows $X \text{ } N^j \text{ } E^{\lambda+q}$ the robot gets to the point $\textbf{v}$. Hence, after the robot follows $X \textbf{ } F$ it gets to the destination point.
\end{proof}
Clearly, in this case we may easily drop the general assumption that $R_{n+1} \in c_{n+1}$. Therefore, this argument also solves the case when there exists a magical east row.
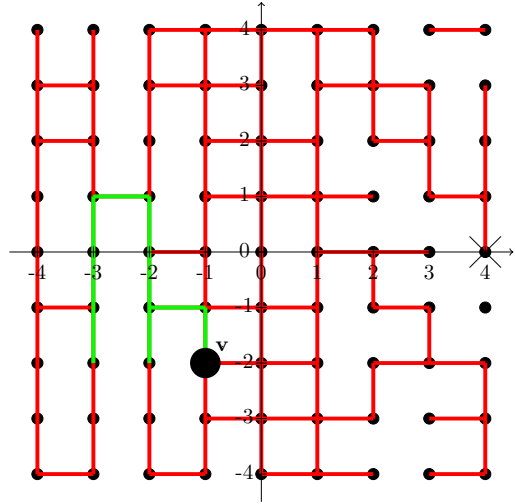
\begin{figure}
\centering
\resizebox{0.5\textwidth}{!}{
\begin{tikzpicture}
    [
        dot/.style={circle,draw=black, fill,inner sep=2pt},
        cross/.style={cross out, draw=black, minimum size=2*(#1-\pgflinewidth), inner sep=4pt, outer sep=4pt},
    ]

\foreach \x in {-4,...,4}{
 \foreach \y in {-4,...,4}{
    \node[dot] at (\x,\y){ };
}}

\draw[line width=0.7mm, red] (-4, -4) -- (-4, 4);
\draw[line width=0.7mm, red] (-3, -4) -- (-3, 4);
\draw[line width=0.7mm, red] (-2, -4) -- (-2, 4);
\draw[line width=0.7mm, red] (-1, -4) -- (-1, 4);
\draw[line width=0.7mm, red] (0, -4) -- (0, 4);
\draw[line width=0.7mm, red] (1, -4) -- (1, 4);
\draw[line width=0.7mm, red] (2, 4) -- (2, 2);
\draw[line width=0.7mm, red] (2, 0) -- (2, -1);
\draw[line width=0.7mm, red] (2, -2) -- (2, -3);
\draw[line width=0.7mm, red] (3, 3) -- (3, 1);
\draw[line width=0.7mm, red] (3, -1) -- (3, -2);
\draw[line width=0.7mm, red] (4, -2) -- (4, -4);
\draw[line width=0.7mm, red] (4, 0) -- (4, 3);

\draw[line width=0.7mm, red] (3, 4) -- (4, 4);
\draw[line width=0.7mm, red] (2, 4) -- (-2, 4);
\draw[line width=0.7mm, red] (-2, 3) -- (0, 3);
\draw[line width=0.7mm, red] (1, 3) -- (3, 3);
\draw[line width=0.7mm, red] (-4, 3) -- (-3, 3);
\draw[line width=0.7mm, red] (-4, 2) -- (-3, 2);
\draw[line width=0.7mm, red] (-1, 2) -- (1, 2);
\draw[line width=0.7mm, red] (2, 2) -- (3, 2);
\draw[line width=0.7mm, red] (-1, 1) -- (2, 1);
\draw[line width=0.7mm, red] (-3, 1) -- (-2, 1);
\draw[line width=0.7mm, red] (3, 1) -- (4, 1);
\draw[line width=0.7mm, red] (-2, 0) -- (-1, 0);
\draw[line width=0.7mm, red] (1, 0) -- (3, 0);
\draw[line width=0.7mm, red] (-2, -1) -- (1, -1);
\draw[line width=0.7mm, red] (2, -1) -- (3, -1);
\draw[line width=0.7mm, red] (-4, -1) -- (-3, -1);
\draw[line width=0.7mm, red] (1, -2) -- (-1, -2);
\draw[line width=0.7mm, red] (2, -2) -- (4, -2);
\draw[line width=0.7mm, red] (-1, -3) -- (2, -3);
\draw[line width=0.7mm, red] (3, -3) -- (4, -3);
\draw[line width=0.7mm, red] (-2, -4) -- (-1, -4);
\draw[line width=0.7mm, red] (0, -4) -- (2, -4);
\draw[line width=0.7mm, red] (3, -4) -- (4, -4);
\draw[line width=0.7mm, red] (-4, -4) -- (-3, -4);

\draw[line width=0.7mm, green] (-3, -2) -- (-3, 1);
\draw[line width=0.7mm, green] (-2, -2) -- (-2, 1);
\draw[line width=0.7mm, green] (-3, 1) -- (-2, 1);
\draw[line width=0.7mm, green] (-2, -1) -- (-1, -1);
\draw[line width=0.7mm, green] (-1, -2) -- (-1, -1);

\draw node[fill,circle,inner sep=5.5pt,minimum size=2pt] at (-1, -2) {};

\draw (-0.7, -1.7) -- node {\textbf{v}} (-0.7, -1.7);

\draw (4,0) node[cross] {};

\foreach \x in {-4,...,4}
    \draw (\x,.1) -- node[below,yshift=-1mm] {\x} (\x,-.1);
\foreach \y in {-4,...,4}
    \draw (.1,\y) -- node[below,xshift=-3mm, yshift=3mm] {\y} (-.1,\y);
\draw[->,line width=0.15mm] (0,-4.5) -- (0,4.5);
\draw[->,line width=0.15mm] (-4.5,0) -- (4.5,0);
\end{tikzpicture}
}
\caption{\textbf{Part III, Case (4)(ii).} There exists a special almost empty west row $r_j$ and let $\textbf{v}$ be the easternmost vertex of the west cutoff of $r_j$. We assume that there are no VEs removed other than the ones shown in the figure. In this example, let $r_j = r_{-2}$ and so $\textbf{v} = (-1, -2)$. Let us see how the robot gets to $\textbf{v}$ after it follows $AME (5,5)$ starting from $(-3, -2)$. As long as the robot is at the west of $\textbf{v}$, the $W$ instructions do not decrease the longitude of the robot, as there are no HEs on $r_j$ at the west of $\textbf{v}$. Therefore, the robot takes the green path to $\textbf{v}$. Once the robot gets at $\textbf{v}$, every subsequent subalgorithm of the form $N^iES^iW$ takes it back to $\textbf{v}$: after $N^iES^i$, the robot is either at $\textbf{v} = (-1, -2)$ or $(0, -2)$; immediately afterwards, the robot follows $W$ and the presence of a HE between $(-1, -2)$ and $(0, -2)$ guarantees that the robot returns back to $\textbf{v}$. Immediately after the robot follows $AME$ and it gets to $\textbf{v}$, it follows $K$ and it visits the destination point. For our example we can take $K = N^3 E^2 N^2 E^2 S^2 E S$.} \label{10}
\end{figure}

\noindent
\textbf{4(ii) } We assume there exists a special almost empty west row $r_j$ and we call $\textbf{v}$ the easternmost vertex of the west cutoff of $r_j$. We recall that an almost empty west row is a row that in the west strip is the complement of an infinite path followed by a non-empty finite path; its west cutoff is its easternmost HNE in the west strip. Then, by inspecting the tertiary rectangle, we can construct an algorithm $K$ such that if the robot starts from $\textbf{v}$ and follows $K$ it gets to the destination point (see Figure~\ref{10}).

We then define the algorithm $auxiliary \_ move \_ east$,
\[
AME (a, e) = ((NESW)(SENW)(N^2ES^2W)(S^2EN^2W) \hdots (S^aEN^aW))^e.
\]
We finally define the algorithm
\[
F:= N^j \text{ } W^q \text{ } AME (\lambda+q, \lambda+q) \text{ } K.
\]
\textbf{Claim. } For any maze in $F_i$, after the robot follows $X \text{ } F$, it visits the destination point. 
\begin{proof}
We may assume without loss of generality that after the robot follows $X \text{ } N^j \text{ } W^q$ it gets on the row $r_j$ at a longitude at most that of $\textbf{v}$. By the choice of parameters, while the robot follows $AME (\lambda+q, \lambda+q)$ it advances eastwards in the west strip oscillating about the row $r_j$ and it remains stuck at the point $\textbf{v}$. Hence, after the robot follows $X \text{ } F$, it reaches the destination point (see Figure~\ref{10}).
\end{proof}
\begin{figure}[h!]
\centering
\resizebox{0.75\textwidth}{!}{
\begin{tikzpicture}
    [
        dot/.style={circle,draw=black, fill,inner sep=2pt},
        cross/.style={cross out, draw=black, minimum size=2*(#1-\pgflinewidth), inner sep=4pt, outer sep=4pt},
    ]

\foreach \x in {-6,...,6}{
 \foreach \y in {-6,...,6}{
    \node[dot] at (\x,\y){ };
}}

\draw[line width=0.7mm, red] (-6, -6) -- (-6, 6);
\draw[line width=0.7mm, red] (-5, -6) -- (-5, 6);
\draw[line width=0.7mm, red] (-4, -6) -- (-4, 6);
\draw[line width=0.7mm, red] (-3, -6) -- (-3, 6);
\draw[line width=0.7mm, red] (-2, -6) -- (-2, 6);
\draw[line width=0.7mm, red] (-1, -6) -- (-1, 6);
\draw[line width=0.7mm, red] (0, -6) -- (0, 6);
\draw[line width=0.7mm, red] (1, -6) -- (1, 6);

\draw[line width=0.7mm, red] (2, -6) -- (2, -4);
\draw[line width=0.7mm, red] (2, -3) -- (2, 6);

\draw[line width=0.7mm, red] (3, -5) -- (3, -4);
\draw[line width=0.7mm, red] (3, -2) -- (3, 6);

\draw[line width=0.7mm, red] (4, -5) -- (4, -4);
\draw[line width=0.7mm, red] (4, -3) -- (4, 6);

\draw[line width=0.7mm, red] (5, -6) -- (5, -5);
\draw[line width=0.7mm, red] (5, -4) -- (5, -1);
\draw[line width=0.7mm, red] (5, 0) -- (5, 6);

\draw[line width=0.7mm, red] (6, -6) -- (6, 6);

\draw[line width=0.7mm, red] (-2, 5) -- (1, 5);
\draw[line width=0.7mm, red] (3, 5) -- (6, 5);

\draw[line width=0.7mm, red] (4, 4) -- (6, 4);
\draw[line width=0.7mm, red] (-5, 4) -- (-4, 4);

\draw[line width=0.7mm, red] (-1, 3) -- (0, 3);
\draw[line width=0.7mm, red] (2, 3) -- (4, 3);
\draw[line width=0.7mm, red] (5, 3) -- (6, 3);

\draw[line width=0.7mm, red] (-2, 2) -- (-3, 2);
\draw[line width=0.7mm, red] (0, 2) -- (1, 2);
\draw[line width=0.7mm, red] (2, 2) -- (3, 2);

\draw[line width=0.7mm, red] (-5, 1) -- (-4, 1);
\draw[line width=0.7mm, red] (-2, 1) -- (-1, 1);
\draw[line width=0.7mm, red] (0, 1) -- (2, 1);
\draw[line width=0.7mm, red] (4, 1) -- (5, 1);

\draw[line width=0.7mm, red] (0, 0) -- (1, 0);
\draw[line width=0.7mm, red] (4, 0) -- (5, 0);

\draw[line width=0.7mm, red] (-1, -1) -- (0, -1);
\draw[line width=0.7mm, red] (2, -1) -- (4, -1);

\draw[line width=0.7mm, red] (-5, -2) -- (-3, -2);
\draw[line width=0.7mm, red] (-1, -2) -- (0, -2);
\draw[line width=0.7mm, red] (1, -2) -- (2, -2);
\draw[line width=0.7mm, red] (4, -2) -- (5, -2);

\draw[line width=0.7mm, red] (-2, -3) -- (0, -3);

\draw[line width=0.7mm, red] (-6, -4) -- (-3, -4);
\draw[line width=0.7mm, red] (-2, -4) -- (-1, -4);
\draw[line width=0.7mm, red] (4, -4) -- (5, -4);

\draw[line width=0.7mm, red] (-5, -5) -- (-3, -5);
\draw[line width=0.7mm, red] (-1, -5) -- (0, -5);
\draw[line width=0.7mm, red] (3, -5) -- (4, -5);

\draw[line width=0.7mm, red] (-1, -6) -- (0, -6);
\draw[line width=0.7mm, red] (4, -6) -- (5, -6);

\draw[line width=1mm, green] (1, 1) -- (2, 1);
\draw[line width=1mm, green] (2, 2) -- (3, 2);
\draw[line width=1mm, green] (3, 3) -- (4, 3);
\draw[line width=1mm, green] (4, 4) -- (6, 4);

\draw node[fill,circle,inner sep=5.5pt,minimum size=2pt] at (-2, 1) {};
\draw node[fill,circle,inner sep=5.5pt,minimum size=2pt] at (2, 1) {};
\draw node[fill,circle,inner sep=5.5pt,minimum size=2pt] at (4, 4) {};
\draw node[fill,circle,inner sep=5.5pt,minimum size=2pt] at (5, -1) {};

\draw (-1.7,1.4) -- node {\textbf{v}} (-1.7,1.4);
\draw (2.4,1) -- node {\textbf{t}} (2.4,1);
\draw (5.4, -1) -- node {\textbf{z}} (5.4, -1);
\draw (4.3, 4.4) -- node {\textbf{w}} (4.3, 4.4);

\draw (3,-4) node[cross] {};

\foreach \x in {-6,...,6}
    \draw (\x,.1) -- node[below,yshift=-1mm] {\x} (\x,-.1);
\foreach \y in {-6,...,6}
    \draw (.1,\y) -- node[below,xshift=-3mm, yshift=3mm] {\y} (-.1,\y);
\draw[->,line width=0.15mm] (0,-6.5) -- (0,6.5);
\draw[->,line width=0.15mm] (-6.5,0) -- (6.5,0);
\end{tikzpicture}
}
\caption{\textbf{Part III, Case (4)(iii).} We assume there exist the special upper west paired HNEs. We assume that there are no VEs removed other than the ones shown in the figure, so $a=1$ and $b= 5$. The upper west pass between $c_1$ and $c_2$ is $HE_1 = (1, 1), (2, 1)$, above all the passes in the obstacle strip; the west ascending chain is coloured green. The chosen path from the west strip to the destination point goes through $R_2 = \{ (2, -3), (2, -2), \hdots \}$, then $R_3$, $R_n = R_4$, $R_5 = \{ (5, -1), (5, -2), (5, -3), (5, -4) \}$, $R_6 = \{ (4, -4), (4, -5) \}$, $R = R_7 = \{ (3, -4), (3, -5) \}$. The point $\textbf{z} = (5, -1)$ is the uppermost vertex of $R_5$. For the purpose of this example, let us assume $c_{uw} = 6$, although this should be larger. To find the upper west paired HNEs, the set of HNEs on $r_1$ is the set of all possible candidates for the upper HNE in the pair. To find the second HNE in the pair, we look on $r_{1-6} = r_{-5}$ (i.e. at latitude $i-c_{uw}$) to find a matching HNE at the same longitude with one on $r_1$ and we choose the easternmost one. If none such HNE exists, we repeat the same process on $r_{-6}$, then on $r_{-7}$ and so on. In this example, we find the upper west paired HNEs to be $(-3, 1), (-2, 1)$ and $(-3, -5), (-2, -5)$. Therefore $\textbf{v} = (-2, 1)$ and $\textbf{t} = (2, 1)$. The only HE between $R_4$ and $R_5$ is $(4, -2), (5, -2)$ at latitude $\nu = -2$, so $\textbf{w} = (4, 4)$. Then, if the robot follows $K_1 = E S^6 E N^6 EE$ starting from $\textbf{v}$ it gets to $\textbf{t}$; if the robot follows $K_2 = S^5 N^5 E S^5 N^5 E S^7 N^7$ starting from $\textbf{t}$ it gets to $\textbf{w}$, passing through the green edges from $R_2$ to $R_3$ and from $R_3$ to $R_4$; if the robot follows $K_3 = S^6 E N^6$ starting from $\textbf{w}$ it gets to $\textbf{z}$; if the robot follows $K_4 = N^3 S^3 W NS W NS E$ starting from $\textbf{z}$ it gets to the destination point; $K_5 = E^{13}$. Finally, we remark that if the robot follows $K= K_1 K_2 K_3 K_4 K_5$ starting on any point of $r_1$ strictly at the west of $\textbf{v}$, then it returns on $r_1$ strictly at the west of $\textbf{v}$.} \label{11}
\end{figure}
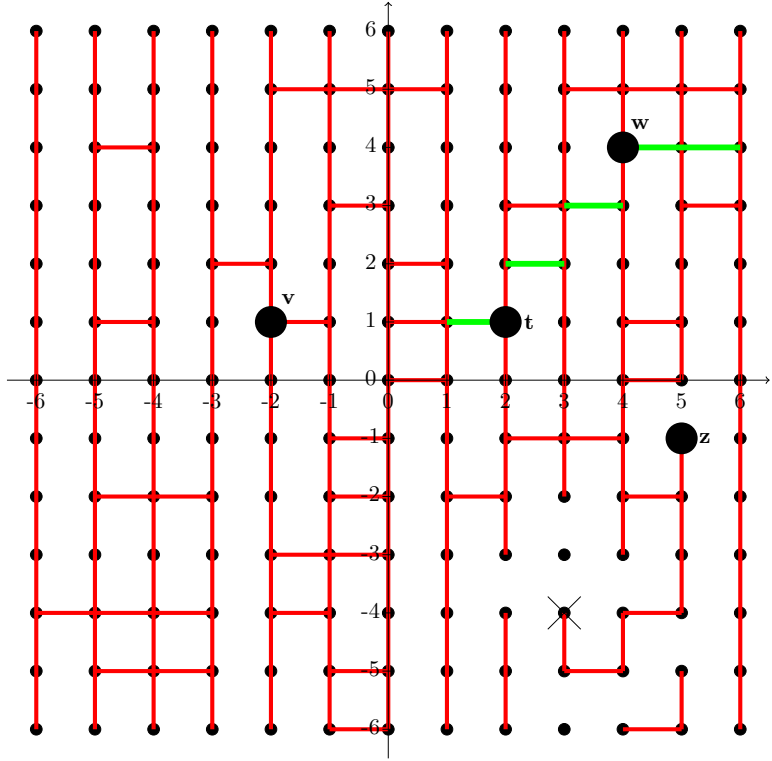

\noindent
\textbf{4(iii) } We assume there exist the special upper west paired HNEs. We recall the following definitions: let $HE_a, \hdots, HE_b$ be the west ascending chain with $HE_a$ being the upper west pass say on some row $r_i$ and also say that $HE_b$ is on some row $r_t$. Then $c_{uw} = t+p$ is the upper west constant, where $p$ is the parameter of the primary rectangle. The upper west paired HNEs are any pair of HNEs with the same longitude, in the west strip, such that the upper HNE is at latitude $i$, on the same row as the upper west pass, and the lower HNE is at latitude at most $i-c_{uw}$. For the special upper west paired HNEs, we choose the upper west paired HNEs with the uppermost easternmost lower HNE. In this subcase, we assume that there exist the special upper west paired HNEs, with the upper HNE on the row $r_i$ and the lower HNE on the row $r_j$, $j \leq i-c_{uw}$.

Let the point $\textbf{v}$ be the easternmost vertex of the upper HNE of the pair and let the point $\textbf{t}$ be the easternmost vertex of the upper west pass. We pick any HE between the upper infinite column $R_n$ and the finite column $R_{n+1}$ at latitude say $\nu$. In the case that $R_{n+1}$ does not exist, we pick the lowermost HE between the upper infinite column $R_n$ and $R$ at latitude say $\nu$. Let the point $\textbf{w}$ be the vertex in the infinite column $R_n$ at latitude $\nu+i-j$. Then the eastern vertex of $HE_{n-1}$ which has a latitude of at most $t$ by definition is in the column $c_n$ below $\textbf{w}$; indeed, $\nu+i-j \geq \nu + t + p$ and $\nu+p \geq 0$. Finally, let the point $\textbf{z}$ be the uppermost vertex of the finite column $R_{n+1}$ if $R_{n+1}$ exists. In the following argument, we assume that $R_{n+1}$ exists and it will be clear how this also naturally treats the case when $R_n$ is connected to $R$ which contains the destination point. For an illustration of all these definitions in a concrete example, see Figure~\ref{11}. 

In what follows, we will construct $5$ algorithms $K_1, \hdots, K_5$, by inspecting the tertiary rectangle.

We start by constructing a finite algorithm $K_1$ of the form $K_1=\prod_{m=1}^{h_1} S^{\epsilon_m}EN^{\epsilon_m}$, where $\epsilon_m \in \{0, i-j \}$ for all $1 \le m \le h_1$, such that after the robot follows $K_1$ starting from the point $\textbf{v}$ it gets to the point $\textbf{t}$. We make use of the fact that in the west strip at the east of the special upper west paired HNEs at each given longitude at least one of the rows $r_i$ and $r_j$ contains a HE. Clearly, $\epsilon_{h_1}=0$.

We construct a finite algorithm $K_2$ of the form $K_2=(\prod_{m=a+1}^{n-1} S^{k_m}N^{k_m}E)S^{k_n}N^{k_n} $, where $k_m$ is a positive integer for all $a+1 \le m \le n$, such that if the robot starts from the point $\textbf{t}$ and it follows $K_2$, it gets to the point $\textbf{w}$. More specifically, if the robot is in the upper infinite column $R_m$ in the column $c_m$ at the easternmost end of $HE_{m-1}$ and it follows $S^{k_m}N^{k_m}E$ it gets in the upper infinite column $R_{m+1}$ in the column $c_{m+1}$ at the easternmost end of $HE_{m}$, for $a+1 \le m \le n-1$; if the robot is in the upper infinite column $R_n$ in the column $c_n$ at the easternmost end of $HE_{n-1}$ and it follows $S^{k_n}N^{k_n}$ it gets to the point $\textbf{w}$.

We construct an algorithm $K_3=S^{i-j} E N^{i-j}$, such that if the robot starts from $\textbf{w}$ and it follows $K_3$, it gets to the point $\textbf{z}$.

We construct an algorithm $K_4$ of the form $K_4=(\prod_{m=n+1}^{k} N^{k_m}S^{k_m}E^{\epsilon_m})$ $N^{k_{k+1}} \text{ } N^{-k_{k+1}}$, where $\epsilon_m \in \{-1, 1\}$ and $k_m$ is an integer for all $n+1 \le m \le k+1$, such that if the robot starts from the point $\textbf{z}$ and it follows $K_4$, it visits the destination point. More specifically, if the robot is at some specified latitude in the finite column $R_m$ and it follows $N^{k_m}N^{-k_m}E^{\epsilon_m}$, it gets to some specified latitude in the finite column $R_{m+1}$ for $n+1 \le m \le k$, where by convention we write $R_{k+1}$ for $R$. If the robot is at some specified latitude inside $R$ and follows $N^{k_{k+1}}N^{-k_{k+1}}$ it visits the destination point.

We define the algorithm $K_5=E^{|K_4|}.$

We define the algorithm $K =K_1 K_2 K_3 K_4 K_5.$ Note that if the robot is on the row $r_i$ strictly at the west of the point $\textbf{v}$ and it follows $K$ it returns on the row $r_i$ strictly at the west of $\textbf{v}$. Indeed, by examining $K_1, \hdots, K_5$ one by one, we conclude that if the robot starts strictly at the west of $\textbf{v}$, while executing $K$ it can only change its longitude at latitudes $i$ or $j$. Thus, the existence of the special west paired HNEs prevents the robot from reaching a longitude at least that of $\textbf{v}$. In particular, if the robot is on the row $r_i$ strictly at the west of the point $\textbf{v}$ and it follows $K$ it does not meet any VNE, so it is easy to see that it returns back to the row $r_i$. Finally, the only $W$ instructions in $K$ could appear as part of $K_4$, which is followed by $K_5=E^{|K_4|}$ in $K$; therefore if the robot is on the row $r_i$ strictly at the west of the point $\textbf{v}$ and it follows $K$ its longitude does not decrease. If the robot starts at the point $\textbf{v}$ and it follows $K$, then it visits the destination point; this follows directly from the definitions of $K_1, \hdots, K_5$ (see Figure~\ref{11}).

Finally, we construct the algorithm
\[
F = N^i \text{ } MW (i-j, q) \text{ } SME(\mu+\lambda+2q,\mu+\lambda+2q,K),
\]
where $\mu = |MW (i-j, q)|.$
\begin{prop}
For any maze in $F_i$, after the robot follows $X \text{ } F$, it visits the destination point.
\end{prop}
\begin{proof}
We may assume without loss of generality that after the robot follows the algorithm $X \text{ } N^i \text{ } MW (i-j, q)$ it gets on the row $r_i$ at a longitude at most that of the point $\textbf{v}$. By the choice of parameters and by Lemma~\ref{l2}, while the robot follows $SME (\mu+\lambda+2q,\mu+\lambda+2q,K)$, it advances eastwards in the west strip oscillating about the row $r_i$. After defining $K$, we checked that it satisfies the conditions required in order to apply Lemma~\ref{l2}. Therefore, by Lemma~\ref{l2}, the robot gets for the first time to the point $\textbf{v}$ not while executing $K$, but while executing a locomotory move. Immediately afterwards, it follows $K$ and it gets to the destination point. The conclusion follows.
\end{proof}
\begin{figure}[h!]
\centering
\resizebox{0.75\textwidth}{!}{
\begin{tikzpicture}
    [
        dot/.style={circle,draw=black, fill,inner sep=2pt},
        cross/.style={cross out, draw=black, minimum size=2*(#1-\pgflinewidth), inner sep=4pt, outer sep=4pt},
    ]

\foreach \x in {-6,...,6}{
 \foreach \y in {-6,...,6}{
    \node[dot] at (\x,\y){ };
}}

\draw[line width=0.7mm, red] (-6, -6) -- (-6, 6);
\draw[line width=0.7mm, red] (-5, -6) -- (-5, 6);
\draw[line width=0.7mm, red] (-4, -6) -- (-4, 6);
\draw[line width=0.7mm, red] (-3, -6) -- (-3, 6);
\draw[line width=0.7mm, red] (-2, -6) -- (-2, 6);
\draw[line width=0.7mm, red] (-1, -6) -- (-1, 6);
\draw[line width=0.7mm, red] (0, -6) -- (0, 6);
\draw[line width=0.7mm, red] (1, -6) -- (1, 6);

\draw[line width=0.7mm, red] (2, -6) -- (2, -4);
\draw[line width=0.7mm, red] (2, -3) -- (2, 6);

\draw[line width=0.7mm, red] (3, -5) -- (3, -4);
\draw[line width=0.7mm, red] (3, -2) -- (3, 6);

\draw[line width=0.7mm, red] (4, -5) -- (4, -4);
\draw[line width=0.7mm, red] (4, -3) -- (4, 6);

\draw[line width=0.7mm, red] (5, -6) -- (5, -5);
\draw[line width=0.7mm, red] (5, -4) -- (5, -1);
\draw[line width=0.7mm, red] (5, 0) -- (5, 6);

\draw[line width=0.7mm, red] (6, -6) -- (6, 6);

\draw[line width=0.7mm, red] (-2, 5) -- (1, 5);
\draw[line width=0.7mm, red] (3, 5) -- (6, 5);

\draw[line width=0.7mm, red] (4, 4) -- (6, 4);
\draw[line width=0.7mm, red] (-5, 4) -- (-4, 4);

\draw[line width=0.7mm, red] (-1, 3) -- (0, 3);
\draw[line width=0.7mm, red] (2, 3) -- (4, 3);
\draw[line width=0.7mm, red] (5, 3) -- (6, 3);

\draw[line width=0.7mm, red] (-2, 2) -- (-3, 2);
\draw[line width=0.7mm, red] (0, 2) -- (1, 2);
\draw[line width=0.7mm, red] (2, 2) -- (3, 2);

\draw[line width=0.7mm, red] (-6, 1) -- (-5, 1);
\draw[line width=0.7mm, red] (-5, 1) -- (-4, 1);
\draw[line width=0.7mm, red] (-2, 1) -- (-1, 1);
\draw[line width=0.7mm, red] (-3, 1) -- (2, 1);
\draw[line width=0.7mm, red] (4, 1) -- (5, 1);

\draw[line width=0.7mm, red] (0, 0) -- (1, 0);
\draw[line width=0.7mm, red] (4, 0) -- (5, 0);

\draw[line width=0.7mm, red] (-1, -1) -- (0, -1);
\draw[line width=0.7mm, red] (2, -1) -- (4, -1);

\draw[line width=0.7mm, red] (-5, -2) -- (-3, -2);
\draw[line width=0.7mm, red] (-1, -2) -- (0, -2);
\draw[line width=0.7mm, red] (1, -2) -- (2, -2);
\draw[line width=0.7mm, red] (4, -2) -- (5, -2);

\draw[line width=0.7mm, red] (-2, -3) -- (0, -3);

\draw[line width=0.7mm, red] (-6, -4) -- (-3, -4);
\draw[line width=0.7mm, red] (-2, -4) -- (-1, -4);
\draw[line width=0.7mm, red] (4, -4) -- (5, -4);

\draw[line width=0.7mm, red] (-5, -5) -- (-3, -5);
\draw[line width=0.7mm, red] (-1, -5) -- (0, -5);
\draw[line width=0.7mm, red] (3, -5) -- (4, -5);

\draw[line width=0.7mm, red] (-1, -6) -- (0, -6);
\draw[line width=0.7mm, red] (4, -6) -- (5, -6);
\draw[line width=0.7mm, red] (-4, -6) -- (-3, -6);

\draw[line width=1mm, green] (1, 1) -- (2, 1);
\draw[line width=1mm, green] (2, 2) -- (3, 2);
\draw[line width=1mm, green] (3, 3) -- (4, 3);
\draw[line width=1mm, green] (4, 4) -- (6, 4);

\draw node[fill,circle,inner sep=5.5pt,minimum size=2pt] at (-3, 1) {};
\draw node[fill,circle,inner sep=5.5pt,minimum size=2pt] at (2, 1) {};
\draw node[fill,circle,inner sep=5.5pt,minimum size=2pt] at (4, 4) {};
\draw node[fill,circle,inner sep=5.5pt,minimum size=2pt] at (5, -1) {};

\draw (-2.7,1.4) -- node {\textbf{v}} (-2.7,1.4);
\draw (2.4,1) -- node {\textbf{t}} (2.4,1);
\draw (5.4, -1) -- node {\textbf{z}} (5.4, -1);
\draw (4.3, 4.4) -- node {\textbf{w}} (4.3, 4.4);

\draw (3,-4) node[cross] {};

\foreach \x in {-6,...,6}
    \draw (\x,.1) -- node[below,yshift=-1mm] {\x} (\x,-.1);
\foreach \y in {-6,...,6}
    \draw (.1,\y) -- node[below,xshift=-3mm, yshift=3mm] {\y} (-.1,\y);
\draw[->,line width=0.15mm] (0,-6.5) -- (0,6.5);
\draw[->,line width=0.15mm] (-6.5,0) -- (6.5,0);
\end{tikzpicture}
}
\caption{\textbf{Part III, Case (4)(iv).} We assume that there do not exist some special upper west paired HNEs and there exists an upper west pipe on the row $r_i$. We assume that there are no VEs removed other than the ones shown in the figure, so $a=1$ and $b= 5$. We have $r_i = r_1$ with the upper west pipe $\{ (-5, 1), (-4, 1), (-3, 1) \}$. The points $\textbf{v}$, $\textbf{t}$, $\textbf{w}$ and $\textbf{z}$ are marked on the figure and the west ascending chain is coloured green. We also have $HE_{special} = \{ (4, -2), (5, -2) \}$, $d=6$ and let us assume for this example that $c_{uw} = 6$, though this value should be larger. Then, if the robot follows $K_1 = E^5$ starting from $\textbf{v}$ it gets to $\textbf{t}$; if the robot follows $K_2 = (S^5N^5E)^2 \text{ } (S^7N^7E) \text{ } W$ starting from $\textbf{t}$ it gets to $\textbf{w}$; if the robot follows $K_3 = S^6 E N^6$ starting from $\textbf{w}$ it gets to $\textbf{z}$;  if the robot follows $K_4 = (N^4 S^4 W) (NSW) NS$ starting from $\textbf{z}$ it gets to the destination point; $K_5 = E^{15}$. We define $K=K_1 K_2 K_3 K_4 K_5$ and note that if the robot follows $K$ starting from $\textbf{v}$ it visits the destination point, but if the robot follows $K$ starting on $r_i = r_1$ strictly at the west of $\textbf{v}$, it returns on $r_i$ strictly at the west of $\textbf{v}$.} \label{12}
\end{figure}
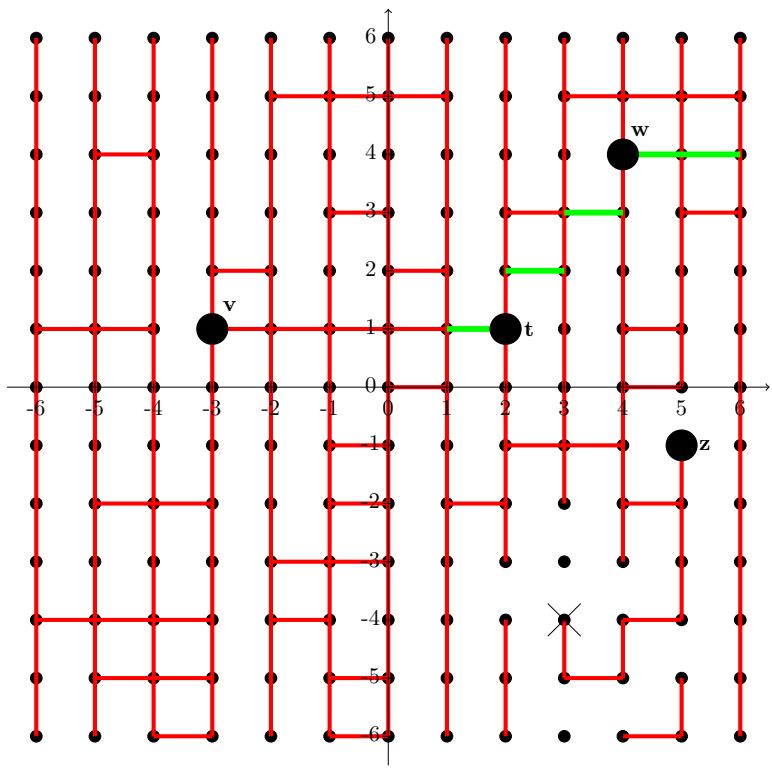

\noindent
\textbf{4(iv) } We assume that there do not exist some special upper west paired HNEs and there exists an upper west pipe on the row $r_i$ which contains the upper west pass. The upper west pipe is the west pipe (the easternmost configuration of a HE followed by a HNE) on the row $r_i$. Let the point $\textbf{v}$ in the west strip be the easternmost vertex of the HNE of the upper west pipe. Let the point $\textbf{t}$ be the easternmost vertex of the upper west pass. Consider the finite sequence of HEs in the west ascending chain $HE_a, HE_{a+1}, \hdots , HE_b$. Let the point $\textbf{w}$ in the upper infinite column $R_n$ in $c_n$ be the westernmost vertex of $HE_n$. Let $HE_{special}$ be a HE between the upper infinite column $R_n$ and the finite column $R_{n+1}$. As in case \textbf{4(iii)}, if $R_{n+1}$ does not exist, let $HE_{special}$ be the lowermost HE between $R_n$ and $R$. Let the constant $d$ be the difference in latitude between $HE_n$ and $HE_{special}$, with $d \geq 0$ from the definition of the upper west pass which is above all passes and special vertices in the obstacle strip. Let $\textbf{z}$ be the uppermost point in the finite column $R_{n+1}$ if $R_{n+1}$ exists. In the following argument, we assume that $R_{n+1}$ exists and it will be clear how this also naturally treats the case when $R_n$ is connected to $R$ which contains the destination point.

In what follows, we will construct $5$ algorithms $K_1, \hdots, K_5$, by inspecting the tertiary rectangle.

We start by constructing the algorithm $K_1=(WS^{c_{uw}}EN^{c_{uw}})^{h_1} E^{h_2}$, where $h_1$ and $h_2$ are positive integers, such that if the robot starts from the point $\textbf{v}$ and follows $K_1$ it gets to the point $\textbf{t}$. We make use of the fact that in the west strip at each given longitude at least one of the rows $r_i$ and $r_j$, $j=i-c_{uw}$ contains a HE. We also make use of the fact that in the west strip the section of the row $r_i$ at the east of the upper west pipe is the complement of a path, followed by a path (which is nonempty from the existence of the upper west pass). However, we remark that if the robot starts on $r_i$ strictly at the west of $\textbf{v}$ and it follows $K_1$, it always remains strictly at the west of $\textbf{v}$, due to the HNE of the west pipe and the fact that there are no VNEs at the west of $\textbf{v}$.

We construct the algorithm $K_2=(\prod_{m=a+1}^{n} S^{k_m}N^{k_m}E)W $, where $k_m$ is a positive integer for all $a+1 \le m \le n$, such that if the robot starts from the point $\textbf{t}$ and follows $K_2$ it gets to the point $\textbf{w}$. More specifically, if the robot is in the upper infinite column $R_m$ in the column $c_m$ at the easternmost point of $HE_{m-1}$ and follows $S^{k_m}N^{k_m}E$, it gets in the upper infinite column $R_{m+1}$ in the column $c_{m+1}$ at the easternmost end of the $HE_{m}$, for $a+1 \le m \le n$. After the robot follows the last instruction in the product, $S^{k_m}N^{k_m}E$, it gets to $c_{n+1}$ at the easternmost point of $HE_n$ and so after it follows the last instruction in $K_2$, that is $W$, the robot gets to the point $\textbf{w}$.

We define the algorithm $K_3=S^d E N^d$, such that if the robot starts from $\textbf{w}$ and it follows $K_3$, it gets to the point $\textbf{z}$. However, we remark that if the robot starts on $r_i$ strictly at the west of $\textbf{v}$ and it follows $K_2 \text{ } K_3$, it always remains strictly at the west of $\textbf{v}$. Indeed, while the robot follows $K_2$ starting strictly at the west of $\textbf{v}$, the HNE of the west pipe prevents it from visiting longitudes greater than that of $\textbf{v}$. Hence, the robot could only potentially get to a large longitude by reaching $\textbf{v}$ after it follows $K_3$; however, this is impossible as the last instruction in $K_2$ is $W$.

We construct the algorithm $K_4=(\prod_{m=n+1}^{k} N^{k_m}N^{-k_m}E^{\epsilon_m})N^{k_{k+1}}N^{-k_{k+1}}$, where $\epsilon_m \in \{-1, 1\}$ and $k_m$ is an integer for all $n+1 \le m \le k+1$, such that if the robot starts from the point $\textbf{z}$ and follows $K_4$ it passes through the destination point. More specifically if the robot is at some specified latitude in the finite column $R_i$ and it follows $N^{k_i}N^{-k_i}E^{\epsilon_i}$, it gets to some specified latitude in the finite column $R_{i+1}$ for $n+1 \le i \le k$, where by convention we write $R_{k+1}$ for $R$. If the robot is at some specified latitude in $R$ and it follows $N^{k_{k+1}}N^{-k_{k+1}}$ it passes through the destination point. However, we remark that if the robot starts on $r_i$ strictly at the west of $\textbf{v}$ and it follows $K_4$ it always remains strictly at the west of $\textbf{v}$, as the robot follows the $E$ instructions at latitude $i$ and the HNE of the west pipe prevents it from visiting longitudes greater than that of $\textbf{v}$.

We finally construct the algorithm $K_5=E^{|K_4|+1}$ and note that if the robot starts on $r_i$ strictly at the west of $\textbf{v}$ and it follows $K_5$ it always remains strictly at the west of $\textbf{v}$.

We define the algorithm $K =K_1 K_2 K_3 K_4 K_5.$ Note that if the robot starts on the row $r_i$ strictly at the west of the point $\textbf{v}$ and it follows $K$ then it returns on the row $r_i$ strictly at the west of $\textbf{v}$. Indeed, the last part follows by the remarks we made on $K_1, \hdots, K_5$ individually and the first part follows from the fact that the robot does not meet any VNEs if it starts on the row $r_i$ strictly at the west of $\textbf{v}$ and it follows $K$. If the robot starts at the point $\textbf{v}$ and it follows $K$, then it visits the destination point; this follows directly from the definitions of $K_1, \hdots, K_5$. Finally, we claim that if the robot starts on the row $r_i$ strictly at the west of the point $\textbf{v}$ and it follows $K$, its longitude does not decrease. Indeed, the only $W$ instructions in $K$ occur either in $K_4$, which is followed by $K_5$ specifically designed to negate them or as the last instruction in $K_2$, which is preceded by an $E$ instruction. Therefore, the claim holds (see Figure~\ref{12}).

Finally, we define the algorithm 
\[
F = N^i \text{ } MW (c_{uw}, q) \text{ } SME (\mu+\lambda+2q,\mu+\lambda+2q,K),
\]
where $\mu = |MW (c_{uw}, q)|.$
\begin{prop}
For any maze in $F_i$, after the robot follows $X \text{ } F$, it visits the destination point.
\end{prop}
\begin{proof}
We may assume without loss of generality that after the robot follows $X N^i$ it gets in the west strip on the row $r_i$. While the robot follows the algorithm $MW (c_{uw}, q)$ it gets on the row $r_i$ at $\textbf{v}$ or to the west of $\textbf{v}$. By the choice of parameters and by Lemma~\ref{l2}, if the robot is in the west strip on the row $r_i$ at the west of the point $\textbf{v}$ and it follows $SME (\mu+\lambda+2q,\mu+\lambda+2q,K)$, it advances eastwards oscillating about the row $r_i$. While the robot is on the row $r_i$ strictly at the west of $\textbf{v}$ and it follows $K$, it remains on the row $r_i$ strictly at the west of $\textbf{v}$. After defining $K$, we checked that it satisfies the conditions required in order to apply Lemma~\ref{l2}. Finally, the robot reaches the point $\textbf{v}$ not while executing $K$, but while executing a locomotory move in $SME$. Immediately afterwards, the robot follows $K$ and it gets to the destination point. The conclusion follows.
\end{proof}
\begin{figure}
\centering
\resizebox{0.5\textwidth}{!}{
\begin{tikzpicture}
    [
        dot/.style={circle,draw=black, fill,inner sep=2pt},
        cross/.style={cross out, draw=black, minimum size=2*(#1-\pgflinewidth), inner sep=4pt, outer sep=4pt},
    ]

\foreach \x in {-4,...,4}{
 \foreach \y in {-4,...,4}{
    \node[dot] at (\x,\y){ };
}}

\draw[line width=0.7mm, red] (-4, -4) -- (-4, 4);
\draw[line width=0.7mm, red] (-3, -4) -- (-3, 4);
\draw[line width=0.7mm, red] (-2, -4) -- (-2, 4);
\draw[line width=0.7mm, red] (-1, -4) -- (-1, -3);
\draw[line width=0.7mm, red] (-1, -2) -- (-1, 4);
\draw[line width=0.7mm, red] (0, -4) -- (0, -2);
\draw[line width=0.7mm, red] (0, -1) -- (0, 4);

\draw[line width=0.7mm, red] (1, -2) -- (1, 4);

\draw[line width=0.7mm, red] (2, -3) -- (2, -1);
\draw[line width=0.7mm, red] (2, 0) -- (2, 4);

\draw[line width=0.7mm, red] (3, -4) -- (3, 4);
\draw[line width=0.7mm, red] (4, -4) -- (4, 4);

\draw[line width=0.7mm, red] (0, 0) -- (2, 0);
\draw[line width=0.7mm, red] (1, 1) -- (3, 1);
\draw[line width=0.7mm, red] (-1, -1) -- (0, -1);
\draw[line width=0.7mm, green] (-2, 2) -- (0, 2);
\draw[line width=0.7mm, green] (0, 3) -- (3, 3);
\draw[line width=0.7mm, red] (-1, 1) -- (0, 1);
\draw[line width=0.7mm, red] (1, -2) -- (2, -2);
\draw[line width=0.7mm, red] (3, -1) -- (4, -1);
\draw[line width=0.7mm, red] (-4, -3) -- (-2, -3);
\draw[line width=0.7mm, red] (-4, -4) -- (-1, -4);

\draw node[fill,circle,inner sep=5.5pt,minimum size=2pt] at (-2, 2) {};
\draw node[fill,circle,inner sep=5.5pt,minimum size=2pt] at (-4, 0) {};
\draw node[fill,circle,inner sep=5.5pt,minimum size=2pt] at (-1, -3) {};

\draw (-1.7,2.4) -- node {\textbf{v}} (-1.7,2.4);
\draw (-1, -2.6) -- node {\textbf{z}} (-1, -2.6);

\draw (2,-3) node[cross] {};

\foreach \x in {-4,...,4}
    \draw (\x,.1) -- node[below,yshift=-1mm] {\x} (\x,-.1);
\foreach \y in {-4,...,4}
    \draw (.1,\y) -- node[below,xshift=-3mm, yshift=3mm] {\y} (-.1,\y);
\draw[->,line width=0.15mm] (0,-4.5) -- (0,4.5);
\draw[->,line width=0.15mm] (-4.5,0) -- (4.5,0);
\end{tikzpicture}
}
\caption{\textbf{Part III, Case (4)(v).} There does not exist a magical west row, there does not exist a special almost empty west row, there does not exist an upper west pipe, there do not exist the special upper west paired HNEs, but there exists an upper west cutoff. We assume that there are no VEs removed other than the ones shown in the figure. Then the row $r_i = r_2$ is the complement of a path in the west strip and all the rows $r_k$ with $k\le j=i - c_{uw} \le -p$ are paths in the west strip and indeed paths in the entire maze. For the purpose of this example, we can take $c_{uw}$ to be any large constant, say $c_{uw}=100$. The points $\textbf{v}$ and $\textbf{z}$ are marked on the figure, $z= -3$, $j=i-c_{uw}=-98$ and $\textbf{w} = (-2, -103)$. Let us see what is the path of the robot as it follows $F = N^2 (S^{100} E N^{200} S^{100} W)^4 N^{105} K$ starting from $(-4, 0)$, where $K$ is any algorithm that takes the robot from $\textbf{v}$ to the destination point. When the robot follows $S^{100} E N^{200} S^{100} W$ starting from $(-4, 2)$, it first reaches a row which is a path after it executes $S^{100}$, so its longitude increases by $1$ after it executes $S^{100} E$; so after the robot executes $S^{100} E N^{200} S^{100}$ it is back on $r_2=r_i$ with its latitude increased by one, at $(-3, 2)$; the $W$ instruction at the end does not change the longitude of the robot, as $r_2$ is the complement of a path in the west strip. Similarly, after the robot follows $S^{100} E N^{200} S^{100} W$ starting from $(-3, 2)$ it gets to $\textbf{v} = (-2, 2)$. After the robot follows $S^{100} E N^{200} S^{100} W$ starting from $\textbf{v}$, it enters the lower infinite column on $c_{a+1}$: after $S^{100} E$ it is at $(a+1, j) = (-1, -98)$; after the robot follows $S^{100} E N^{200} S^{100}$, it is at $(-1, -103)$; finally, after the robot follows $S^{100} E N^{200} S^{100} W$, it is at $\textbf{w} = (-2, -103)$. Similarly, we can see that after the robot follows each subsequent instance of $S^{100} E N^{200} S^{100} W$ starting at $\textbf{w}$, it returns to $\textbf{w}$. After the robot follows enough instances of $S^{100} E N^{200} S^{100} W$ to reach $\textbf{w}$, it follows $N^{i+c_{uw}-z} = N^{105}$ and it reaches $\textbf{v}$; immediately afterwards, the robot follows $K$ and it reaches the destination point.} \label{13}
\end{figure}
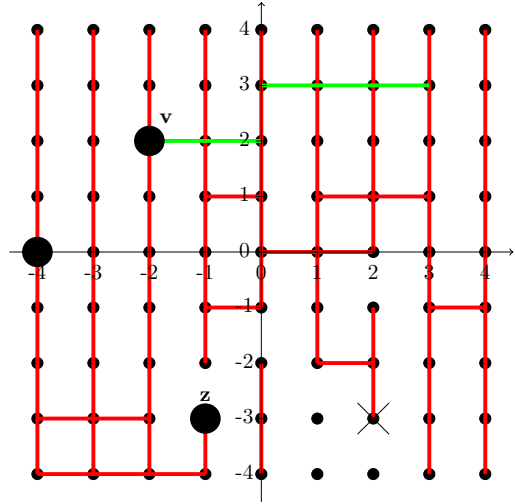

\noindent
\textbf{4(v) } We assume that there does not exist a magical west row, there does not exist a special almost empty west row, there does not exist an upper west pipe, there do not exist the special upper west paired HNEs, but there exists an upper west cutoff. We recall that the upper west cutoff is the easternmost HNE in the west strip on the row $r_i$ which contains the upper west pass. Then the row $r_i$ is the complement of a path in the west strip and all the rows $r_k$ with $k\le j=i - c_{uw} \le -p$ are paths in the west strip and indeed paths in the entire maze (from the non existence of the special upper west paired HNEs). Let $\textbf{v} = (a, i)$ be the easternmost vertex of the row $r_i$ in the west strip. Let $\textbf{z} = (a+1, z)$ be the uppermost vertex of the westernmost lower infinite column in the column $c_{a+1}$. Let $\textbf{w} = (a, z-c_{uw})$. By inspecting the tertiary rectangle, we can construct an algorithm $K$ that takes the robot from $\textbf{v}$ to the destination point.

We define the algorithm
\[
F := N^i \text{ } (S^{i-j} E N^{2i-2j} S^{i-j} W)^{\lambda+q} \text{ } N^{i+c_{uw}-z} \text{ } K.
\]
\textbf{Claim. } For any maze in $F_i$, after the robot follows $X \text{ } F$, it visits the destination point.
\begin{proof}
We may assume without loss of generality that after the robot follows $X \text{ } N^i$ it gets in the west strip on the row $r_i$. By the choice of exponents, while the robot follows $(S^{i-j} E N^{2i-2j} S^{i-j} W)^{\lambda+q}$ it gets to the point $\textbf{w}$ and remains stuck there. Indeed, while the robot follows each instance of $S^{i-j} E N^{2i-2j} S^{i-j} W$, it advances one unit to the east, oscillating about the row $r_i$ until it gets to $\textbf{v}$. Immediately afterwards, it follows $S^{i-j} E N^{2i-2j} S^{i-j} W$ and gets to $\textbf{w}$. After the robot gets to $\textbf{w}$, after each other instance of $S^{i-j} E N^{2i-2j} S^{i-j} W$, the robot gets back to $\textbf{w}$. If the robot starts at $\textbf{w}$ and it follows $N^{i+c_{uw}-z}$, it gets to $\textbf{v}$. Therefore, after the robot follows $X \text{ } F$, it gets to the destination point. The conclusion follows.
\end{proof}
\begin{figure}
\centering
\resizebox{0.5\textwidth}{!}{
\begin{tikzpicture}
    [
        dot/.style={circle,draw=black, fill,inner sep=2pt},
        cross/.style={cross out, draw=black, minimum size=2*(#1-\pgflinewidth), inner sep=4pt, outer sep=4pt},
    ]

\foreach \x in {-4,...,4}{
 \foreach \y in {-4,...,4}{
    \node[dot] at (\x,\y){ };
}}

\draw[line width=0.7mm, red] (-4, -4) -- (-4, 4);
\draw[line width=0.7mm, red] (-3, -4) -- (-3, 4);
\draw[line width=0.7mm, red] (-2, -4) -- (-2, 4);
\draw[line width=0.7mm, red] (-1, -4) -- (-1, 4);

\draw[line width=0.7mm, red] (0, -2) -- (0, 4);

\draw[line width=0.7mm, red] (1, -1) -- (1, 4);
\draw[line width=0.7mm, red] (1, -2) -- (1, -3);

\draw[line width=0.7mm, red] (2, -2) -- (2, 4);
\draw[line width=0.7mm, red] (2, -3) -- (2, -4);

\draw[line width=0.7mm, red] (3, -1) -- (3, 4);
\draw[line width=0.7mm, red] (3, -2) -- (3, -4);

\draw[line width=0.7mm, red] (4, -4) -- (4, 4);

\draw[line width=0.7mm, red] (3, 3) -- (4, 3);

\draw[line width=0.7mm, red] (1, 2) -- (3, 2);

\draw[line width=0.7mm, red] (-4, 1) -- (4, 1);

\draw[line width=0.7mm, red] (-4, 1) -- (4, 1);

\draw[line width=0.7mm, red] (-4, 0) -- (4, 0);

\draw[line width=0.7mm, red] (-3, -1) -- (-2, -1);
\draw[line width=0.7mm, red] (0, -1) -- (3, -1);

\draw[line width=0.7mm, red] (-4, -2) -- (-3, -2);
\draw[line width=0.7mm, red] (-2, -2) -- (-1, -2);
\draw[line width=0.7mm, red] (0, -2) -- (2, -2);

\draw[line width=0.7mm, red] (-3, -3) -- (-2, -3);
\draw[line width=0.7mm, red] (3, -3) -- (4, -3);

\draw[line width=0.7mm, red] (-3, -4) -- (-1, -4);

\draw[line width=0.7mm, green] (-1, 0) -- (0, 0);

\draw node[fill,circle,inner sep=5.5pt,minimum size=2pt] at (3, 2) {};
\draw node[fill,circle,inner sep=5.5pt,minimum size=2pt] at (1, -1) {};
\draw node[fill,circle,inner sep=5.5pt,minimum size=2pt] at (-3, 0) {};

\draw (3.3, 2.3) -- node {\textbf{v}} (3.3, 2.3);
\draw (1.3, -0.7) -- node {\textbf{w}} (1.3, -0.7);

\draw (1,-3) node[cross] {};

\foreach \x in {-4,...,4}
    \draw (\x,.1) -- node[below,yshift=-1mm] {\x} (\x,-.1);
\foreach \y in {-4,...,4}
    \draw (.1,\y) -- node[below,xshift=-3mm, yshift=3mm] {\y} (-.1,\y);
\draw[->,line width=0.15mm] (0,-4.5) -- (0,4.5);
\draw[->,line width=0.15mm] (-4.5,0) -- (4.5,0);
\end{tikzpicture}
}
\caption{\textbf{Part III, Case (4)(vi).} We assume there exists an upper west HNE on some row $r_j$. We further assume there does not exist a magical west row, there does not exist a magical east row and there does not exist an upper west cutoff and that there are no VEs removed other than the ones shown in the figure. Then all the rows $r_m$ with $i\le m < j$ are paths in the maze. In this example, the upper west pass is coloured green, the uppermost westernmost VNE is $\{ (1, -2), (1, -1) \}$, the upper west HNE is $\{(3, 2), (4, 2) \}$ and so $j=2$. The vertices $\textbf{v}$ and $\textbf{w}$ are marked on the figure. We can take $K = S^3 (WS)^2$, so if the robot follows $K$ starting from $\textbf{v}$ it visits the destination point. Let us observe how the robot follows $F = (ES^3 N^3)^{10}K$ starting from $(-3, 0)$. As long as the robot is in the west strip, each instance of $ES^3 N^3$ increases its longitude by one. Eventually, the robot gets to $(0, 0)$. After that, the robot follows $S^3 N^3$ and it gets to $(0,1)$. Considering that every row at latitude between $i=0$ and $j=2$ is a path in the maze, every further instance of $ES^3 N^3$ increases the longitude of the robot by one, until it arrives at $\textbf{v} = (3, 2)$, as its latitude is determined by the uppermost VNEs at the west of $\textbf{v}$. Once the robot reaches $\textbf{v}$, we can see that after each instance of $ES^3 N^3$, the robot returns to $\textbf{v}$. Finally, the robot follows $K$ and it visits the destination point.} \label{14}
\end{figure}
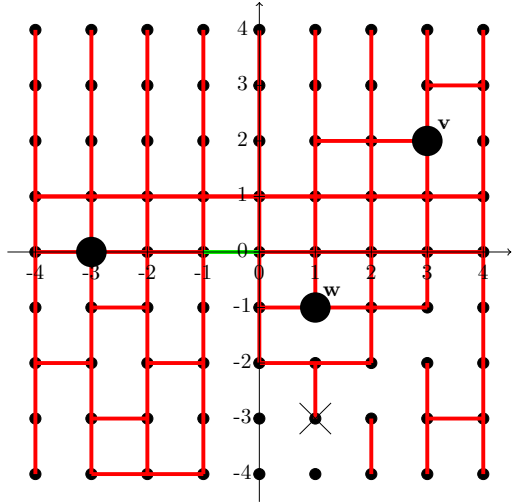

\noindent
\textbf{4(vi) } We assume there exists an upper west HNE on some row $r_j$. We recall that the upper west HNE is the lowermost westernmost HNE at the north-east of the uppermost westernmost VNE. We further assume there does not exist a magical west row, there does not exist a magical east row and there does not exist an upper west cutoff. Then all the rows $r_m$ with $i\le m < j$ are paths in the maze (from the minimality of $j$ and the non-existence of a magical east row). Let $\textbf{v}$ be the western vertex of the upper west HNE. Let $\textbf{w} = (x_w, y_w)$ be the upper vertex of the uppermost westernmost VNE. Then $\textbf{v}$ is at the east of $\textbf{w}$. By inspecting the tertiary rectangle, we construct an algorithm $K$ which takes the robot from $\textbf{v}$ to the destination point (see Figure~\ref{14}). 

We define the algorithm
\[
F= N^i (ES^{j-y_w}N^{j-y_w})^{\lambda+q}K.
\]
\textbf{Claim. } For any maze in $F_i$, after the robot follows $X \text{ } F$, it visits the destination point.
\begin{proof}
We may assume without loss of generality that after the robot follows $X N^i$ it gets in the west strip on the row $r_i$. In the west strip, while the robot follows $ES^{j-y_w}N^{j-y_w}$ it advances eastwards oscillating about the row $r_i$. In the obstacle strip, while the robot follows $ES^{j-y_w}N^{j-y_w}$ it advances eastwards, potentially increasing its latitude as it meets VNEs. It eventually gets on the row $r_j$ and remains stuck at the point $\textbf{v}$. Therefore, after the robot follows $X \text{ } F$ it gets to the destination point. The conclusion follows (see Figure~\ref{14}). 
\end{proof}
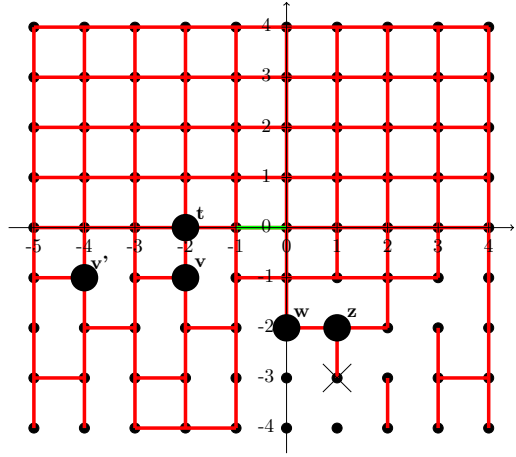
\begin{figure}
\centering
\resizebox{0.5\textwidth}{!}{
\begin{tikzpicture}
    [
        dot/.style={circle,draw=black, fill,inner sep=2pt},
        cross/.style={cross out, draw=black, minimum size=2*(#1-\pgflinewidth), inner sep=4pt, outer sep=4pt},
    ]

\foreach \x in {-5,...,4}{
 \foreach \y in {-4,...,4}{
    \node[dot] at (\x,\y){ };
}}

\draw[line width=0.7mm, red] (-5, -4) -- (-5, 4);
\draw[line width=0.7mm, red] (-4, -4) -- (-4, 4);
\draw[line width=0.7mm, red] (-3, -4) -- (-3, 4);
\draw[line width=0.7mm, red] (-2, -4) -- (-2, 4);
\draw[line width=0.7mm, red] (-1, -4) -- (-1, 4);

\draw[line width=0.7mm, red] (0, -2) -- (0, 4);

\draw[line width=0.7mm, red] (1, -1) -- (1, 4);
\draw[line width=0.7mm, red] (1, -2) -- (1, -3);

\draw[line width=0.7mm, red] (2, -2) -- (2, 4);
\draw[line width=0.7mm, red] (2, -3) -- (2, -4);

\draw[line width=0.7mm, red] (3, -1) -- (3, 4);
\draw[line width=0.7mm, red] (3, -2) -- (3, -4);

\draw[line width=0.7mm, red] (4, -4) -- (4, 4);

\draw[line width=0.7mm, red] (-5, 4) -- (4, 4);

\draw[line width=0.7mm, red] (-5, 3) -- (4, 3);

\draw[line width=0.7mm, red] (-5, 2) -- (4, 2);

\draw[line width=0.7mm, red] (-5, 1) -- (4, 1);

\draw[line width=0.7mm, red] (-5, 0) -- (4, 0);

\draw[line width=0.7mm, red] (-5, -1) -- (-4, -1);

\draw[line width=0.7mm, red] (-3, -1) -- (-2, -1);
\draw[line width=0.7mm, red] (-1, -1) -- (3, -1);

\draw[line width=0.7mm, red] (-4, -2) -- (-3, -2);
\draw[line width=0.7mm, red] (-2, -2) -- (-1, -2);
\draw[line width=0.7mm, red] (0, -2) -- (2, -2);

\draw[line width=0.7mm, red] (-5, -3) -- (-4, -3);
\draw[line width=0.7mm, red] (-3, -3) -- (-2, -3);
\draw[line width=0.7mm, red] (3, -3) -- (4, -3);

\draw[line width=0.7mm, red] (-3, -4) -- (-1, -4);

\draw[line width=0.7mm, green] (-1, 0) -- (0, 0);

\draw node[fill,circle,inner sep=5.5pt,minimum size=2pt] at (-2, -1) {};
\draw node[fill,circle,inner sep=5.5pt,minimum size=2pt] at (0, -2) {};
\draw node[fill,circle,inner sep=5.5pt,minimum size=2pt] at (-2, 0) {};
\draw node[fill,circle,inner sep=5.5pt,minimum size=2pt] at (1, -2) {};
\draw node[fill,circle,inner sep=5.5pt,minimum size=2pt] at (-4, -1) {};

\draw (-1.7, -0.7) -- node {\textbf{v}} (-1.7, -0.7);
\draw (-3.7, -0.7) -- node {\textbf{v'}} (-3.7, -0.7);
\draw (0.3, -1.7) -- node {\textbf{w}} (0.3, -1.7);
\draw (-1.7, 0.3) -- node {\textbf{t}} (-1.7, 0.3);
\draw (1.3, -1.7) -- node {\textbf{z}} (1.3, -1.7);

\draw (1,-3) node[cross] {};

\foreach \x in {-5,...,4}
    \draw (\x,.1) -- node[below,yshift=-1mm] {\x} (\x,-.1);
\foreach \y in {-4,...,4}
    \draw (.1,\y) -- node[below,xshift=-4mm, yshift=3mm] {\y} (-.1,\y);
\draw[->,line width=0.15mm] (0,-4.5) -- (0,4.5);
\draw[->,line width=0.15mm] (-5.5,0) -- (4.5,0);
\end{tikzpicture}
}
\caption{\textbf{Part III, Case (4)(vii).} We assume there does not exist a magical west row, there does not exist a magical east row, there does not exist an upper west cutoff, there does not exist an upper west HNE, but there does exist a special west pipe on some row $r_j$. We assume that there are no VEs removed other than the ones shown in the figure. The upper west pass is coloured green and it is on the row $r_i=r_0$. From the assumptions, it follows that for every $m \geq i$, the row $r_m$ is a path in the maze. The special west pipe is $\{ (-3, -1), (-2, -1), (-1, -1) \}$ on $r_j = r_{-1}$. We take $R_{n+1}$ to be $\{(1, -2), (1, -3)\}$, accessible from $R_n = \{(0, -2), (0, -1), \hdots \}$ via $HE_{special} = \{(0, -2), (1, -2)\}$ on $r_{\gamma} = r_{-2}$. Then, if the robot follows $K_1 = NE^2S^3N^3W^2S$ starting from $\textbf{v}$, it gets to $\textbf{t}$ passing from $\textbf{w}$; however, note that if the robot follows $K_1$ starting from $\textbf{v'}$ (which is the eastern vertex of the HE of the ``fake west pipe'' $\{(-5, -1), (-4, -1), (-3, -1) \}$ on $r_j$ strictly at the west of $\textbf{v}$), it returns to $\textbf{v'}$. If the robot follows $K_2 = E^3 W S^2 E N^2$ starting from $\textbf{t}$, it gets to $\textbf{z}$; however, if the robot follows $K_2$ starting from $\textbf{v'}$ it gets back to $\textbf{v'}$; in general, we are certain that if the robot follows $K_2$ starting from $\textbf{v'}$ it either gets back to $\textbf{v'}$ or to the western neighbour of $\textbf{v'}$. If the robot follows $K_3 = NSW$ starting from $\textbf{z}$ it visits the destination point. In this case, $K_4 = E^4$. Therefore, if the robot follows $K_3  K_4$ starting either from $\textbf{v'}$ or from the western neighbour of $\textbf{v'}$, it gets to $\textbf{v'}$.} \label{15}
\end{figure}

\noindent
\textbf{4(vii) } We assume there does not exist a magical west row, there does not exist a magical east row, there does not exist an upper west cutoff, there does not exist an upper west HNE, but there does exist a special west pipe on some row $r_j$. We recall that the special west pipe is the west pipe (the easternmost configuration in the west strip of a HE followed by a HNE) on the smallest row that has a west pipe with respect to the standard well order on $\mathbb{Z}$. Then all the rows $r_m$ with $m \geq i$ are paths in the maze (from the non existence of an upper west HNE and the non existence of a magical east row). 

Let $\textbf{v} = (x_v,j)$ be the eastern vertex of the HE of the special west pipe. Let $\textbf{w} = (a+1,y_w)$ be the lowermost vertex of the westernmost upper infinite column $R_{a+1}$. Let $\textbf{t} = (x_v,i)$ be the vertex at the intersection between the column $c_{x_v}$ and the row $r_i$. Let $\textbf{z} = (n+1, y_z)$ be the uppermost vertex of the finite column $R_{n+1}$ or the uppermost vertex of the lower infinite column $R_{n+1} = R$ that contains the destination point. The special case that the destination point is in the upper infinite column $R_{n+1} = R$ is much more easy and we will make a note on how to solve it before defining the finish algorithm $F$. Let $HE_{special}$ be a HE on some row $r_{\gamma}$ between the upper infinite column $R_n$ and the finite column $R_{n+1}$. Let $\textbf{v'}$ be the eastern vertex of the HE of any ``fake west pipe'', i.e. a configuration in the west strip on $r_j$ that is formed by a HE followed by a HNE, strictly at the west of the special west pipe (see Figure~\ref{15}).

We define the algorithm $K_1=N^{i-j}E^{a+1-x_v}S^{2i-j-y_w}N^{2i-j-y_w}W^{a+1-x_v}S^{i-j}$ with the property that if the robot starts from $\textbf{v}$ and follows $K_1$ it passes through the point $\textbf{w}$ and gets to the point $\textbf{t}$. However, if the robot starts at $\textbf{v'}$ and it follows $K_1$ then it returns at $\textbf{v'}$. The second statement follows from the fact that the robot moves at every instruction in $K_1$: indeed, while the robot executes $N^{i-j}$ starting from $\textbf{v'}$, it is in the west strip which contains no VNEs, so it changes its latitude to $i$; considering that $r_i$ is a path in the maze, when the robot continues to follow $E^{a+1-x_v}$, its longitude increases by exactly $a+1-x_v$ which is the exact difference in longitude between $\textbf{v}$ and the westernmost column in the obstacle strip, $c_{a+1}$; as $\textbf{v'}$ is strictly at the west of $\textbf{v}$, we conclude that after the robot follows $N^{i-j}E^{a+1-x_v}$ starting from $\textbf{v'}$, it is still in the west strip on the row $r_i$ which is a path in the maze; hence, if the robot follows $K_1$ starting from $\textbf{v'}$, it gets back to $\textbf{v'}$. Similarly, we can show the first statement about $K_1$, that if the robot starts from $\textbf{v}$ and follows $K_1$ it gets to the point $\textbf{t}$; in this case, we note that the only instructions in $K_1$ that do not change the position of the robot are instructions of type $S$ from the group $S^{2i-j-y_w}$ that occur immediately after the robot reaches $\textbf{w}$ (see Figure~\ref{15}).

We define the algorithm $K_2=E^{n+1-x_v}WS^{i-\gamma}EN^{i-\gamma}$ such that if the robot starts from $\textbf{t}$ and follows $K_2$ it gets to the point $\textbf{z}$. This is clear as the robot starts on $r_i$ which is a path, so after it follows $E^{n+1-x_v}W$ it gets at the point $(n, i)$ and so after it follows $K_2$ it is in $R_{n+1}$; moreover, as the upper west pass at latitude $i$ is above all the passes in the obstacle strip and so, in this case, also above all the VNEs, the robot actually gets to $\textbf{z}$ in $R_{n+1}$ after it follows $K_2$ starting from $\textbf{t}$. However, if the robot follows $K_2$ starting from $\textbf{v'}$, it does not move after it follows $E^{n+1-x_v}$ and its longitude decreases by $1$ after it follows $E^{n+1-x_v}W$. Hence, if the robot follows $K_2$ starting from $\textbf{v'}$, it either gets back to $\textbf{v'}$ or it gets to the western neighbour of $\textbf{v'}$ (see Figure~\ref{15}).

By inspecting the tertiary rectangle, we construct the algorithm $K_3$ of the form $K_3=(\prod_{m=n+1}^{k} N^{k_m}N^{-k_m}E^{\epsilon_m})N^{k_{k+1}}N^{-k_{k+1}}$, where $\epsilon_m \in \{-1, 1\}$ and $k_m$ is an integer for all $n+1 \le m \le k+1$, such that if the robot starts from the point $\textbf{z}$ and follows $K_3$ it passes through the destination point. More specifically, if the robot is at some specified latitude in the finite column $R_m$ and follows $N^{k_m}N^{-k_m}E^{\epsilon_m}$ it gets to some specified latitude in the finite column $R_{m+1}$ for $n+1 \le m \le k$, where by convention we write $R_{k+1}$ for $R$. If the robot is at some specified latitude inside $R$ and it follows $N^{k_{k+1}}N^{-k_{k+1}}$, it visits the destination point.

We construct the algorithm $K_4=E^{|K_3|+1}$. We note that from the structure of a fake west pipe and its position in the west strip, if the robot starts either at $\textbf{v'}$ or at the western neighbour of $\textbf{v'}$ and it follows $K_3 K_4$, it gets to $\textbf{v'}$.

We define the algorithm $K=K_1 K_2 K_3 K_4$ with the property that if the robot starts at $\textbf{v}$ and it follows $K$, it passes through the destination point. However, if the robot starts at $\textbf{v'}$ and it follows $K$, it gets back to $\textbf{v'}$. In the special case when $\textbf{z}$ does not exist and so the destination point $(n+1, \delta)$ is in the upper infinite column $R_{n+1}=R$ we define $K_2' = E^{n+1 - x_v} N^{\delta - i} S^{\delta - i}$. In this case we define $K = K_1 K_2'$ instead and we note that, as before, if the robot starts at $\textbf{v}$ and it follows $K$, it passes through the destination point; moreover, if the robot starts at $\textbf{v'}$ and it follows $K$, it gets back to $\textbf{v'}$.

We recall the algorithm $WPF(a,e):=(E^eWS^aEN^a)^e$, defined in the case \textbf{2(ii)}. Finally, we define the algorithm
\[
F= N^i \text{ } W^{\lambda-x_v} \text{ } S^{i-j} \text{ } (WPF(j-i, 2\lambda+q)KN^{i-j}ES^{i-j})^{2 \lambda+q}.
\]
\textbf{Claim. } For any maze in $F_i$, after the robot follows $X \text{ } F$, it visits the destination point.
\begin{proof}
We may assume without loss of generality that after the robot follows $X \text{ } N^i \text{ } W^{\lambda - x_v} \text{ } S^{i-j}$ it gets in the west strip on the row $r_j$ at the west of the point $\textbf{v}$. While the robot follows each instance of $WPF(j-i, 2\lambda+q)$ it advances eastwards to the easternmost vertex $\textbf{v'}$ of a HE of a fake west pipe on the row $r_j$. If $\textbf{v'}$ is strictly at the west of $\textbf{v}$, after the robot follows the algorithm $K$ it returns to the point $\textbf{v'}$; after the robot follows the algorithm $N^{i-j}ES^{i-j}$ starting from $\textbf{v'}$, it advances to the east of $\textbf{v'}$ on the row $r_j$. By the choice of parameters, the robot eventually gets to the point $\textbf{v'}=\textbf{v}$. Immediately afterwards, it follows $K$ and it gets to the destination point. The conclusion follows.
\end{proof}
\begin{figure}
\centering
\resizebox{0.5\textwidth}{!}{
\begin{tikzpicture}
    [
        dot/.style={circle,draw=black, fill,inner sep=2pt},
        cross/.style={cross out, draw=black, minimum size=2*(#1-\pgflinewidth), inner sep=4pt, outer sep=4pt},
    ]

\foreach \x in {-5,...,4}{
 \foreach \y in {-4,...,4}{
    \node[dot] at (\x,\y){ };
}}

\draw[line width=0.7mm, red] (-5, -4) -- (-5, 4);
\draw[line width=0.7mm, red] (-4, -4) -- (-4, 4);
\draw[line width=0.7mm, red] (-3, -4) -- (-3, 4);
\draw[line width=0.7mm, red] (-2, -4) -- (-2, 4);
\draw[line width=0.7mm, red] (-1, -4) -- (-1, 4);

\draw[line width=0.7mm, red] (0, -2) -- (0, 4);

\draw[line width=0.7mm, red] (1, -1) -- (1, 4);
\draw[line width=0.7mm, red] (1, -2) -- (1, -3);

\draw[line width=0.7mm, red] (2, -2) -- (2, 4);
\draw[line width=0.7mm, red] (2, -3) -- (2, -4);

\draw[line width=0.7mm, red] (3, -1) -- (3, 4);
\draw[line width=0.7mm, red] (3, -2) -- (3, -4);

\draw[line width=0.7mm, red] (4, -4) -- (4, 4);

\draw[line width=0.7mm, red] (-5, 4) -- (4, 4);

\draw[line width=0.7mm, red] (-5, 3) -- (4, 3);

\draw[line width=0.7mm, red] (-5, 2) -- (4, 2);

\draw[line width=0.7mm, red] (-5, 1) -- (4, 1);

\draw[line width=0.7mm, red] (-5, 0) -- (4, 0);

\draw[line width=0.7mm, red] (-5, -1) -- (4, -1);

\draw[line width=0.7mm, red] (0, -2) -- (2, -2);

\draw[line width=0.7mm, red] (3, -3) -- (4, -3);

\draw[line width=0.7mm, green] (-1, 0) -- (0, 0);

\draw node[fill,circle,inner sep=5.5pt,minimum size=2pt] at (-1, -2) {};
\draw node[fill,circle,inner sep=5.5pt,minimum size=2pt] at (0, -2) {};
\draw node[fill,circle,inner sep=5.5pt,minimum size=2pt] at (-1, 2) {};
\draw node[fill,circle,inner sep=5.5pt,minimum size=2pt] at (-3, 0) {};

\draw (-0.7, -1.7) -- node {\textbf{v}} (-0.7, -1.7);
\draw (0.3, -1.7) -- node {\textbf{z}} (0.3, -1.7);
\draw (-0.7, 2.3) -- node {\textbf{w}} (-0.7, 2.3);

\draw (1,-3) node[cross] {};

\foreach \x in {-5,...,4}
    \draw (\x,.1) -- node[below,yshift=-1mm] {\x} (\x,-.1);
\foreach \y in {-4,...,4}
    \draw (.1,\y) -- node[below,xshift=-4mm, yshift=3mm] {\y} (-.1,\y);
\draw[->,line width=0.15mm] (0,-4.5) -- (0,4.5);
\draw[->,line width=0.15mm] (-5.5,0) -- (4.5,0);
\end{tikzpicture}
}
\caption{\textbf{Part III, Case (4)(viii).} We assume that there does not exist a magical west row, there does not exist a magical east row, there does not exist an upper west cutoff, there does not exist an upper west HNE, but there exists a natural special empty west row on $r_{-2} = r_j$. We assume that there are no VEs removed other than the ones shown in the figure. Let us assume that the robot starts at $(-3, 0)$ and it follows $F = N^{-2} (N^2 E S^6 N^4 W)^{10} S^4 K$, where $K = N(ES)^2$ is an algorithm with the property that if the robot follows it starting from $\textbf{v}$ it reaches the destination point. While the robot is on $r_j = r_{-2}$ strictly at the west of $\textbf{v}$, its longitude increases by one after each instance of $N^2 E S^6 N^4 W$. After the robot reaches $\textbf{v}$ and it follows $N^2 E S^6 N^4 W$, it gets to $\textbf{w}$. If the robot follows $N^2 E S^6 N^4 W$ starting from $\textbf{w}$ it gets back to $\textbf{w}$.} \label{16}
\end{figure}
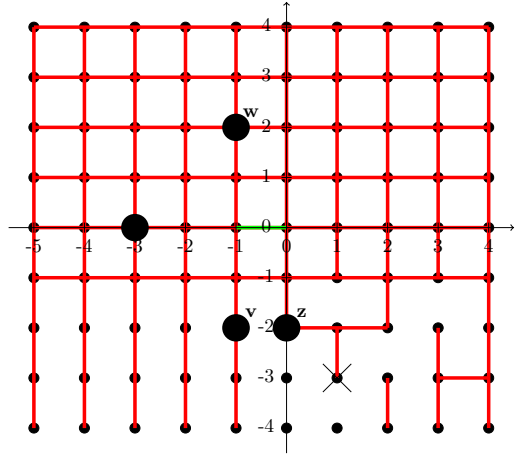

\noindent
\textbf{4(viii) } We assume that there does not exist a magical west row, there does not exist a magical east row, there does not exist an upper west cutoff, there does not exist an upper west HNE, but there exists a natural special empty west row on $r_j$. Then, as in \textbf{4(vii)}, all the rows $r_m$ for $m \geq i$ are paths in the maze. Let $\textbf{v} = (a,j)$ be the easternmost vertex of the row $r_j$ in the west strip. Let $\textbf{z} = (a+1, \gamma)$ be the lowermost vertex of the westernmost upper infinite column $R_{a+1}$. Let $\textbf{w} = (a, 2i-\gamma)$. By inspecting the tertiary rectangle, we construct an algorithm $K$ that takes the robot from $\textbf{v}$ to the destination point. 

We define the algorithm
\[
F=N^j \text{ } (N^{i-j}ES^{3i-2 \gamma-j}N^{2i-2 \gamma}W)^{\lambda+q} \text{ } S^{2i-\gamma-j} \text{ } K.
\]
\textbf{Claim. } For any maze in $F_i$, after the robot follows $X \text{ } F$, it visits the destination point.
\begin{proof}
We may assume without loss of generality that after the robot follows $X \text{ }N^j$, it gets in the west strip on the row $r_j$. While the robot follows each instance of $N^{i-j}ES^{3i-2\gamma-j}N^{2i-2\gamma}W$, it advances eastwards one unit making an oscillation about the row $r_j$. By the choice of exponent, the robot eventually gets to the point $\textbf{v}$. Immediately afterwards, it follows $N^{i-j}ES^{3i-2\gamma-j}N^{2i-2\gamma}W$ and gets to the point $\textbf{w}$. The robot remains stuck at $\textbf{w}$, i.e. while it follows each instance of $N^{i-j}ES^{3i-2\gamma-j}N^{2i-2\gamma}W$, it gets back to $\textbf{w}$ (see Figure~\ref{16}). Hence after the robot follows $X \text{ } N^i \text{ } W^{\lambda-a} \text{ } S^{i-j} \text{ } (N^{i-j}ES^{3i-2\gamma-j}N^{2i-2\gamma}W)^{\lambda+q}$ $S^{2i-\gamma-j}$, it gets to $\textbf{v}$. Hence, after the robot follows $X \text{ } F$, it gets to the destination point. The conclusion follows.
\end{proof}
\noindent
\textbf{4(ix) } As a final case, we may assume that there does not exist a magical west/east row, there does not exist a special west pipe, there does not exist a natural special empty west row, there does not exist a special almost empty west row. Then all the rows are paths in the maze and hence the maze does not contain any HNE. Therefore, both the latitude and the longitude of the robot are known and, by inspecting the primary rectangle, we can write an algorithm $F$ that takes the robot from its known position to the destination point. The conclusion follows.

This finally solves \textbf{Case (4)} in which the destination point is connected to the west strip by a (finite, possibly empty) sequence of finite columns followed by a (finite, non-empty) sequence of upper infinite columns.

We have therefore treated all possible cases, as detailed in the arguments above. This completes the proof of Theorem~\ref{ch2th2}.
\end{proof}

\section{Proof of Proposition~\ref{ch2prop1}}\label{sectiunea1}
In this short section we present a proof of the slightly technical but easy Proposition~\ref{ch2prop1}.

The following observation represents the main idea of the proof.

\begin{obs}\label{ch7obs1}
Let $o,d$ be fixed vertices in $\mathbb{Z}^2$ and let $\mathbb{B}$ be a set of subgraphs of $\mathbb{Z}^2$ which is compact in the product topology. Let $A$ be a possibly infinite algorithm that solves the set of mazes $\mathbb{A}=\{(B,o,d)\mid B\in \mathbb{B}\}$. Then there exists a finite initial segment $A_0$ of $A$ that solves $\mathbb{A}$.
\end{obs}

\begin{proof}
Assume for a contradiction that there does not exists such an initial segment $A_0$. For each $i\ge 1$, let $A_i$ be the initial segment of $A$ with the first $i$ instructions. By assumption, for each $i\ge 1$ there exists a board $B_i \in \mathbb{B}$ such that $A_i$ does not solve $B_i$. By compactness there exists a subsequence $(B_{i_j})_{j\ge 1}$ such that $\displaystyle \lim_{j\rightarrow \infty} B_{i_j}=B_0 \in \mathbb{B}$ in the product topology. As $A$ solves $B_0$, there exists an initial segment $A_0$ of $A$ which solves $(B_0,o,d)$. As $\displaystyle \lim_{j\rightarrow \infty} B_{i_j}=B_0 \in \mathbb{B}$, $A_0$ solves $(B_{i_j})_{j\ge 1}$ for all $j\ge |A_0|$ sufficiently large. This gives the desired contradiction.
\end{proof}

We are now ready to prove Proposition~\ref{ch2prop1}. 
\begin{proof}[Proof of Proposition~\ref{ch2prop1}]
By hypothesis (1) and (3) and by Observation~\ref{ch7obs1},for all $i\in \{1,2\}$, all origins $o\in \mathbb{Z}^2$, all destination $d\in \mathbb{Z}^2$ and all paths $P$ between $o$ and $d$, there exists a finite initial segment $A_{i,P}$ of $A_i$ that solves the set of mazes  $\{(M, o, d)\mid (M, o, d)\in \mathcal{A}_i, P\le M \}$ that contain the path $P$ (this set of mazes might be empty). By hypothesis (2), for all $i\in \{1,2\}$, all origins $o\in \mathbb{Z}^2$ and all $j\in \mathbb{N}$, there exists a finite initial segment $A_{i,o,j}$ of $A_i$ that guides the robot to visit all accessible points at distance at most $j$ from the origin $o$ in the set of mazes $\{(M, o, d)\mid (M, o, d)\in \mathcal{A}_i\}$ that have origin $o$ (notice that here the destination $d$ plays no role so we might as well drop it). But then for all $i\in \{1,2\}$ and all $j, k\in \mathbb{N}$, there exists a finite initial segment $A_{i,j,k}$ of $A_i$ such that for any origin $o$ at distance at most $k$ from $\mathbf{0}$ in the graph $\mathbb{Z}^2$, the algorithm guides the robot to visit all accessible points at distance at most $j$ from the origin $o$ in the set of mazes $\{(M, o)\mid (M, o)\in \mathcal{A}_i\}$ that have origin $o$.

In order to construct the algorithm $A$, we define the algorithms $B_i$ recursively to be $B_i=A_{\lfloor\frac{i}{2}\rfloor,(2|B_1\hdots B_{i-1}|+1),(2|B_1\hdots B_{i-1}|+1)}$ and take $A:=B_1B_2\hdots$. Clearly, the algorithm $A$ has the desired properties.
\end{proof}

\section{Open Problems}\label{conclusions}
As we emphasised in the proof of Theorem~\ref{mthm}, we strongly believe that there exists an algorithm which solves the set of all mazes with arbitrarily many HNEs and finitely many VNEs. The only case in our proof where an argument for this result breaks down is \textbf{Case 4} of \textbf{Part III}. We believe that this problem, together with Conjecture~\ref{plmconj} below could be solved using similar techniques with those developed in this paper. 
\begin{conj}\label{plmconj}
There exists an algorithm that solves the set of all mazes with arbitrarily many HNEs and arbitrarily many VNEs in one column.
\end{conj}

Furthermore, we believe the following positive result to hold.
\begin{conj}
Consider the subset $\mathcal{N}$ of mazes in which the connected component of the origin is a simple (possibly infinite) path. Then there exists an algorithm that solves $\mathcal{N}$.
\end{conj}
In the opposite direction, we believe the following to be true.
\begin{conj}\label{mconj}
There is no algorithm that solves the class $\mathcal{M}$ of all mazes.
\end{conj}
From another perspective, let us call $\mathcal{M}_k \subseteq \mathcal{M}$ the set of mazes for which the destination is at distance $k$ from the origin. From Proposition~\ref{ch2prop1}, the following conjecture is equivalent to Conjecture~\ref{mconj}.
\begin{conj}\label{ojlk}
There exists a $k$ for which $\mathcal{M}_k$ is not solvable.
\end{conj}
Perhaps the following stronger results also hold.
\begin{conj}\label{pbneb}
Let $\mathcal{N}_3 \subset \mathcal{M}$ be the set of all mazes for which there are only HNEs between the pairs of columns $(c_{-4}, c_{-3})$ and $(c_3, c_4)$. Then there is no algorithm that solves $\mathcal{N}_3$.
\end{conj}
\begin{conj}
Conjecture~\ref{ojlk} holds for $k=10$.
\end{conj}
Conjecture~\ref{pbneb} is one of the main reasons why we think Conjecture~\ref{mconj} holds. 

Finally, we strongly believe that the classes of mazes in higher dimensions arising from the lattice $\mathbb{Z}^k$ with suitable mild restrictions should represent a captivating further study.

\end{document}